\documentclass[12pt,letterpaper,reqno]{amsart}

\newcommand{\R}{\mathbb R}

\newcommand{\cR}{\mathbb{R}}

\newcommand{\indentalign}{\hspace{0.3in}&\hspace{-0.3in}}
\newcommand{\la}{\langle}
\newcommand{\ra}{\rangle}

\newcommand{\ds}{\displaystyle}
\newcommand{\sech}{\operatorname{sech}}
\newcommand{\defeq}{\stackrel{\rm{def}}{=}}

\usepackage{amssymb}
\usepackage{amsthm}
\usepackage{amsxtra}
\usepackage{graphicx}
\usepackage{listings}
\usepackage{color}

\newtheorem{theorem}{Theorem}

\newtheorem{proposition}[theorem]{Proposition}
\newtheorem{lemma}[theorem]{Lemma}
\newtheorem{corollary}[theorem]{Corollary}

\theoremstyle{remark}

\numberwithin{equation}{section}

\numberwithin{theorem}{section}

\numberwithin{table}{section}

\numberwithin{figure}{section}

\author[L. G. Farah]{Luiz Gustavo Farah}
\address{Department of Mathematics\\
Universidade Federal de Minas Gerais\\Belo Horizonte, Brazil}
\curraddr{}
\email{farah@mat.ufmg.br}
\thanks{}

\author[J. Holmer]{Justin Holmer}
\address{Department of Mathematics\\Brown University\\Providence, RI, USA}
\curraddr{} 
\email{justin\_holmer@brown.edu}
\thanks{}

\author[S. Roudenko]{Svetlana Roudenko}
\address{Department of Mathematics \& Statistics\\Florida International University,  Miami, FL, USA}
\curraddr{}
\email{sroudenko@fiu.edu}
\thanks{}

\author[K. Yang]{Kai Yang}
\address{College of Mathematics and Statistics, Chongqing University, Chongqing, China}
\curraddr{}
\email{yangkai10@cqu.edu.cn}
\thanks{} 



\title[Blow-up for 2D cubic ZK]{Blow-up in the 2D cubic Zakharov-Kuznetsov\\ equation in finite or infinite time}

\begin{document}

\begin{abstract}
We consider the 2D cubic Zakharov-Kuznetsov (ZK) equation, a physically relevant model in 
physics, which is a higher-dimensional extension of the generalized KdV equation. The cubic ZK equation is $L^2$-critical in 2D, and exhibits instability of solitons as we have shown in \cite{FHR3}. Here we prove that near-threshold negative energy solutions to this ZK equation blow-up in finite or infinite time, 
the first such result for higher-dimensional extensions of the gKdV family of equations.
The proof consists of several steps.  First, we show that if the blow-up conclusion is false, there are negative energy solutions arbitrarily close to the threshold that are globally bounded in $H^1$ and are spatially localized, uniformly in time.  In the second step, we show that such solutions must in fact be exact remodulations of the ground state, and hence, have zero energy, which is a contradiction. This second step, a nonlinear Liouville theorem, is proved by contradiction, with a limiting argument producing a nontrivial solution to a (linear) linearized ZK equation obeying uniform-in-time spatial localization.  Such nontrivial linear solutions are excluded by a local-virial space-time estimate. 
We introduce several new features to handle the 2D cubic ZK equation.  
\end{abstract}

\maketitle


\section{Introduction}

We consider the generalized Zakharov-Kuznetsov (gZK) equation
\begin{equation}\label{gZK}
\partial_t u + \partial_x \left(\Delta u + u^p \right) = 0, 
\end{equation}
where $(x,y_1,\ldots, y_{N-1}) \in \mathbb{R}^N$, $t\in \mathbb{R}$, $u$ is real-valued and $\Delta = \partial_x^2 +\partial_{y_1}^2 + \cdots +\partial_{y_{N-1}}^2$.   In particular, we study gZK in two dimensions ($N=2$) with the power of nonlinearity $p=3$. This equation is a higher-dimensional extension of the well-studied Korteweg-de Vries (KdV) equation, which models, for example, the weakly nonlinear waves in shallow water:
\begin{equation}
\label{gKdV}
\qquad \partial_t u + \partial_x^3 u + \partial_x (u^p) = 0, \quad p=2, \qquad x \in \cR, \qquad t \in \cR. \qquad
\end{equation}
When other integer powers $p \neq 2$ are considered, it is referred to as the {\it generalized} KdV (gKdV) equation. Despite its universality, the gKdV equation is limited as a spatially one-dimensional model. While there are several higher dimensional generalizations of it, in this paper we are interested in the gZK equation \eqref{gZK}. In the three dimensional setting and quadratic power ($N=3$, $p=2$), the equation \eqref{gZK} was originally proposed by Zakharov and Kuznetsov to describe weakly magnetized ion-acoustic waves in a strongly magnetized plasma \cite{ZK}, thus, the name of the equation. In two dimensions, it has also physical meaning; for example, with $p=2$, it governs the behavior of weakly nonlinear ion-acoustic waves in a plasma comprising cold ions and hot isothermal electrons in the presence of a uniform magnetic field \cite{MP-1, MP-2}. Melkonian and Maslowe \cite{MM-longwaves} showed that the equation \eqref{gZK} is the amplitude equation for two-dimensional long waves on the free surface of a thin film flowing down a vertical plane with moderate values of the surface fluid tension and large viscosity. Lannes, Linares \& Saut in \cite{LLS} derived the equation \eqref{gZK} from the Euler-Poisson system with magnetic field in the long wave limit, yet another derivation was carried by Han-Kwan in \cite{HK} from the Vlasov-Poisson system in a combined cold ions and long wave limit.

In this paper we consider $H^1$ solutions to the 2D cubic focusing Zakharov-Kuznetsov (ZK) equation 
\begin{equation}\label{ZK}
\partial_t u + \partial_x (\Delta u + u^3) =0.
\end{equation}
Such solutions have a maximal forward lifespan $[0,T)$, where either $T<+\infty$ or $T=+\infty$.   It follows from the local well-posedness theory\footnote{In fact, Linares \& Pastor \cite{LP} obtain local well-posedness in $H^s(\mathbb{R}^2)$ for $s>\frac34$, and Ribaud \& Vento \cite{RV} obtain local well-posedness for $s>\frac14$.} (see Faminski \cite{Fam}, Linares \& Pastor \cite{LP}, Ribaud \& Vento \cite{RV}) that if $T<+\infty$, then
$$
\lim_{t\nearrow T} \|\nabla u(t) \|_{L_{xy}^2} = +\infty.
$$
If $T=+\infty$, then either 
$$\lim_{t\nearrow T} \|\nabla u(t) \|_{L_{xy}^2} = +\infty \; \text{ or } \;\liminf_{t\nearrow T} \|\nabla u(t) \|_{L_{xy}^2} < +\infty$$
is \emph{a priori} possible. In either case $T<+\infty$ or $T=+\infty$, we say that $u(t)$ \emph{blows-up} at forward time $T$ if 
$$
\liminf_{t\nearrow T} \|\nabla u(t) \|_{L_{xy}^2} = +\infty.
$$

During their lifespan $[0,T)$, solutions $u(x,y,t)$ to ZK conserve mass and energy:
\begin{equation}\label{MC}
M(u(t))=\int_{\cR^2} u^2(t)\, dxdy = M(u(0))
\end{equation}
and
\begin{equation}\label{EC}
E(u(t))=\dfrac{1}{2}\int_{\cR^2}|\nabla u(t)|^2\;dxdy - \dfrac{1}{4}\int_{\cR^2} u^{4}(t)\;dxdy = E(u(0)).
\end{equation}

Similar to the gKdV equation, for solutions $u(x,y,t)$ decaying at infinity on $\cR^2$ the following invariance holds
\begin{equation}
\label{L1-inv}
\int_{\cR} u(x,y,t) \, dx = \int_{\cR} u(x, y,0) \, dx,
\end{equation}
which is obtained by integrating the original equation on $\cR$ in the first coordinate $x$.

One of the useful symmetries in the evolution equations is {\it the scaling invariance}, which states that an appropriately rescaled version of the original solution is also a solution of the equation. For the equation \eqref{gZK} it is
$$
u_\lambda(x,y,t)=\lambda^{\frac{2}{p-1}} u(\lambda x, \lambda y,\lambda^3t).
$$
This symmetry has the following effect on the Sobolev norm:
\begin{equation*}
\|u_\lambda (\bullet,\bullet,0) \|_{\dot{H}^s}=\lambda^{\frac{2}{p-1}+s-1} \|u_0\|_{\dot{H}^s},
\end{equation*}
and the index $s$ gives rise to the critical-type classification of equations.  For the gKdV equation \eqref{gKdV} the critical index is $s=\frac12-\frac2{p-1}$, and for the gZK equation \eqref{gZK} it is $s = \frac{N}2- \frac2{p-1}$. When $s=0$, the equation \eqref{gZK} is called the $L^2$-critical equation, in two dimensions this corresponds to $p=3$. The gZK equation has other invariances such as space and time translation, time reversal symmetry, and sign invariance.  We also note that gZK is not integrable unlike several other generalizations of KdV.

The gZK equation has a family of traveling waves (or solitary waves), moving only in positive $x$-direction
\begin{equation}\label{Eq:TW}
u(x,y,t) = Q_c(x-c t, y), \quad c>0,
\end{equation}
with $Q_c(x,y) \to 0$ as $|x| \to + \infty$. Here, $Q_c$ is the dilation of the ground state $Q$:
$$
Q_c(x,y) = c^{\frac{2}{{p-1}}} Q (cx,cy)
$$
with $Q$ being a radial positive solution in $H^1(\cR^2)$ of the well-known nonlinear elliptic equation $-\Delta Q+Q  - Q^p = 0$.  Note that $Q \in C^{\infty}(\R^2)$, $\partial_r Q(r) <0$ for any $r = |(x,y)|>0$ and for any multi-index $\alpha$
\begin{equation}\label{prop-Q}
|\partial^\alpha Q(x,y)| \leq c(\alpha) \, e^{-r} \quad \mbox{for any}\quad (x,y) \in \cR^2,
\end{equation}
see, for example, \cite[Thm. 1]{BL1983} or \cite[Prop. B.7]{Tao2006}.

It follows from energy and mass conservations and the sharp Gagliardo-Nirenberg inequality that if $\|u\|_{L^2}< \|Q\|_{L^2}$, then
\begin{equation}\label{E:Weinstein1}
\| \nabla u(t) \|_{L^2}^2 \leq 2\,  \left( {1 - \frac{\|u\|_{L^2}^2}{\|Q\|_{L^2}^2}}\right)^{-1} \, E(u).
\end{equation}
Thus, in this context solutions do not blow-up in finite or infinite time, and blow-up is only possible for solutions with $\|u\|_{L^2} \geq \|Q\|_{L^2}$.  

In her study of dispersive solitary waves in higher dimensions, de Bouard \cite{DeB} showed (in dimensions 2 and 3) that the traveling waves of the form \eqref{Eq:TW} are stable for $p<p_c$ and unstable for $p>p_c$, where $p_c=3$ in 2D. She followed the ideas developed for the gKdV equation by Bona, Souganidis \& Strauss \cite{BSS} for the instability, and Grillakis, Shatah \& Strauss \cite{GSS} for the stability arguments.  The first three authors of the present paper proved in \cite{FHR3} the instability of the traveling wave solutions of the form \eqref{Eq:TW} for the $p=3$ case, thus, completing the stability picture for the two-dimensional ZK equation (see also the alternative proof of instability for the $L^2$-supercritical gZK , $p>3$, in \cite{FHR2}). 
Further study, or a more refined question of stability, the asymptotic stability, was studied by C\^{o}te, Mu\~{n}oz, Pilod \& Simpson in \cite{CMPS}, where authors considered 2D quadratic ZK equation and were able to show the asymptotic stability, which could also be extended for nonlinearities slightly above quadratic, though not all the way up to the critical power $p< p_c=3$. This interesting work also motivated our studies of the 2D ZK equation.  

Once instability is shown, the next question to answer is if the instability leads to blow-up. Unlike other dispersive equations such as the nonlinear Schr\"odinger equation (NLS), the KdV-type equations do not have a convenient virial identity which usually gives a straightforward proof of existence of blow-up solutions.  The absence of a good virial identity makes it very difficult to prove the existence of finite time blow-up solutions in the gKdV-type equations, although intuitively for large nonlinearities such solutions should exists, \cite{K}, also, see numerical investigations in \cite{BDKM}. (The exception is gKP-I, where a virial-type identity shows the existence of blow-up; for $p \geq 4$ see \cite{Sau1} and a formal derivation in \cite{TuFa}, a refined version down to the (critical) power $p\geq 4/3$ is in \cite{Liu}.)  The breakthrough for the critical gKdV ($p=5$) equation in terms of proving the existence of blow-up solutions was obtained by Merle in \cite{M-blowup} (in finite or infinite time) using Martel \& Merle's previous result in \cite{MM-liouville} with further details about blow-up solutions in subsequent papers. In particular, they obtain an explicit construction of blow-up solutions close to solitons. The proof of the finite time blow-up was then developed in \cite{MM-KdV3} provided initial data have certain spatial decay, furthermore, an upper estimate on the blow-up rate was also given; later the universality of blow-up profile and lower bound estimates were obtained in \cite{MM-KdV4}.
The only extension of this approach to showing existence of blow-up to other KdV-type equations is \cite{KMR}, where authors show the existence of blow-up solutions in the finite or infinite time for the dispersion generalized $L^2$-critical Benjamin-Ono equation (dgBO) via a perturbation of the $L^2$-critical gKdV theory.  
\smallskip

The main result of this paper is
\begin{theorem}[main theorem]
\label{T:main}
There exists $\alpha_0>0$ such that the following holds.
Suppose that $u(t)$ is an $H^1$ solution to the ZK equation \eqref{ZK} with $E(u)<0$ and
\begin{equation}
\label{E:alpha-def}
\alpha(u) \defeq \|u\|^2 - \|Q\|^2 \leq \alpha_0.
\end{equation}
Then $u(t)$ blows up in finite or infinite forward time.
\end{theorem}
Note that $E(u)<0$ implies that $\alpha(u) >0$ from the inequality \eqref{E:Weinstein1}.  Also note that by time reversal symmetry, the result applies to negative time as well.
\smallskip

The entire paper is devoted to the proof of Theorem \ref{T:main}, which consists of several steps.  The first step, Prop. \ref{P:reduction}, under the assumption that the theorem is false, is the construction of a sequence of well-behaved solutions $\{ \tilde u_n \}$ of the ZK equation \eqref{ZK}, which are ultimately shown not to exist in Prop. \ref{P:nonlin-liouville}.
\begin{proposition}\label{P:reduction}{\rm (reduction to uniformly-in-time $H^1$ bounded and spatially localized sequence)}
If the statement of Theorem \ref{T:main} is false, then there exists a sequence of $H^1$ solutions $\{\tilde u_n(t)\}$ to the ZK equation \eqref{ZK} such that $E(\tilde u_n)<0$ and $\alpha(\tilde u_n) \to 0$ as $n \to \infty$  such that
\begin{enumerate}
\item (global existence and bounds)  each $\tilde u_n(t)$ exists globally in time  such that for all $t$
$$
\frac12 \|\nabla Q \|_{L^2} \leq \| \nabla \tilde u_n(t)\|_{L^2} \leq 2 \|\nabla Q\|_{L^2},
$$
\item (uniform-in-time spatial localization)  there exists a path $(\tilde x_n(t), \tilde y_n(t))$ and a scale $\tilde \lambda_n(t)$ such that the remainder function
$$
\tilde \epsilon_n(x,y,t) \defeq \tilde \lambda_n(t) \, \tilde u_n(\tilde\lambda_n(t) x + \tilde x_n(t), \tilde \lambda_n(t) y + \tilde y_n(t), t)-Q(x,y)
$$
satisfies the orthogonality conditions $\la \tilde \epsilon_n, \nabla Q\ra =0$, $\la \tilde \epsilon_n, Q^3 \ra=0$ and the uniform-in-time spatial localization: for each $r>0$,
$$
\| \tilde \epsilon_n  \|_{L_t^\infty L_{B(0,r)^c}^2} \lesssim e^{-\omega r} \| \tilde \epsilon_n \|_{L_t^\infty L_{xy}^2},
$$
where $B(0,r)$ denotes the ball, centered at $0$ and of radius $r>0$ in $\mathbb{R}^2$, $B(0,r)^c$ denotes the complement in $\mathbb{R}^2$, and $\omega>0$ is some absolute constant.
\end{enumerate} 
\end{proposition}

First, a preliminary $x$-compactness property is obtained -- there exists a path $\tilde x_n(t)$ such that for all $x$ and $t$,
$$
\| \tilde u_n(x+\tilde x_n(t),y,t) \|_{L_y^2} \lesssim C_n^{1/2}  e^{-|x|/32}.
$$
This is used to prove (1) in the statement of Prop. \ref{P:reduction}.  With (1) in hand, the stronger localization in (2) is obtained.  We next obtain the following rigidity statement (we keep the notation of a sequence $\tilde u_n$ rather than $u_n$ to remain consistent with the previous proposition).

\begin{proposition}[nonlinear Liouville property]
\label{P:nonlin-liouville}
Given a sequence of $H^1$ solutions $\{ \tilde u_n(t)\}$ to the ZK equation \eqref{ZK} such that  ~$0 \leq \alpha(\tilde u_n) \to 0$ as $n \to \infty$ satisfying properties (1) and (2) of Prop. \ref{P:reduction}, then the following holds.  For $n$ sufficiently large, there exist constants $\tilde \lambda_n>0$, $(\tilde x_n,\tilde y_n) \in \mathbb{R}^2$ such that 
$$
\tilde \lambda_n \, \tilde u_n\left(\tilde \lambda_n x + \tilde \lambda_n^{-2}t+  \tilde x_n,\tilde \lambda_n y + \tilde y_n, t\right) = Q(x,y).
$$
\end{proposition}
This proposition implies $\tilde{\lambda}_n^2 E(\tilde u_n) = E(Q)=0$ for sufficiently large $n$, which contradicts $E(\tilde u_n)<0$ for Prop. \ref{P:reduction}, completing the proof of Theorem \ref{T:main}.

The proof of Prop. \ref{P:nonlin-liouville} involves two steps.  First, assuming the conclusion of Prop. \ref{P:nonlin-liouville} is false, the proof of the convergence of remainders $\tilde \epsilon_n$ (after passing to a subsequence) to a nontrivial solution to a 
linearized ZK equation exhibiting uniform-in-time spatial localization.  This is given in Prop. \ref{P:conv-linear}. On the other hand, any uniformly-in-time spatially localized solution to the 
linearized ZK equation must be trivial by Prop. \ref{P:linear-liouville}, a \emph{linear} Liouville property.
\medskip

{\it Update.} The original version of this paper was posted to arxiv in 2018.  This revision streamlines \S\ref{section-M} by combining rotated monotonicity estimates from \S9 of the previous version and we provide additional Figures \ref{F:1}-\ref{F:4}. We remark that in our 3D ZK asymptotic stability paper \cite{FHRY}, we found a way to avoid using the rotated estimates and directly obtain decay outside the $\pi/3$ cone centered on the negative $x$-axis.  We left our original method in place here, to show an alternative approach.  Another update to the 2018 version included here is a clarification of the role of numerics in confirming the spectral estimates, which is unavoldable currently in this type of problems (see further comments at the end of the introduction and in the proof of Lemma \ref{L:v-virial}), we further investigate these properties with different numerically-assisted approaches in \cite{HRY2025}. The bibliography is also updated.   

\subsection{Outline}
We start with an outline of the proof of Prop. \ref{P:reduction}.  The arguments take place in the context of ``near threshold and negative energy'', for which there is a modulational characterization of solutions given by Lemmas \ref{L:varchar}, \ref{L:geom}, \ref{L:epparam} below.  It states that for every such solution $u(t)$ of the ZK equation \eqref{ZK}, there exists parameters of scale $\lambda(t)>0$ and position $(x(t),y(t)) \in \mathbb{R}^2$ such that the remainder
$$
\epsilon(x,y,t) = \lambda(t) u(\lambda(t)x+x(t),\lambda(t)y+y(t),t) - Q(x,y)
$$
is small in $H^1_{xy}$.    

While we are following the general pattern of argument introduced by Merle \cite{M-blowup} to address the corresponding problem for the $L^2$-critical gKdV equation, there are several aspects, in which we had to to make significant alternations to handle the 2D $L^2$-critical ZK setting, which are indicated in the discussion below.

The negation of Theorem \ref{T:main} yields a sequence of ZK solutions $\{u_n\}$, which, after renormalization, have the following properties: $\alpha(u_n) \to 0$ as $n\to \infty$; for each $n$ $E(u_n)<0$ (they are near threshold and negative energy solutions), and for all $t\geq 0$,
$$
(1-\frac1n) \| \nabla Q \|_{L^2_{xy}} \leq \| \nabla u_n(t) \|_{L^2_{xy}}.
$$
There is no \emph{a priori} upper bound on $\| \nabla u_n(t) \|_{L^2_{xy}}$, although we are granted the existence of sequence of times $t_{n,m} \to +\infty$ such that
$$
\lim_{m\to +\infty} \| \nabla u_n(t_{n,m}) \|_{L^2_{xy}} = \|\nabla Q \|_{L^2_{xy}}.
$$
For each $u_n(t)$ and the time sequence $t_{n,m} \to +\infty$, we can extract a weak limit
$$
u_n(\bullet + x(t_{n,m}), \bullet + y(t_{n,m}), t_{n,m}) \rightharpoonup \tilde u_n(\bullet, \bullet, 0) \quad \mbox{in} ~~ H^1_{xy}.
$$
The properties that $\alpha(\tilde u_n) \leq \alpha(u_n)$ and $E(\tilde u_n)<0$, proved in Lemma \ref{L:energyconstraints}, are inherited from the weak limit and modulational characterization of the near threshold solutions.

Taking $\tilde u_n(t)$ to be the ZK evolution of initial data $\tilde u_n(0)$, we have a result on the stability of weak convergence (Lemma \ref{L:weakchar})
\begin{equation}
\label{E:intro-1}
u_n(x+x_n(t_{n,m}+t), y+y_n(t_{n,m}+t), t_{n,m}+t)
\rightharpoonup \tilde u_n(x+\tilde x_n(t), y+\tilde y_n(t), t).
\end{equation}
We then show that the weak limiting process effectively strips away radiation leaving $\tilde u_n(t)$ with uniform in time, exponential decay in the $x$-spatial direction.    This is a consequence of monotonicity estimates (Lemma \ref{L:mon1}) analogous to that of Merle \cite{M-blowup} for gKdV.  It is key here that no upper bound on $\|\nabla u_n(t) \|_{L^2_{xy}}$ is required (equivalently, no lower bound on $\lambda_n(t)$).  To adapt the monotonicity estimates to the 2D context, we needed a  Gagliardo-Nirenberg estimate with a spatial weight in an external region (Lemma \ref{L:GN-weight}).  The monotonicity estimates are applied on long time scales for the solutions $u_n(t)$, which has the implication of decay for the weak limit $\tilde u_n(t)$ owing to the convergence \eqref{E:intro-1}.

With the decay estimates on hand, we can now use an integral type conservation law.  For smooth, rapidly decaying solutions to the ZK equation \eqref{ZK}, one can verify by direct computation from the  equation \eqref{ZK} that (integrating only in the variable $x$)
$$
\partial_t \int u(x,y,t) \, dx =0
$$
for each $y\in \mathbb{R}$, and thus, 
$$ 
\int u(x,y,t) \, dx \quad \text{ is constant in time}.
$$
Our solutions $\tilde u_n(t)$ do not have a high level of regularity, but at least do belong to $L_y^2L_x^1$ uniformly in time, from our decay estimates, and thus, the ability to approximate solutions by regular solutions (a consequence of the local theory machinery) yields that
$$
\left\| \int u(x,y,t) \, dx \right\|_{L_y^2} \text{ is constant in time}.
$$
This provides a means for controlling $\|\nabla \tilde u_n(t)\|_{L^2_y}$ from above (equivalently, the scale parameter $\tilde \lambda_n(t)$ from below).  This gives all the properties of $\tilde u_n(t)$ claimed in Prop. \ref{P:reduction}.

Let us highlight, at the more technical level, two ways in which our proof adds new elements to the method of Merle \cite{M-blowup}.  The ``stability of weak convergence" (Lemma \ref{L:weakchar}) is the statement that if $v_m(0) \to v(0)$ weakly in $H^1$, then $v_m(t) \to v(t)$ weakly in $H^1$, where $v_m(t)$ and $v(t)$ are the ZK flows of $v_m(0)$ and $v(0)$, respectively.  This was done in Appendix D of Martel \& Merle \cite{MM-liouville} for gKdV, and they used the fact that an $L^2$ local theory is available, which removes the need for an \emph{a priori} bound on the $H^1$ norm of $v_m(t)$.   In our case, an $L^2$ local theory for the ZK equation \eqref{ZK} is not available, and we need to assume an \emph{a priori} bound on the $H^1$ norms of $v_m(t)$, which means that we need to strengthen the bootstrap hypothesis on the admissible time scale $(-t_1(n),t_2(n))$, see \eqref{E:timeframe}, to include this assumption.  In the end, the argument still works, since the convergence of $\lambda_m(t) \to \lambda(t)$ and local theory estimates as employed in Lemma \ref{L:nontrivial} imply that a strengthening of the left part of the bootstrap hypothesis \eqref{E:timeframe} implies a corresponding strengthening of the right part.

In the monotonicity estimates (Lemma \ref{L:mon1}), we need a Gagliardo-Nirenberg estimate with spatial weight in an external spatial region (Lemma \ref{L:GN-weight}).  For this we found a nice classical proof, iterating the standard 1D estimate $\|f\|_{L^\infty}^2 \leq \|f\|_{L^2} \|f'\|_{L^2}$ with the weight appropriately distributed between the two terms.   It was necessary to have an estimate that positioned the weight fully on the gradient term in the right-side -- we could not use the method employed by \cite{CMPS} of passing through $H^1$ estimates, since then the $\alpha$ threshold of the result would be $H^1$ upper bound dependent, which is not suitable for our purposes.  It was suitable in the asymptotic stability context of \cite{CMPS} but not in our blow-up context, where there is no control on scale. 

Let us now give some details on the proof of Prop. \ref{P:nonlin-liouville}.  We note that the conclusion of Prop. \ref{P:nonlin-liouville} is equivalent to $\tilde \epsilon_n \equiv 0$.  Indeed, if $\tilde \epsilon_n \equiv 0$, we have for all $t$ that
$$
\tilde \lambda_n(t) \tilde u_n( \tilde\lambda_n(t) x + \tilde x_n(t), \tilde\lambda_n(t) y + \tilde y_n(t), t) = Q(x,y).
$$
Moreover, the estimates for parameters given in \eqref{E:param-dynam} yield, in this case, that $(\tilde \lambda_n)_t =0$, $(\tilde y_n)_t =0$ and $(\tilde x_n)_t = \tilde \lambda_n^{-2}$, from which it follows that $\tilde \lambda_n$ and $\tilde y_n$ are constant, and $\tilde x_n = \lambda_n^{-2}t + \tilde x_{n0}$.  Substituting, we obtain the conclusion as stated in Prop. \ref{P:nonlin-liouville}.  

We begin the proof of Prop. \ref{P:nonlin-liouville} by negating the conclusion, obtaining that there exists a subsequence (still labeled with $n$-subscript) such that $\tilde \epsilon_n \neq 0$ for all $n$.  We renormalize this sequence as follows.  Let $b_n = \|\tilde \epsilon_n \|_{L_t^\infty L_{xy}^2}$, which by our construction makes $b_n>0$ for all $n$.  Let $t_n\in \mathbb{R}$ be a time so that $\| \tilde \epsilon_n(t_n) \|_{L_{xy}^2} \geq \frac12 b_n$.  Then consider
\begin{equation}
\label{E:renorm-remainder}
w_n(t) = \frac{\tilde \epsilon_n(t+t_n)}{b_n}.
\end{equation}
By Prop. \ref{P:normcomp}, we in fact have that $\| \tilde \epsilon_n \|_{L_t^\infty H_{xy}^1} \sim \| \tilde \epsilon_n \|_{L_t^\infty L_{xy}^2}$, and hence, $\| w_n \|_{L_t^\infty H_{xy}^1} \lesssim 1$ for all $n$.  Moreover, by Prop. \ref{P:reduction}\,(2), it follows that  $w_n$ has a uniform-in-time spatial localization: for each $r>0$,
$$
\| w_n  \|_{L_t^\infty L_{B(0,r)^c}^2} \lesssim e^{-\omega r}, 
$$
where $B(0,r)$ denotes the ball, centered at $0$ and of radius $r>0$ in $\mathbb{R}^2$, $B(0,r)^c$ denotes the complement in $\mathbb{R}^2$, and $\omega>0$ is some absolute constant.  

By the Rellich-Kondrachov theorem, we can pass to a subsequence (still labeled with index $n$) so that $w_n(0)$ converges strongly in $L_{xy}^2$.  Denoting the limit by $w_\infty(0)$, it follows that $\|w_\infty(0)\|_{L_{xy}^2} \geq \frac12$ from the fact that $\| \tilde \epsilon_n(t_n) \|_{L_{xy}^2} \geq \frac12 b_n$.  Moreover, we prove, in Prop. \ref{P:conv-linear} that for each $T>0$, the strong convergence
$$w_n(t) \to w_\infty(t) \quad \text{in} \quad C([-T,T]; L_{xy}^2)$$
holds, where $w_\infty$ solves the 
linearized ZK equation \eqref{E:w-inf-eqn}.  The orthogonality conditions and uniform-in-time spatial decay are inherited in the limit, i.e.,
$$
\| w_\infty  \|_{L_t^\infty L_{B(0,r)^c}^2} \lesssim e^{-\omega r}. 
$$
A contradiction is obtained by appealing to the linear Liouville theorem, Prop. \ref{P:linear-liouville}, which forces that $w_\infty \equiv 0$.  The proof of Prop. \ref{P:linear-liouville} proceeds by deducing the additional orthogonality condition $\la w_\infty, Q\ra =0$, from which it follows that $\la \mathcal{L}w_\infty, w_\infty \ra$ is constant in time (here, $\mathcal{L}=1-\Delta - 3 Q^2$ is the linearized operator), and thus, the result follows from a dispersive estimate, the local virial estimate Lemma \ref{L:w-virial}
$$
\| w_\infty \|_{L_t^2 H_{xy}^1} \lesssim \| \la x \ra^{1/2} w_\infty \|_{L_t^\infty L_{xy}^2}.
$$
This type of estimate is ordinarily proved by a positive commutator argument leading to a spectral coercivity estimate for a Schr\"odinger operator $\tilde{\mathcal{L}}$ (different from $\mathcal{L}$).  If this is applied directly in this case, the corresponding $\tilde{\mathcal{L}}$ does not satisfy the needed coercivity estimate.  To escape this problem, we pass to the adjoint problem 
$$
v = (1-\delta \Delta)^{-1} \mathcal{L} w_\infty,
$$ 
for which the positive commutator argument yields an operator $\tilde{\tilde{\mathcal{L}}} $ that does indeed satisfy the coercivity estimate. This technique was used in the gKdV setting by Martel \cite[p. 775]{Martel-dual}.  The operator $\tilde{\tilde{\mathcal{L}}}$ (denoted by $A$ in Section \ref{S:virial-1}, see Lemma \ref{L:v-virial}) is of standard Schr\"odinger type but with a rank 2 perturbation (operator $P$ in our case, see Lemma \ref{L:v-virial}). 
To obtain coercivity estimate for $\tilde{\tilde{\mathcal{L}}}$ (or for the operator $A$ in our case), we obtain its discrete spectrum numerically and then use the angle lemma (Lemma \ref{L:angle}) to prove coercivity of $\tilde{\tilde{\mathcal L}}$,
giving the local viral estimate Lemma \ref{L:v-virial} 
$$
\| v \|_{L_t^2 H_{xy}^1} \lesssim \| \la x \ra^{1/2} v \|_{L_t^\infty L_{xy}^2}.
$$
The estimates allowing the conversion from $w_\infty$ to $v$ are given in \S \ref{S:conversion}.  

Prop. \ref{P:reduction}\,(2) combined with Prop. \ref{P:nonlin-liouville} were established as a ``nonlinear Liouville theorem'' in the case of $L^2$-critical gKdV by Martel \& Merle \cite{MM-liouville}.  Our proof uses some of the similar elements adapted to the 2D case -- for example, the comparability of $L^2$ and $H^1$ norms of the remainder functions $\tilde \epsilon_n$ (see Prop. \ref{P:normcomp}) proved by a virial-type estimate, and the convergence of renormalized remainders \eqref{E:renorm-remainder} to $w_\infty$, solving the 
linearized equation (see Prop. \ref{P:conv-linear}), proved via adaptation of the local ZK theory estimates in an appropriate reference frame.  However, our proof differs in the following aspects.  The decay estimate with sharp coefficient (as in our Prop. \ref{P:reduction}\,(2)) is proved in the gKdV case by Martel \& Merle \cite{MM-liouville} by appealing to the $L^2$-critical scattering theory available for that equation (Kenig, Ponce \& Vega \cite{KPV}). No such result is available for 2D cubic ZK (the best well-posedness result for 2D cubic ZK is $H^s$ for $s\geq \frac14$ by Kinoshita \cite{Kin}), so we instead prove a monotonicity estimate directly on the remainder $\epsilon$ (rescaled to $\eta$) -- see Lemma \ref{L:ep-strong-decay} and Corollary \ref{C:tilde-ep-decay}.   This type of calculation was previously done by Martel \& Merle \cite[Claim 14]{MM-refined} in the gKdV context, and we combine it here with our rotation method (see Lemma \ref{L:mon2}) 
to obtain the decay in $y$-direction. Finally, we prove the linear virial estimate, Lemma \ref{L:w-virial}, key to the proof of the linear Liouville property (Prop. \ref{P:linear-liouville}), by transforming from $w$ to $v= (1-\delta \Delta)^{-1}\mathcal{L}w$, and addressing the analogous estimate for $v$, which appears as Lemma \ref{L:v-virial}.   In \cite{MM-liouville}, this method was not used and the estimate for $w$ was proved directly, essentially by calculating $\partial_t \int x w^2 \, dx$, which leads to a positivity estimate for a Schr\"odinger operator $\tilde{\mathcal{L}}$.  If this were done in our case of 2D ZK, by computing $\partial_t \iint x w^2 \, dx \,dy$, we would obtain a Schr\"odinger operator $\tilde{\mathcal{L}}$, for which the positivity estimate seems to fail (as suggested by our numerics).  The method of converting from $w$ to $v$ was mentioned 
by Martel  \cite{Martel-dual} in the gKdV context, where the transformation $v= \mathcal{L}w$ is used.  The addition of the regularization operator, $v= (1-\delta \Delta)^{-1} \mathcal{L} w$ was used by Kenig \& Martel \cite{KM-BO} in their treatment of asymptotic stability for the Benjamin-Ono equation. We use it here, together with a few additional commutator arguments, that we establish in \S \ref{S:conversion}.  

Finally, we mention that the asymptotic stability for the 2D quadratic ZK was established by C\^{o}te, Mu\~{n}oz, Pilod \& Simpson \cite{CMPS} using a virial estimate established with the transformation $v=\mathcal{L}w$, without the use of regularization (and thus, requiring additional higher regularity estimates).  A key difference between \cite{CMPS} and our paper is the choice of orthogonality conditions.  The orthogonality conditions used in \cite{CMPS} do not seem to work in our context, as they are not sufficient to establish the spectral coercivity of $\tilde{\tilde{\mathcal{L}}}$ for nonlinearities $\partial_x (|u|^{p-1}u)$ outside $1.8<p<p^* \approx 2.1491$ as remarked in \cite[Appendix A.2.2]{CMPS}.  Our different choice of orthogonality conditions leads to a different linear operator $\tilde{\tilde{\mathcal{L}}}$, which includes a rank two projection operator. To obtain the discrete spectrum, we use numerics detailed in \S \ref{S:numerical}, and then prove the positivity of $\tilde{\tilde{\mathcal{L}}}$ via an angle lemma (Lemma \ref{L:angle}). 
Investigation of the spectrum of $\tilde{\tilde{\mathcal L}}$ (and also of $\tilde{\mathcal L}$) can be done via different approaches, which include numerically-assisted verifications to a certain extent, we discuss this elsewhere, for example, in \cite{HRY2025}; however, we note that even in the 1D case of the critical gKdV equation, a numerically-assisted proof was used in order to numerically obtain a function, an inverse under a certain elliptic operator of the explicit (in 1D) ground state $Q$ and verify the sign of an inner product/angle, see \cite[Lemma 26]{MM-liouville} and further discussion in the proof of Lemma \ref{L:v-virial}. 

\vspace{-.5cm}
{\small
\subsection{Acknowledgments}
L.G.F. was partially supported by Conselho Nacional de Desenvolvimento Cient\'ifico e Tecnol\'ogico (Brazil) - CNPq grant 307323/2023-4, Coordena\c{c}\~ao de Aperfei\c{c}oamento de Pessoal de N\'ivel Superior - CAPES grant 88881.974077/2024-01 and CAPES /COFECUB grant 88887.879175/2023-00.
J.H. was partially supported by the NSF grants DMS-1500106, DMS-2055072 and DMS 2452781.
S.R. was partially supported by the NSF CAREER grant DMS-1151618, NSF grants DMS-1815873/1927258, DMS-2055130 and DMS-2452782. 
K.Y. research support to work on this project came from grants DMS-1151618 and 2055130 (PI: Roudenko).
}

\section{Notation}

Throughout the paper we consistently use 
$$\alpha(u) = \|u\|_{L^2}^2 - \|Q\|_{L^2}^2.$$

Many results apply generally for solutions $u(t)$ to ZK that are ``near threshold negative energy'', that is, there exists $\alpha'>0$ (small) such that for all $H^1$ solutions $u(t)$ such that $\alpha(u)<\alpha'$ and $E(u)<0$, the result applies on the maximal lifespan of the solution $u(t)$ (perhaps together with some other hypotheses).  For such solutions, Lemmas \ref{L:varchar}, \ref{L:geom}, and \ref{L:epparam} apply, giving a description of the solution in terms of modulation parameters of scale $\lambda(t)>0$ and position $(x(t),y(t))$ and a remainder function
$$
\epsilon(x,y,t) \defeq \lambda(t)\, u(\lambda(t)x + x(t), \lambda(t)y+y(t),t)- Q(x,y)
$$
with dynamical information about the parameters and remainder function given by Lemma \ref{L:epparam}.
Results that apply in this general situation include the monotonicity estimate, Lemma \ref{L:mon1}, and the integral conservation law yielding scale control, Lemma \ref{L:scale-control}.  For each result of this type, one needs $\alpha(u)$ sufficiently small, so that we introduce a list of thresholds 
$$
\alpha_1 \geq \alpha_2 \geq \alpha_3 \geq \alpha_4 \geq \alpha_5  >0.
$$
Each new threshold $\alpha_j>0$ is taken smaller than the previously introduced thresholds $\alpha_1$, $\ldots$, $\alpha_{j-1}$, so that the earlier results apply as well.

The linearized operator is defined as
\begin{equation}
\label{L-def}
\mathcal{L} \defeq - \Delta  + 1 - 3 Q^2,
\end{equation}
where $Q$ is the unique radial positive solution in $H^1(\cR^2)$ of the well-known nonlinear elliptic equation 
\begin{equation}
\label{Q-def}
-\Delta Q+Q  - Q^3 = 0.
\end{equation}

We also introduce the generator $\Lambda$ of scaling symmetry
\begin{equation}
\label{Lambda}
\Lambda f = f + x\partial_x f + y \partial_y f.
\end{equation}

\section{Local theory estimates and ground state properties}

Let $U(t)\phi$ denote the solution to the linear homogenous problem
$$
\begin{cases}
\partial_t \rho + \partial_x \Delta \rho =0 \\
\rho(x,y,0) = \phi(x,y).
\end{cases}
$$
Then
$$
u(t) = U(t) \phi + \int_0^t U(t-t') \, \partial_x [u(t')^3] \, dt'.
$$

\begin{lemma}[linear homogeneous estimates]
\label{L:linhom}
We have
\begin{enumerate}
  \item
$\ds \|U(t) \phi \|_{L_T^\infty H_{xy}^1} \lesssim \| \phi \|_{H_{xy}^1},$

  \item
$\ds \|\partial_x U(t) \phi \|_{L_x^\infty L_{yT}^2} \lesssim \| \phi \|_{L_{xy}^2}.$
\end{enumerate}
For $0< T \leq 1$,
\begin{enumerate}
  \item[(3)]
$\ds \| U(t) \phi \|_{L_x^4 L_{yT}^\infty} \lesssim \|\phi \|_{H_{xy}^1}.$
\end{enumerate}
\end{lemma}
\begin{proof}
The first estimate is a standard consequence of Plancherel and Fourier representation of the solution.
The second estimate (local smoothing) is Faminskii \cite{Fam}, Theorem 2.2 on p. 1004.    The third estimate (maximal function estimate) is a special case ($s=1$) of Faminskii \cite{Fam}, Theorem 2.4 on p. 1007.  All of these estimates are used by Linares-Pastor \cite{LP}, and quoted as Lemma 2.7 on p. 1326 of that paper.
\end{proof}

\begin{lemma}[linear inhomogeneous estimates]
\label{L:lininhom}
For $0<T\leq 1$,
\begin{enumerate}
  \item
$\ds \left\| \int_0^t U(t-t') \, \partial_x f(t') \, dt' \right\|_{L_T^\infty H_{xy}^1 \cap L_x^4 L_{yT}^\infty} \lesssim \| \partial_x f \|_{L_x^1 L_{yT}^2}+ \| \partial_y f \|_{L_x^1 L_{yT}^2},$

  \item
$\ds \left\| \int_0^t U(t-t') f(t') \, dt' \right\|_{L_T^\infty H_{xy}^1 \cap L_x^4 L_{yT}^\infty} \lesssim \| f \|_{L_T^1 H_{xy}^1}.$
\end{enumerate}
\end{lemma}
\begin{proof}
These follow from Lemma \ref{L:linhom} by duality, $T^*T$, and the Christ-Kiselev lemma.  The needed version of the Christ-Kiselev lemma is provided by Molinet \& Ribaud \cite{MolRib}.  It is stated as Lemma 3 on p. 287 and proved in Appendix A on p. 307--311 of that paper.
\end{proof}

Let us now summarize the proof of $H_{xy}^1$ local well-posedness following from these estimates.  We note that Linares-Pastor \cite{LP}, in fact, achieved local well-posedness in $H_{xy}^s$ for $s>\frac34$, although we only need the $s=1$ case.
Let $X$ be the $R$-ball in the Banach space $C([0,T]; H_{xy}^1) \cap L_x^4 L_{yT}^\infty$, for $T$ and $R$ yet to be chosen.  Consider the mapping $\Lambda$ defined for $u\in X$ by
$$
\Lambda u = U(t) \phi + \int_0^t U(t-t') \partial_x [u(t')^3] \, dt'.
$$
Then we claim that for suitably chosen $R>0$ and $T>0$, we have $\Lambda:X\to X$ and $\Lambda$ is a contraction.  Indeed, by the estimates in Lemmas \ref{L:linhom} and \ref{L:lininhom}, we have
$$
\| \Lambda u \|_X \lesssim \| \phi\|_{H_{xy}^1} + \| \partial_x (u^3) \|_{L_x^1 L_{yT}^2} + \| \partial_y(u^3) \|_{L_x^1 L_{yT}^2}.
$$
We estimate
$$
\| u_x u^2 \|_{L_x^1 L_{yT}^2} \lesssim \| u_x \|_{L_x^2L_{yT}^2} \|u\|_{L_x^4 L_{yT}^\infty}^2 \leq  T^{1/2} \| u_x \|_{L_T^\infty L_{xy}^2} \|u\|_{L_x^4 L_{yT}^\infty}^2,
$$
and similarly, for the $x$ derivative replaced by the $y$ derivative.  Consequently,
$$
\| \Lambda u \|_X \leq C \| \phi\|_{H_{xy}^1} + C T^{1/2} \|u\|_X^3
$$
for some constant $C>0$.  By similar estimates,
$$
\| \Lambda u_2 - \Lambda u_1 \|_X \leq CT^{1/2} \| u_2-u_1 \|_X \max(\|u_1\|_X,\|u_2\|_X)^2.
$$
We can thus take $R = 2C\|\phi \|_{H_{xy}^1}$ and $T>0$ such that $CR^2T^{1/2}=\frac12$ to obtain that $\Lambda:X\to X$ and is a contraction.  The fixed point is the desired solution.

For the uniqueness statement, we can take $R \geq 2C\|\phi \|_{H_{xy}^1}$ large enough so that the two given solutions $u_1,u_2$ lie in $X$, and then take $T$ so that $CR^2T^{1/2}=\frac12$.  Then $u_1$ and $u_2$ are both fixed points of $\Lambda$ in $X$, and since fixed points of a contraction are unique, $u_1=u_2$.

We now state the properties of the linearized operator $\mathcal{L} = - \Delta + 1 -3 Q^2 $ (see Kwong \cite{K89} for all dimensions, Weinstein \cite{W85} for dimension 1 and 3, also Maris \cite{M02} and  Chang et al. \cite{CGNT}).
\begin{theorem}\label{L-prop}
The following holds for an operator $\mathcal{L}$ defined in \eqref{L-def}:
\begin{itemize}
\item
$\mathcal{L}$ is a self-adjoint operator and $\sigma_{ess}(\mathcal{L}) = [ 1, +\infty )$,

\item
$\ker \mathcal{L} = \mbox{span} \{Q_{x_1}, Q_{x_2} \}$,
	
\item
$\mathcal{L}$ has a unique single negative eigenvalue $-\lambda_0$ (with $\lambda_0 > 0$) associated to a
positive radially symmetric eigenfunction $\chi_0$. Without loss of generality, $\chi_0$ can be chosen
such that $\| \chi_0 \|_{L^2} = 1$. Moreover, there exists $\delta > 0$ such that
$|\chi_0(x)| \lesssim e^{- \delta |x|}$
for all $x \in \cR^2$.
\end{itemize}
\end{theorem}

\begin{lemma}\label{L-prop2}
The following identities hold
\begin{enumerate}
\item[(i)]
$\mathcal{L} \,Q = -2 Q^3$,

\item[(ii)]
$\mathcal{L} (\Lambda Q) = - 2 Q$, where $\Lambda$ is defined in \eqref{Lambda}.
Moreover,   $\int Q \, \Lambda Q = 0$.

\end{enumerate}
\end{lemma}

In \cite{FHR3} we summarized several known positivity estimates for the operator $\mathcal{L}$ following the works of Chang et al. \cite{CGNT} and Weinstein \cite{W85} (see Lemmas 3.3-3.6 in \cite{FHR3}). In this paper, we use the following set of orthogonality conditions to keep the quadratic form, generated by $\mathcal{L}$, positive-definite.

\begin{lemma}\label{Lemma-ort2}
For any $f \in H^1(\R^2)$ such that
\begin{equation}\label{Ort-Cond2}
\la f, Q^3 \ra = \la f, Q_{x_j} \ra =0, \quad j=1,2,
\end{equation}
there exists a universal constant $C_1>0$ such that
$$
 \la f,f \ra \leq C_1\, \la \mathcal{L}f,f \ra.
$$
\end{lemma}
\begin{proof}
See Lemma 3.5 in \cite{FHR3}.
\end{proof}

\section{Blow-up conclusion fails implies there exists\\ renormalized $u_n$ sequence and time sequence $t_{n,m}$}

We assume that the statement of Theorem \ref{T:main} is false, and hence, there is a sequence of solutions $\{ \bar u_n(t) \}$ of the ZK equation \eqref{ZK} such that for each given $n \in N$, the solutions $\bar u_n(t)$ are defined for all times $t \geq 0$ with $E(\bar u_n) < 0$ and $\alpha_n = \alpha(\bar u_n) \to 0$ as $n \to \infty$, and moreover, for each $n$
$$
\ell_n = \liminf_{t\to +\infty} \| \nabla \bar u_n(t)\|_{L_{xy}^2}  < +\infty.
$$
We note that for each $n$ and for each $t$, by Gagliardo-Nirenberg inequality
$$
0< -4E(\bar u_n) \leq \|\bar u_n(t)\|_{L_{xy}^4}^4 \lesssim \| \nabla \bar u_n(t) \|_{L_{xy}^2}^2 \| \bar u_n \|_{L_{xy}^2}^2,
$$
and hence, $\ell_n>0$.  By definition of $\ell_n$, there exists $\bar t_n \geq 0$ such that 
$$
\text{for all } t\geq \bar t_n \,, \quad \ell_n(1-\tfrac1n)  \leq \|\nabla \bar u_n(t) \|_{L_{xy}^2}
$$ 
and
$$
\| \nabla \bar u_n (\bar t_n ) \|_{L_{xy}^2} \leq \ell_n(1+\tfrac1n).
$$
Let us renormalize as
$$
u_n(x,y,t) \defeq \frac{ \|\nabla Q \|_{L^2}}{\ell_n} \,\,  \bar u_n\bigg( \frac{ \|\nabla Q \|_{L^2}}{\ell_n}x, \frac{ \|\nabla Q \|_{L^2}}{\ell_n} y, \frac{ \|\nabla Q \|_{L^2}^3}{\ell_n^3} t+\bar t_n \bigg)
$$
so that
\begin{equation}\label{E:upper}
\text{for all } n \text{ and } t\geq 0 \,, \qquad (1-\tfrac1n)\| \nabla Q\|_{L_{xy}^2}  \leq \|\nabla  u_n(t) \|_{L_{xy}^2}
\end{equation}
and
$$
\forall \; n\,, \quad \| \nabla u_n (0) \|_{L_{xy}^2} \leq (1+\tfrac1n) \|\nabla Q\|_{L_{xy}^2}
$$
as well as
\begin{equation}
\label{E:liminf}
\forall \; n\,, \quad  \liminf_{t\nearrow +\infty} \| \nabla u_n(t) \|_{L_{xy}^2} = \| \nabla Q\|_{L_{xy}^2}.
\end{equation}
We in addition have $E(u_n)<0$ for all $n$, and moreover, 
$$
\alpha_n \defeq \alpha(u_n)  = \alpha(\bar u_n) \to 0 \text{ as }n\to \infty
$$
as before.  By \eqref{E:liminf}, for each $n$, there exists a sequence $t_{n,m} \to +\infty$ as $m\to +\infty$ such that
\begin{equation}
\label{E:grad-conv}
\lim_{m\to +\infty} \| \nabla u_n(t_{n,m}) \|_{L^2_{xy}} = \| \nabla Q\|_{L_{xy}^2}.
\end{equation}
By passing to a subsequence, we can assume without loss that, for each $n$, the gaps are expanding:
$$
\lim_{m\to \infty} (t_{n,m+1}-t_{n,m}) = +\infty.
$$

\begin{lemma}[variational characterization and uniqueness of the ground state]
\label{L:varchar}
For each $\eta>0$ there exists $\alpha_1>0$ such that the following holds.  For each $\phi\in H^1$ with $E(\phi)<0$ and $\alpha(\phi) \leq \alpha_1$, there exists $x_0\in \mathbb{R}$, $y_0\in \mathbb{R}$, $\lambda>0$, $\mu \in \{-1,1\}$ such that 
$$
\| \mu \lambda \phi( \lambda x + x_0, \lambda y + y_0) -Q(x,y)\|_{H_x^1} \leq \eta.
$$
\end{lemma}
\begin{proof}
The proof is similar to that sketched in Lemma 1 of Merle  \cite{M-blowup}, p. 563.  It uses that any minimizer of the functional
\begin{equation}
\label{E:Weinstein}
u \mapsto \frac{\| u\|_{L^2}^2 \| \nabla u \|_{L^2}^2}{ \|u\|_{L^4}^4}
\end{equation}
is a minimal mass solution to the ground state equation, and moreover, one has uniqueness of solutions to the ground state equation (of minimal mass) up to translation and phase.  The lemma follows from these considerations plus concentration--compactness (or profile decomposition) lemmas applied to a minimizing sequence of the functional in \eqref{E:Weinstein}.
\end{proof}

Hence, we know that for each $n$ and $t\geq 0$, there exists $x_n(t)\in \mathbb{R}$, $y_n(t)\in \mathbb{R}$, $\mu_n(t) \in \{-1,1\}$ and $\lambda_n(t)>0$ such that  for all $t\geq 0$
$$
\| \mu_n(t) \lambda_n(t)\, u_n( \lambda_n(t) x + x_n(t), \lambda_n(t) y + y_n(t),t) -Q(x,y)\|_{H_x^1} \leq \delta_n \defeq \delta(\alpha_n),
$$
where $\delta(\alpha)>0$ is some function such that $\delta(\alpha) \searrow 0$ as $\alpha \searrow 0$. By continuity of the ZK flow in $t$, we know that each $\mu_n(t) \in \{-1,1\}$ is constant (independent of $t$).  Thus, we might as well redefine $u_n(t)$ as $\mu_n u_n(t)$ so that we can drop the $\mu_n$ parameter entirely.   

\begin{lemma}[geometrical decomposition]
\label{L:geom}
There exists\footnote{We require that $0<\alpha_2\leq \alpha_1$ so that Lemma \ref{L:varchar} applies.} $\alpha_2>0$ such that if $E(u)<0$ and $\alpha(u) \leq \alpha_2$, then there exist functions $\lambda(t)>0$, $x(t) \in \mathbb{R}$, $y(t) \in \mathbb{R}$ such that the remainder function
$$
\epsilon(x,y,t) \defeq \lambda(t)\, u(\lambda(t)x+x(t), \lambda(t) y + y(t), t) - Q(x,y)
$$
satisfies the orthogonality conditions
$$
\la \epsilon(t), \nabla Q \ra =0 \quad \mbox{and} \quad \la \epsilon(t), Q^3 \ra =0.
$$
\end{lemma}
\begin{proof}
Apply the implicit function theorem after invoking Lemma \ref{L:varchar}. (For a similar proof see Proposition 5.1 in \cite{FHR3}.) 
\end{proof}

For the solutions $u_n(t)$ discussed above, we apply Lemma \ref{L:geom} to modify the parameters $x_n(t)$, $y_n(t)$, $\lambda_n(t)$ so that the remainder function
$$
\epsilon_n(x,y,t) \defeq \lambda_n(t)\, u_n(\lambda_n(t) x + x_n(t), \lambda_n(t) y + y_n(t),t)-Q(x,y)
$$
satisfies the orthogonality conditions for all $t\geq 0$
$$
\la \epsilon_n(t), \nabla Q \ra =0 \quad \mbox{and} \quad \la \epsilon_n(t), Q^3 \ra =0.
$$ 

\begin{lemma}[properties of remainder function and parameters]
\label{L:epparam}
Let $u(t)$ solve ZK such that  $E(u)<0$ and $\alpha(u) \leq \alpha_2$ as in Lemma \ref{L:geom}, and let $\lambda(t)$, $x(t)$, and $y(t)$ be the parameters given by Lemma \ref{L:geom} on the maximal time interval of existence of $u(t)$. Then
$\lambda(t)$, $x(t)$, and $y(t)$ are $C^1$ functions and 
\begin{equation}
\label{E:param-dynam}
| \lambda^2 \lambda_t | +| \lambda^2 x_t - 1| + |\lambda^2 y_t| \lesssim \| \epsilon(t) \|_{L^2},
\end{equation}
and moreover,
there exists\footnote{We require that $0<\alpha_3\leq \alpha_2$ so that Lemmas \ref{L:varchar}, \ref{L:geom} apply as well.} $\alpha_3>0$ such that if $\alpha(u) \leq \alpha_3$, then 
\begin{equation}
\label{E:ep-control}
\| \epsilon(t) \|_{H^1} \lesssim \sqrt{\alpha(u)}.
\end{equation}
\end{lemma}  
\begin{proof}
The equation for $\epsilon$ is deduced by expressing $u$ in terms of $\epsilon$ and $Q$ and substituting into the equation for $u$ (the ZK equation), to obtain
$$
\lambda^3 \partial_t \epsilon = 
\begin{aligned}[t]
&\partial_x(\mathcal{L}\epsilon) + \lambda^2 \lambda_t \Lambda Q + (\lambda^2 x_t -1, \lambda^2y_t) \cdot \nabla Q \\
&+ \lambda^2 \lambda_t \Lambda \epsilon +  (\lambda^2 x_t -1, \lambda^2y_t) \cdot \nabla \epsilon - 3(Q \epsilon^2)_x - (\epsilon^3)_x. 
\end{aligned}
$$
The estimates \eqref{E:param-dynam} follow from computing $\partial_t$ of the orthogonality conditions and substituting the above equation for $\epsilon$.

The equation \eqref{E:ep-control} is a consequence of the following consideration.    Let 
$$
Z(u) = \frac12 M(u) +E(u),
$$ 
and observe that $Z'(Q)=0$ and $Z''(Q) = \mathcal{L} = 1- \Delta - 3Q^2$, the linearized operator. 
Note that Taylor's expansion
$$
Z(u_\lambda) = Z(Q+ \epsilon) = Z(Q) + \la Z'(Q),\epsilon\ra + \frac12 \la Z''(Q)\epsilon, \epsilon\ra + O(\epsilon^3),
$$
where $u_\lambda(x,y,t) = \lambda u(\lambda x, \lambda y, t)$, yields
$$
\frac12 \la Z''(Q)\epsilon, \epsilon\ra = Z(u_\lambda) - Z(Q) + O(\epsilon^3) = \alpha(u) + \lambda E(u) + O(\epsilon^3) \lesssim \alpha(u) + O(\epsilon^3),
$$
where we have used that $E(u)<0$.  Spectral considerations imply that the orthogonality conditions yield positivity of $\mathcal{L}$, which together with elliptic regularity (upgrading from $L^2$ to $H^1$) implies \eqref{E:ep-control}.
\end{proof}

Note that as a consequence of \eqref{E:ep-control}, we have the following.  By scaling,
$$
\lambda(t)^2 \| \nabla u(t)\|_{L^2}^2  = \| \nabla (\epsilon+ Q) \|_{L^2}^2 = \| \nabla \epsilon(t) \|_{L^2}^2 + 2 \la \nabla \epsilon, \nabla Q \ra + \| \nabla Q\|_{L^2}^2,
$$
from which we obtain
$$
\left|  \lambda(t)^2 \frac{\| \nabla u(t) \|_{L^2}^2}{ \| \nabla Q\|_{L^2}^2} - 1 \right| \lesssim \| \epsilon(t) \|_{H^1} \lesssim \sqrt{\alpha(u)}.
$$
This gives us conversion formulas
\begin{equation}
\label{E:gradient-lambda1}
 \frac{ \|\nabla Q\|_{L^2}}{\| \nabla u(t) \|_{L^2}} (1-C\sqrt{\alpha}) \leq \lambda(t) \leq \frac{ \|\nabla Q\|_{L^2}}{ \|\nabla u(t) \|_{L^2}} (1+ C\sqrt{\alpha})
 \end{equation}
and
\begin{equation}
\label{E:gradient-lambda2}
\lambda(t)^{-1} \| \nabla Q\|_{L^2} (1- C\sqrt{\alpha})  \leq  \|\nabla u(t) \|_{L^2} \leq \lambda(t)^{-1} \| \nabla Q\|_{L^2} (1+C\sqrt{\alpha})
\end{equation}
for some absolute constant $C>0$.

In later parts of the paper, we will need a more precise statement of the remainder equation and parameter dynamics.

\begin{lemma}
Suppose that $u(t)$ solves the ZK equation \eqref{ZK}, $\alpha(u) \ll 1$ and $E(u)<0$, so that the geometrical decomposition applies with orthogonality conditions $\la Q^3, \epsilon \ra =0$ and $\la \nabla Q, \epsilon \ra =0$. Rescaling time to $s(t)$, where $\frac{ds}{dt} = \lambda^{-3}$, we have that the remainder function satisfies the equation
\begin{equation}
\label{E:ep-eqn}
\partial_s \epsilon =
\begin{aligned}[t]
&\partial_x \mathcal{L} \epsilon + \frac{\lambda_s}{\lambda} \Lambda Q + (\frac{x_s}{\lambda}-1) Q_x  + \frac{y_s}{\lambda} Q_y +  \frac{\lambda_s}{\lambda} \Lambda \epsilon + (\frac{x_s}{\lambda}-1) \epsilon_x  + \frac{y_s}{\lambda} \epsilon_y \\
& - \partial_x ( 3Q\epsilon^2 + \epsilon^3),
\end{aligned}
\end{equation}
recalling that $\Lambda = 1+ x\partial_x + y \partial_y$ is the generator of scaling as in \eqref{Lambda}, 
and $\mathcal{L} = 1 - \Delta - 3Q^2$ is the linearized operator as defined in \eqref{L-def}.\footnote{We note, for the purposes of computation, that $\Lambda$ is skew-adjoint and $\mathcal{L}$ is self-adjoint.}  
 
Moreover, if we let $b \defeq \|\epsilon\|_{L_s^\infty L_{xy}^2}$, then 
the parameters satisfy the equations
\begin{equation}
\label{E:param-bds}
| \frac{\lambda_s}{\lambda} - \la f_1, \epsilon \ra | \lesssim b^2 \,, \qquad |( \frac{x_s}{\lambda} -1 ) - \la f_2, \epsilon \ra | \lesssim b^2 \,, \qquad |\frac{y_s}{\lambda} - \la f_3, \epsilon \ra | \lesssim b^2,
\end{equation}
where $f_j$ are the smooth, rapidly decaying spatial functions
$$
f_1 =  \frac{2}{\|Q\|_{L^4}^4} \mathcal{L}(Q^3)_x \,, \quad f_2 = \frac{1}{\|Q_x\|_{L^2}^2} \mathcal{L}Q_{xx} \,, \quad f_3 = \frac{1}{\|Q_y\|_{L^2}^2} \mathcal{L}Q_{xy}.
$$
\end{lemma}
\begin{proof}
Taking $\partial_s$ of the orthogonality conditions, we obtain three equations, which are collectively expressed as
$$
(A + B(\epsilon)) \begin{bmatrix} \frac{\lambda_s}{\lambda} \\ \frac{x_s}{\lambda}-1 \\ \frac{y_s}{\lambda} \end{bmatrix} =  \begin{bmatrix} \la \mathcal{L} (Q^3)_x, \epsilon \ra \\ \la \mathcal{L} Q_{xx} , \epsilon \ra \\ \la \mathcal{L} Q_{xy}, \epsilon  \ra \end{bmatrix} + \begin{bmatrix} 3\la (Q^3)_x Q, \epsilon^2 \ra + \la (Q^3)_x, \epsilon^3 \ra \\ 3 \la Q_{xx}Q, \epsilon^2 \ra + \la Q_{xx}, \epsilon^3 \ra \\ 3\la Q_{xy}Q, \epsilon^2 \ra + \la Q_{xy}, \epsilon^3 \ra \end{bmatrix},
$$
where 
$$
A = \begin{bmatrix} \frac12 \|Q\|_{L^4}^4 & & \\ & \|Q_x\|_{L^2}^2 & \\ & & \| Q_y \|_{L^2}^2 \end{bmatrix} \text{ and } B(\epsilon) = \begin{bmatrix} \la \Lambda (Q^3), \epsilon \ra & \la (Q^3)_x, \epsilon\ra  &  \la (Q^3)_y, \epsilon \ra \\ \la \Lambda Q_x, \epsilon \ra & \la Q_{xx}, \epsilon \ra & \la Q_{xy}, \epsilon \ra \\ \la \Lambda Q_y, \epsilon \ra & \la Q_{xy}, \epsilon \ra & \la Q_{yy}, \epsilon \ra \end{bmatrix}.
$$
Note that every entry $b_{ij}(\epsilon)$ of the matrix $B(\epsilon)$ satisfies $| b_{ij}(\epsilon)| \lesssim \|\epsilon \|_{L^2}$.  Using that 
$$
(A+B(\epsilon))^{-1} = [A(I+A^{-1}B(\epsilon))]^{-1} = (I+A^{-1}B(\epsilon))^{-1} A^{-1}
$$
and Neumann expansion of $(I+A^{-1}B(\epsilon))^{-1}$  that, if $b= \| \epsilon \|_{L_s^\infty L_{xy}^2} \ll 1$, then \eqref{E:param-bds} holds.
\end{proof}

Since $\frac{ds}{dt} = \frac{1}{\lambda^3}$, we have the conversions
$$ 
\frac{\lambda_s}{\lambda} = \lambda^2 \lambda_t \,, \qquad \frac{x_s}{\lambda} - 1 = \lambda^2( x_t - \lambda^{-2}) \,, \qquad \frac{y_s}{\lambda} = \lambda^2 y_t.
$$

As a side remark, for use in subsequent sections, define 
$$
\eta(x,y,t) = \lambda^{-1} \epsilon( \lambda^{-1}x, \lambda^{-1}y,t)
$$
and
$$
\tilde f_j(x,y) = \lambda^{-1} f_j(\lambda^{-1}x, \lambda^{-1}y).
$$
Then change of variables gives
$$
| \lambda^2 \lambda_t - \la \tilde f_1, \eta \ra | \lesssim b^2 \,, \quad |( \lambda^2 x_t -1 ) - \la \tilde f_2, \eta \ra | \lesssim b^2 \,, \quad | \lambda^2 y_t  - \la \tilde f_3, \eta\ra| \lesssim b^2,
$$
which is used in sections \ref{section-M} and \ref{S-10}.

Also, let
$$
\zeta(x,y,t) = b^{-1} \lambda^{-1} \epsilon( \lambda^{-1}(x-x(t)), \lambda^{-1}(y-y(t)),t),
$$
where $b = \|\epsilon\|_{L_t^\infty L_{xy}^2}$ and 
$$
\bar f_j(x,y) = \lambda^{-1} f_j(\lambda^{-1}(x-x(t)), \lambda^{-1}(y-y(t))).
$$
Then change of variables gives
$$
| b^{-1} \lambda^2 \lambda_t - \la \bar f_1, \zeta \ra | \lesssim b \,, \quad |b^{-1}( \lambda^2 x_t -1 ) - \la \bar f_2, \zeta \ra | \lesssim b \,, \quad | b^{-1}\lambda^2 y_t  - \la \bar f_3, \zeta \ra| \lesssim b,
$$
which is used in Sections \ref{section-C1} and \ref{section-C}.

\section{Extraction of ``future state'' weak limit $\tilde u_n(0)$, introduction of time frame $(-t_1(n),t_2(n))$, stability of the weak limit and applications}

By passing to a subsequence, we can assume that, for each $n$, 
\begin{equation}
\label{E:futurestate3}
u_n(\bullet + x(t_{n,m}), \bullet+ y(t_{n,m}), t_{n,m}) \rightharpoonup \tilde u_n(\bullet, \bullet, 0) \text{ weakly in } H_{xy}^1 \text{ as }m\to\infty
\end{equation}
for some $\tilde u_n(0) \in H^1_{xy}$.   

\begin{lemma}[energy constraints on $\tilde u_n$]
\label{L:energyconstraints} \qquad
\begin{enumerate}
\item For all $n$, we have $E(\tilde u_n(0)) \leq E(u_n(0))$ and $0< \alpha(\tilde u_n(0))\leq \alpha(u_n(0))$,
\item $\tilde u_n(0) \to Q$ in $H^1$ strongly as $n\to \infty$.
\end{enumerate}
\end{lemma}
\begin{proof}
This follows the proof of Lemma 7 in Merle \cite{M-blowup} given on p. 571-572 of that paper.
\end{proof}

Let $\tilde u_n(t)$ be the $H_{xy}^1$ evolution of initial data $\tilde u_n(0)$ by ZK.  Let $\tilde x_n(t)$, $\tilde y_n(t)$, and $\tilde \lambda_n(t)$ be the geometrical parameters associated with $\tilde u_n(t)$ on its maximal time interval of existence, as given by  Lemma \ref{L:geom}, noting that the corresponding remainder $\tilde \epsilon_n$ and these parameters satisfy the properties delineated in Lemma \ref{L:epparam}.

Let $(-t_1(n),t_2(n))$ be the \emph{maximal} time interval, on which 
\begin{equation}
\label{E:timeframe}
\frac12 \leq \tilde \lambda_n(t) \leq 2  \quad \text{and} \quad \frac12 \leq \liminf_{m\to \infty} \inf_t \lambda_n(t_{n,m}+t) \leq \limsup_{m\to \infty} \sup_t \lambda_n(t_{n,m}+t) \leq 2 
\end{equation}
hold.  By \eqref{E:gradient-lambda1} and \eqref{E:gradient-lambda2}, these can equivalently be viewed as upper and lower bounds on $\|\nabla \tilde u_n(t) \|_{L^2}$ and $\liminf_{m\to +\infty} \| \nabla u(t_{n,m}+t) \|_{L^2}$, $\limsup_{m\to +\infty} \|\nabla u(t_{n,m}+t)\|_{L^2}$.
We will ultimately show that for $n$ sufficiently large, $t_1(n)=t_2(n) = \infty$.  But in the meantime, we can argue that the time interval $(-t_1(n),t_2(n))$ is nontrivial (Lemma \ref{L:nontrivial}), which is not \emph{a priori} obvious due to the limits that appear in the definition \eqref{E:timeframe}.

Let
\begin{equation}
\label{E:futurestate}
v_{n,m}(\bullet, \bullet, 0) = u_n(\bullet + x(t_{n,m}), \bullet+ y(t_{n,m}), t_{n,m}).
\end{equation}
Let $v_{n,m}(t)$ denote the nonlinear evolution by ZK of initial condition $v_{n,m}(0)$.  Then by spatial and time translation invariance of the ZK flow,
\begin{equation}
\label{E:futurestate2}
v_{n,m}(x,y,t) = u_n(x+x(t_{n,m}), y+y(t_{n,m}), t_{n,m}+t).
\end{equation}

We will drop the $n$ subscript for a little while, since we consider fixed $n$, and state and prove two general lemmas (Lemma \ref{L:weakchar} and Lemma \ref{L:weakstab}) before applying Lemma \ref{L:weakstab} to $v_{n,m}$ as defined in \eqref{E:futurestate2} in Corollary \ref{C:weakstab} on the time frame $(-t_1(n),t_2(n))$ defined in \eqref{E:timeframe}.

\begin{lemma}[nontriviality of the time frame]
\label{L:nontrivial}
For each $n$, we have 
$$t_1(n), t_2(n) \gtrsim 1.$$
\end{lemma}
\begin{proof}
Let $q(x,t) = Q(x-t)$  and $\eta_m(t) = v_m(t) - q(t)$.  Note that \eqref{E:grad-conv} and \eqref{E:gradient-lambda1} imply that $|\lambda_n(t_{n,m}) - 1| \lesssim \sqrt{\alpha(u_n)}$ and hence $\|\eta_m(0) \|_{H^1} \lesssim \alpha(u_n)^{1/2}$.  

By substituting into the ZK equation for $v_m$, we obtain
\begin{equation}
\label{E-nu}
\partial_t \eta_m + \partial_x \Delta \eta_m + \partial_x ( (\eta_m+q)^3 - q^3 ) =0.
\end{equation}
Expanding the nonlinear term, we get 
$$
(\eta_m+q)^3 - q^3 = 3\eta_m q^3 + 3 \eta_m^2 q + \eta_m^3.
$$
We set up \eqref{E-nu} in Duhamel form and apply estimates in Lemmas \ref{L:linhom}, \ref{L:lininhom}, using that $\| \eta_m(0)\|_{H^1}\lesssim \sqrt{\alpha(u_n)}$, which is small, to obtain a solution $\eta_m(t)$ on an $O(1)$ time frame with $\|\eta_m(t)\|_{H^1} \lesssim \sqrt{\alpha(u_n)}$.  This implies that, on this unit time frame, $|\lambda_n(t_{n,m}+t) - 1| \lesssim \sqrt{\alpha(u_n)}$, and hence, the second part of \eqref{E:timeframe} is satisfied (note the estimate is uniform in $m$).  

The first part of \eqref{E:timeframe} is proved in the same way, using $v(t)$ as the reference solution in place of $v_m(t)$.
\end{proof}

For fixed smooth function $\chi(x,y)$ on $\mathbb{R}^2$ with $\chi(x,y) =1$ for $|(x,y)| \leq 1$ and $\chi(x,y)=0$ for $|(x,y)| \geq 2$, set
$$
\mathbf{1}_{\leq k}(x,y) = \chi \left( \frac{x}k, \frac{y}{k} \right) \,, \qquad \mathbf{1}_{\geq k}(x,y) = 1- \chi \left( \frac{x}k, \frac{y}{k} \right)
$$
for $k\in \mathbb{N}$.

The next statement is known in some variance, we state and prove it for completeness and convenience of the reader as well as adapting it to our notation. 
\begin{lemma}[characterization of weak limits] 
\label{L:weakchar}
The following are equivalent for a sequence $\{ v_m \} \subset H^1_{xy}$ bounded in $H^1_{xy}$, and $v\in H^1_{xy}$.  
\begin{enumerate}
\item $v_m \rightharpoonup v$ weakly in $H^1_{xy}$
\item for each $k\in \mathbb{N}$, $v_{m} \mathbf{1}_{\leq k} \to v \mathbf{1}_{\leq k}$ strongly in $L^2_{xy}$.    Equivalently, we state $v_m \to v$ strongly in $L^2_{\textnormal{loc}}$. 
\item For each subsequence $v_{m'}$ of $v_m$, there exists a subsequence $v_{m''}$ of $v_{m'}$ such that the following holds:  for each $k\in \mathbb{N}$, $v_{m''}\mathbf{1}_{\leq k} \to v \mathbf{1}_{\leq k}$ strongly in $L^2_{xy}$. 
\item For each subsequence $v_{m'}$ of $v_m$, there exists a subsequence $v_{m''}$  of $v_{m'}$ and radii $\rho_{m''} \to \infty$ such that the following holds:  $v_{m''} \mathbf{1}_{\leq \rho_{m''}} \to v$ strongly in $L^2_{xy}$.
\end{enumerate}
\end{lemma}
\begin{proof}
First, we note that $(2) \iff (3)$, which is a standard analytic fact proved by showing the negations are equivalent.

$(1) \implies (3)$.  Our argument will involve successive passive to subsequences starting with $v_{m'}$ ultimately yielding $v_{m''}$, although for convenience we will use the notation $v_m$ to refer to each such subsequence.  By the Rellich-Kondrachov compactness theorem, for $k=1$, there exists a subsequence in $m$ of $v_m \mathbf{1}_{\leq k}$ such that $v_m \mathbf{1}_{\leq k}$ converges strongly in $L^2_{xy}$.  Passing to a further subsequence, we can arrange that $v_m \mathbf{1}_{\leq k}$ converges strongly in $L^2_{xy}$ for $k=2$, and so forth.  By taking the diagonal subsequence, we now have a subsequence such that  for each $k\in \mathbb{N}$, the sequence $v_m \mathbf{1}_{\leq k}$ converges strongly in $L^2_{xy}$ as $m\to \infty$.  By the definition of weak convergence, it follows that, for fixed $k$, the value of this limit is $v \mathbf{1}_{\leq k}$.

$(2) \implies (4)$.  Let $v_{m'}$ be a given subsequence of $v_m$.   For convenience we relabel $v_{m'}$ as $v_m$.  By (2) we know that for each $k\in \mathbb{N}$, $v_{m} \mathbf{1}_{\leq k}$ converges strongly in $L^2_{xy}$.  Now we will pass to a subsequence as follows.  For each $\ell \in \mathbb{N}$, we will show that there exists $k_\ell$ and $m_{\ell}$ such that
\begin{equation}
\label{E:stretchingl2}
\| v_{m_\ell} \mathbf{1}_{\leq k_\ell} - v \|_{L^2_{xy}} \leq \frac{1}{\ell}
\end{equation}
and 
$$
m_\ell \to +\infty \text{ and } k_{\ell } \to +\infty \text{ as } \ell \to +\infty.
$$
Indeed, for a given $\ell$, first take $k_\ell$ sufficiently large (also requiring $k_\ell \geq \ell$) so that 
$$
\| v \mathbf{1}_{\leq k_\ell} - v \|_{L^2_{xy}} \leq \frac{1}{2\ell}.
$$
Then for this $k_\ell$, find $m_{\ell}$ sufficiently large (also requiring $m_\ell \geq \ell$) so that
$$
\| v_{m_\ell} \mathbf{1}_{\leq k_\ell} - v \mathbf{1}_{\leq k_{\ell}} \|_{L^2_{xy}} \leq \frac{1}{2\ell}.
$$
Combining the two gives \eqref{E:stretchingl2}.  Now it is convenient to relabel \eqref{E:stretchingl2} as the statement 
\begin{equation}
\label{E:stretchingl2b}
\| v_m \mathbf{1}_{\leq \rho_m} - v \|_{L^2_{xy}} \leq \frac{1}{m},
\end{equation}
which is achieved by replacing the original $v_m$ sequence with the subsequence constructed, and where now $\rho_m \to +\infty$ as $m\to +\infty$.

$(4) \implies (3)$.  Straightforward, since for fixed $k$, for sufficiently large $m''$, we have
$$
(v_{m''}-v) \mathbf{1}_{\leq k} = (v_{m''}-v) \mathbf{1}_{\leq \rho_{m''}} \mathbf{1}_{\leq k}
$$
and the right side $\to 0$ in $L^2$ by (4).

$(2) \implies (1)$.  Since we have assumed that $\{ v_m \} \subset H^1$ is bounded, we can apply the density of $C_c^\infty (\mathbb{R}^2)$ in $H^1(\mathbb{R}^2)$ to reduce to a test function in $\phi \in C_c^\infty(\mathbb{R}^2)$.  Let $k$ be any fixed integer larger than the support radius of $\phi$.  Then 
$$
|\la v_m -v , \phi\ra_{H^1}| =|\la v_m -v , (1-\Delta) \phi\ra_{L^2}| = |\la (v_m-v)\mathbf{1}_{\leq k}, (1-\Delta) \phi \ra_{L^2}|  
$$
$$
\leq \|  (v_m-v)\mathbf{1}_{\leq k}\|_{L^2} \| (1-\Delta) \phi\|_{L^2} \to 0
$$
by Cauchy-Schwarz and (2).
\end{proof}

\begin{lemma}[stability of weak limits under ZK flow]
\label{L:weakstab}
Suppose that $v_m(0) \rightharpoonup v(0)$ weakly in $H^1_{xy}$ and that there exists $k\in \mathbb{N}$ such  that for all $m$, $\|v_m(0) \mathbf{1}_{\geq k} \|_{L_x^2} \leq \frac12 \|Q\|_{L_x^2}$.  Letting $v_m(t)$ and $v(t)$ denote evolution of initial data $v_m(0)$ and $v(0)$, respectively, under the ZK flow, suppose that moreover 
$$
\| v(t) \|_{L_{[-T_-,T_+]}^\infty H_{xy}^1} < +\infty \qquad \text{ and } \qquad \limsup_{m\to \infty} \|v_m(t)\|_{L_{[-T_-,T_+]}^\infty H_{xy}^1} <+\infty,
$$ 
where $0\leq T_\pm <\infty$ (that is, $[-T_-,T_+]$ is a \emph{finite} time interval).   Then on $[-T_-,T_+]$,  $v_m(t) \rightharpoonup v(t)$ weakly in $H^1_{xy}$ and for each $k\in \mathbb{N}$, $v_m(t) \mathbf{1}_{\leq k}\to v(t) \mathbf{1}_{\leq k}$ strongly in $C([-T_-,T_+]; L^2)$.  
\end{lemma}
\begin{proof}
Let 
$$
M = \max( \limsup_{m\to \infty} \|v_m(t)\|_{L_{[-T_-,T_+]}^\infty H_{xy}^1}, \|v(t) \|_{L_{[-T_-,T_+]}^\infty H_{xy}^1}).
$$

Let $v_{m'}(0)$ be any subsequence of $v_m(0)$, and invoke Lemma  \ref{L:weakchar}  $(1)\implies (4)$ to obtain a subsequence $v_{m''}(0)$ and $\rho_{m''}\to \infty$ such that $v_{m''}(0)\mathbf{1}_{\leq \rho_{m''}} \to v(0)$ strongly in $L^2$.  We will ultimately show that $v_{m''}(t) \mathbf{1}_{\leq \rho_{m''}/2} \to v(t)$ strongly in $C([-T_-,T_+];L^2)$, and thus, can invoke Lemma \ref{L:weakchar} $(4)\implies (1)$ to conclude that $v_m(t) \rightharpoonup v(t)$ weakly in $H^1$ for each $t\in [-T_-,T_+]$.   Moreover, since $[-T_-,T_+]$ is compact, a statement similar to Lemma \ref{L:weakchar} holds, in which $L^2$ is replaced by $C([-T_-,T_+]; L^2)$ in (2), (3), (4), with the same proof of equivalence as given there. 

Replace $m''$ by $m$ for notational convenience, so that we have
$$
v_m(0) = v(0) + z_m(0) + q_m(0),
$$
where 
$$
z_m(0) = v_m(0) \mathbf{1}_{\leq \rho_m} - v(0) \quad  \text{ and } \quad q_m(0) = v_m(0) \mathbf{1}_{\geq \rho_m}.
$$
Let $z_m(t)$ and $q_m(t)$ be the ZK (nonlinear) evolution of $z_m(0)$ and $q_m(0)$, respectively.  Since $\|z_m(0) \|_{L^2} \to 0$ as $m\to \infty$, we can restrict to $m$ sufficiently large so that $\|z_m(0) \|_{L^2} \leq \frac12 \|Q\|_{L^2}$, and thus, energy estimates imply that $z_m(t)$ is globally bounded in $H^1$.  Also, our assumptions imply that $\|q_m(0) \|_{L^2} \leq \frac12 \|Q\|_{L^2}$, so that $q_m(t)$ is globally bounded in $H^1$ by energy estimates (see Claim 1-2 below). Let $r_m(t)$ be defined by
$$
v_m(t) = v(t) + z_m(t) + q_m(t) + r_m(t).
$$
We will prove that $z_m \to 0$ and $r_m \to 0$ strongly in $C([-T_-,T_+];L^2)$ and 
$$
\| q_m(t) \mathbf{1}_{\leq \rho_m/2}\|_{L^2}^2 \lesssim  M^2\max(T_-,T_+) / \rho_m,
$$
and hence,
$$
\lim_{m\to \infty}\| q_m(t) \mathbf{1}_{\leq \rho_m/2}\|_{C([-T_-,T_+];L^2)}  =0.
$$
Thus,
$$
\lim_{m\to \infty}\| (v_m(t)-v(t)) \mathbf{1}_{\leq \rho_m/2}\|_{C([-T_-,T_+];L^2)}=0,
$$
which will complete the proof.  Now we present some of the details.

\bigskip

\noindent\emph{Claim 1}.  $z_m(t)$, $q_m(t)$ are global in $H^1$ with 
$$
\| \nabla z_m(t) \|_{L^2} \leq 4M \,, \qquad \| \nabla q_m(t) \|_{L^2} \leq 4M,
$$
where 
$$
M = \max( \limsup_{m\to \infty} \|v_m(t)\|_{L_{[-T_-,T_+]}^\infty H_{xy}^1}, \|v(t) \|_{L_{[-T_-,T_+]}^\infty H_{xy}^1}).
$$

\bigskip

\noindent\emph{Proof of Claim 1}.  This is just use of the Weinstein inequality \eqref{E:Weinstein}.

\bigskip

\noindent\emph{Claim 2}.  On time intervals $I$ of length $|I| \lesssim M^{-4}$, we have
$$ 
\|z_m \|_{L_x^4L_{yI}^\infty}, \|q_m \|_{L_x^4L_{yI}^\infty}, \|v_m \|_{L_x^4L_{yI}^\infty}, \|v \|_{L_x^4L_{yI}^\infty} \lesssim  M.
$$

\bigskip

\noindent\emph{Proof of Claim 2}.  These estimates are inherited from the assumption of global $H^1$ bound of $M$ and the local theory estimates in Lemmas \ref{L:linhom}, \ref{L:lininhom}.

\bigskip

\noindent\emph{Claim 3}.  $\|q_m(t) \mathbf{1}_{<\rho_{m/2}} \|_{L^2_{xy}}^2 \lesssim t \rho_m^{-1}M^2$.

\bigskip

\noindent\emph{Proof of Claim 3}.  This is a localized mass estimate, which is obtained as follows:  let $\chi_m(x,y) = \chi( \frac{(x,y)}{\rho_m})$.  Then
$$
\partial_t \int \chi_m v_m(t)^2 \, dxdy = -2 \int \chi_m v_m \big( (v_m)_{xxx} + (v_m)_{xyy} + (v_m)^3_x \big) \, dxdy
$$
$$ 
=  \int \big[-3(\chi_m)_x (v_m)_x^2 - (\chi_m)_x (v_m)_y^2 - 2(\chi_m)_y (v_m)_x (v_m)_y + (\chi_m)_{xyy}v^2 + \tfrac12 (\chi_m)_x v_m^4 \big] \, dxdy.
$$
Hence,
$$
\left| \partial_t \int \chi_m v_m(t)^2 \, dxdy \right| \lesssim \frac{1}{\rho_m} (\| \nabla v_m(t)\|_{L^2}^2 + \|v_m(t)\|_{L^2}^2+ \|v_m(t) \|_{L^4}^4).
$$
By Gagliardo-Nirenberg, we have
$$
\left| \partial_t \int \chi_m v_m(t)^2 \, dxdy \right| \lesssim \frac{M^2}{\rho_m}.
$$
Integration in time yields the claim, using that $\chi_m q_m(0)=0$.  

\bigskip

\noindent\emph{Claim 4}.  On time interval $I=[t_\ell,t_r]$ of length $|I| \lesssim M^{-4}$, we have
$$
\|r_m(t) \|_{L^2} \leq 2 \|r_m(t_\ell )\|_{L^2} + \omega(m) M^2,
$$
where $\omega(m) \to 0$ as $m\to \infty$, uniformly in $t$ over $[T_-,T_+]$.  Specifically,
$$
\omega(m) \sim M^{-2} \|q_m\|_{L_I^\infty L_{xy}^2} + \| \mathbf{1}_{>\rho_m/2} v\|_{L_{xy}^2 L_I^2} + \| z_m \|_{L_I^\infty L_{xy}^2}.
$$
Before proceeding with the proof, let us note why $\lim_{m\to \infty} \| \mathbf{1}_{>\rho_m/2} v\|_{L_{xy}^2 L_I^2}=0$.  Since
$$
\|v\|_{L_{xy}^2 L_I^2} = \| v\|_{L_I^2 L_{xy}^2} \leq |I|^{1/2} \|v\|_{L_I^\infty L_{xy}^2} <\infty,
$$
it follows (by dominated convergence) that
$$
\lim_{k\to +\infty} \| \mathbf{1}_{> k} v\|_{L_{xy}^2 L_I^2} = 0.
$$

\bigskip

\noindent\emph{Proof of Claim 4}.  By plugging the decomposition into the equation for $v_m$, we obtain an equation for $r_m(t)$ 
$$
\partial_t r_m + \partial_x \Delta r_m + \partial_x F = 0,
$$
where
$$
F=(r_m+v+z_m+q_m)^3 - v^3 - z_m^3 - q_m^3.
$$
In the expansion of the nonlinearity 
$$
F = \sum_j F_j,
$$
there are no pure cubic terms, except for $r_m^3$, otherwise, there are only cross terms.   From Claim 1-2, we know that 
$$
\|r_m(t) \|_{H_{xy}^1}, \|r_m \|_{L_x^4 L_{yI}^\infty}  \lesssim M.
$$
We set up the $r_m$ equation and apply the estimates in Lemmas \ref{L:linhom}, \ref{L:lininhom} at the $L^2$ level (instead of the $H^1$ level) and obtain 
$$
\| r_m(t) \|_{L^2} \leq \|r_m(t_\ell)\|_{L^2} + C \|F \|_{L_x^1 L_{yI}^2}
$$
for some absolute constant $C>0$.  For each term $F_j = h_1h_2h_3$ in the expansion of $F$, we estimate as
$$
\| F_j \|_{L_x^1 L_{yI}^2} \leq \|h_1\|_{L_x^2L_{yI}^2} \| h_2 \|_{L_x^4 L_{yI}^\infty} \| h_3 \|_{L_x^4 L_{yI}^\infty}.
$$
For $h_1$, we could estimate as
$$
\|h_1\|_{L_x^2L_{yI}^2} \leq |I|^{1/2} \|h_1 \|_{L_I^\infty L_{xy}^2},
$$
although for terms in Case C below, we do it differently.

We consider the following three cases:

Case A.  At least one $r_m$ is present.  Put $h_1=r_m$ and absorb on the left.

Case B. No $r_m$ is present, and at least one $z_m$ is present.  Put $h_1 = z_m$.

Case C. No $r_m$ is present and no $z_m$ is present.  This consists of $vq_m^2$ and $v^2q_m$ and we do an in/out spatial decomposition, as follows:  for example, for $vq_m^2$, decompose as 
$$
vq_m^2 = (\mathbf{1}_{<\rho_m/2} q_m) v q_m + (\mathbf{1}_{>\rho_m/2} v) q_m^2,
$$
so
\begin{align*}
\| vq_m^2 \|_{L_x^1L_{yI}^2} &\leq \| (\mathbf{1}_{<\rho_m/2} q_m) v q_m \|_{L_x^1L_{yI}^2} + \| (\mathbf{1}_{>\rho_m/2} v) q_m^2 \|_{L_x^1L_{yI}^2} \\
&\leq |I|^{1/2} \| \mathbf{1}_{<\rho_m/2} q_m \|_{L_I^\infty L_{xy}^2} \| v \|_{L_x^4 L_{yI}^\infty}  \| q_m \|_{L_x^4 L_{yI}^\infty}  +  \| \mathbf{1}_{>\rho_m/2} v \|_{L_x^2 L_{yI}^2} \| q_m \|_{L_x^4 L_{yI}^\infty}^2 \\
&\leq ( |I|^{1/2} \| \mathbf{1}_{<\rho_m/2} q_m \|_{L_I^\infty L_{xy}^2}   +  \| \mathbf{1}_{>\rho_m/2} v \|_{L_x^2 L_{yI}^2}) M^2,
\end{align*}
which completes the proof of Claim 4.

\bigskip

Now we apply Claim 4 as follows.  Decompose $[0,T_+]$ into intervals 
$$
I_1=[t_0,t_1] \,, \; I_2=[t_1,t_2]\,,  \ldots , I_J=[t_{J-1},t_J]
$$ 
of length $\sim M^{-4}$ (thus, $J \sim M^4T_+$) so that Claim 4 applies on each subinterval. We have 
$$ \|r_m(t_1)\|_{L^2} \leq \omega_1(m) M^2,$$
$$\| r_m(t_2) \|_{L^2} \leq 2\|r_m(t_1)\|_{L^2} + \omega_2(m)M^2,$$
$$\| r_m(t_3) \|_{L^2} \leq 2\|r_m(t_2)\|_{L^2} + \omega_3(m) M^2,$$
and thus, combining, we obtain
$$
\|r_m(t_3)\|_{L^2} \leq (4 \omega_1(m)+ 2\omega_2(m) + \omega_3(m))M^2.
$$
Finally, after reaching $I_J$, we have
$$
\|r_m(T_+) \|_{L^2} \leq (2^{J-1} \omega_1(m) + \cdots + 2 \omega_{J-1}(m) + \omega_J(m))M^2.
$$
Hence, 
$$
\lim_{m\to \infty} \|r_m(t) \|_{L^2} = 0
$$
uniformly on $0\leq t \leq T_+$.  A similar argument applies to $[-T_-,0]$.  

\end{proof}

\begin{corollary}[application of stability of weak limits] 
\label{C:weakstab} 
For any fixed $n$, take any $t_1$, $t_2< \infty$ such that $(-t_1,t_2) \subset (-t_1(n),t_2(n))$.\footnote{In other words, if $t_2(n)<\infty$, then we can take $t_2=t_2(n)$, but if $t_2(n)=+\infty$, then we can take $t_2$ to be any \emph{finite} positive number.  Similarly, for $t_1$ in relation to $t_1(n)$.}  Then for $t\in [-t_1,t_2]$, we have
$$
u_n(t_{n,m}+t,x_n(t_{n,m})+\bullet, y_n(t_{n,m})+\bullet) \rightharpoonup \tilde u_n(t, \bullet, \bullet) \text{ weakly in }H^1\text{  as }m\to \infty,
$$
and for each $n$, for each $k\in \mathbb{N}$, 
$$
u_n(x_n(t_{n,m})+\bullet,y_n(t_{n,m})+ \bullet, t_{n,m}+t)  \mathbf{1}_{\leq k} \to \tilde u_n(\bullet,\bullet, t) \mathbf{1}_{\leq k}
$$
strongly in  $C([-t_1,t_2]; L^2)$  as $m\to \infty$. 
\end{corollary}
\begin{proof}
This is just Lemma \ref{L:weakstab} applied to $v_{n,m}(t)$ as defined in \eqref{E:futurestate2}, noting the definition of $\tilde u_n(0)$ in \eqref{E:futurestate3}, and $\tilde u_n(t)$ defined as the ZK evolution of initial data $\tilde u_n(0)$.  Recall that the time frame $(-t_1(n),t_2(n))$ has been defined so that the hypotheses of Lemma \ref{L:weakstab} are satisfied for the time interval $[-t_1,t_2]$.
\end{proof}

\begin{lemma}[convergence of geometric parameters]  
\label{L:conv-params} 
For any fixed $n$, take any $t_1$, $t_2< \infty$ such that $(-t_1,t_2) \subset (-t_1(n),t_2(n))$.  Then
$$\lambda_n(t_{n,m} + t) \to \tilde \lambda_n(t),$$
$$x_n(t_{n,m}+t)-x_n(t_{n,m}) \to \tilde x_n(t),$$ 
$$y_n(t_{n,m}+t)-y_n(t_{n,m}) \to \tilde y_n(t)$$
in  $C([-t_1,t_2]; \mathbb{R})$ as $m\to \infty$.
\end{lemma}
\begin{proof}
This is a consequence of Lemma \ref{L:epparam} and Corollary \ref{C:weakstab}.
\end{proof}

\begin{corollary}
\label{C:adjusted-conv}
For any fixed $n$, take any $t_1$, $t_2< \infty$ such that $(-t_1,t_2) \subset (-t_1(n),t_2(n))$.  Then for $t\in [-t_1,t_2]$, we have
$$
u_n(t_{n,m}+t,x_n(t_{n,m} +t )+\bullet, y_n(t_{n,m}+t)+\bullet) \rightharpoonup \tilde u_n(t, \bullet+\tilde x_n(t), \bullet+\tilde y_n(t))
$$ 
weakly in $H^1$ as $m\to \infty$
and for each $n$, for each $k\in \mathbb{N}$, 
$$
u_n(x_n(t_{n,m}+t)+\bullet,y_n(t_{n,m}+t)+ \bullet, t_{n,m}+t)  \mathbf{1}_{\leq k} \to \tilde u_n(\bullet+\tilde x_n(t),\bullet+\tilde y_n(t), t) \mathbf{1}_{\leq k}
$$
strongly in  $C([-t_1,t_2]; L^2)$  as $m\to \infty$. 
\end{corollary}
\begin{proof}
This follows from Corollary \ref{C:weakstab} and Lemma \ref{L:conv-params}.
\end{proof}

\section{$\tilde u_n$ has exponential decay in space, uniformly in time, via monotonicity}\label{section-M}

Recall that $\tilde u_n(x,y,0)$ is defined in \eqref{E:futurestate3}.  Denoting its time-evolution by $\tilde u_n(x,y,t)$, we observed Corollary \ref{C:weakstab} giving
$$
u_n(t_{n,m}+t,x_n(t_{n,m})+\bullet, y_n(t_{n,m})+\bullet) \rightharpoonup \tilde u_n(t, \bullet, \bullet) \text{ weakly in }H^1\text{  as }m\to \infty.
$$
In light of Lemma \ref{L:conv-params}, we have
$$
u_n(t_{n,m}+t,x_n(t_{n,m}+t)+\bullet, y_n(t_{n,m}+t)+\bullet) \rightharpoonup \tilde u_n(t, \tilde x_n(t)+ \bullet, \tilde y_n(t)+\bullet) \text{ weakly in }H^1\text{  as }m\to \infty.
$$

Before estimating $\tilde u_n$, we introduce a new weighted 2D Gagliardo-Nirenberg estimate.

\begin{lemma}[2D Gagliardo-Nirenberg with weight]
\label{L:GN-weight}
Suppose that $\psi(x,y)\geq 0$ is differentiable with the pointwise bound $|\nabla \psi(x,y)| \lesssim \psi(x,y)$, and for any $R_0\geq 0$, let
$$
B = \{ \, (x,y) \in \cR^2 : \, |x|>R_0 \text{ or }|y|>R_0 \, \}.
$$
Then
\begin{equation}
\label{E:weightedGN}
 \iint_B \psi(x) u(x,y)^4 \,dx \, dy \lesssim \| u \|_{L_B^2}^2 \iint_B \psi(x) (|\nabla u(x,y)|^2+ u(x,y)^2) \, dx \, dy
\end{equation}
with implicit constant independent of $R_0$ and $\psi$ (except for the implicit constant in the pointwise bound $|\nabla \psi(x,y)| \lesssim \psi(x,y)$).
\end{lemma}
\begin{proof}
We will apply the 1D inequality
$$
|v(x_0)|^2 \leq 2\|v\|_{L^2_{x>x_0}} \|v_x\|_{L^2_x}
$$
or
$$
|v(x_0)|^2 \leq 2\|v\|_{L^2_{x<x_0}} \|v_x\|_{L^2_x}
$$
in both $x$ (with $y$ fixed) and in $y$ (with $x$ fixed).  Because the argument relies only on these estimates, it clearly localizes to the spatial region $|x|\geq R_0$ or $|y|\geq R_0$.   For expository convenience we will ignore this spatial localization.  

For fixed $y$, we have, where $\psi = \psi_1\psi_2$ and $\psi_1$ and $\psi_2$ are yet to be chosen, by ``sup-ing out''
\begin{equation}
\label{E:GN1}
\int_x \psi u^4 \, dx \lesssim  \| \psi_1 u^2 \|_{L^\infty_x} \| \psi_2^{1/2} u \|_{L_x^2}^2.
\end{equation}
Apply the 1D estimate in $x$,
$$
\| \psi_1 u^2 \|_{L_x^\infty} \lesssim \int_x |(\psi_1)_x| u^2 \, dx + \int_x \psi_1 |u_x| |u| \, dx.
$$
In the first term, use $|(\psi_1)_x| \lesssim \psi_1$.  Split $\psi_1 = \psi_{11}\psi_{12}$ with $\psi_{11}$ and $\psi_{22}$ yet to be determined satisfying $|(\psi_{11})_x| \lesssim \psi_{11}$, and apply Cauchy-Schwarz to get
$$
\| \psi_1 u^2 \|_{L_x^\infty}  \lesssim (\| \psi_{11} u_x \|_{L_x^2}+ \| \psi_{11} u \|_{L_x^2}) \| \psi_{12} u \|_{L_x^2}.
$$
Plug this into \eqref{E:GN1} to obtain 
$$
\int_x \psi u^4 \, dx \lesssim  (\| \psi_{11} u_x \|_{L_x^2}+ \|\psi_{11} u \|_{L_x^2}) \| \psi_{12} u \|_{L_x^2} \| \psi_2^{1/2} u \|_{L_x^2}^2.
$$
Take $\psi_{12} = \psi_2^{1/2}$ to equalize the weights,
$$
\int_x \psi u^4 \, dx \lesssim  (\| \psi_{11} u_x \|_{L_x^2}+\| \psi_{11} u\|_{L_x^2}) \| \psi_2^{1/2} u \|_{L_x^2}^3.
$$
Apply the $y$ integral, and on the right-side Cauchy-Schwarz in $y$
$$
\int_x\int_y \psi u^4 \, dx \, dy \lesssim (\| \psi_{11} u_x \|_{L^2_{xy}}+\| \psi_{11} u \|_{L^2_{xy}}) \| \psi_2^{1/2} u \|_{L_y^6L_x^2}^3.
$$
Apply Minkowski's integral inequality to switch the norms in the last term
\begin{equation}
\label{E:GN2}
\int_x\int_y \psi u^4 \, dx \, dy \lesssim (\| \psi_{11} u_x \|_{L^2_{xy}}+\| \psi_{11} u \|_{L^2_{xy}} ) \| \psi_2^{1/2} u \|_{L_x^2L_y^6}^3.
\end{equation}
Now we focus on the ``inside'' of the last term for fixed $x$,  
$$
\| \psi_2^{1/2} u \|_{L_y^6}^6 = \int \psi_2^3 u^6 \, dy \lesssim \|\psi_2^3 u^4 \|_{L_y^\infty} \int  u^2  \, dy \lesssim  \|\psi_2^{3/2} u^2 \|_{L_y^\infty}^2 \|  u \|_{L_y^2}^2.
$$ 
Passing to the $1/6$ power, we get
$$ 
\| \psi_2^{1/2} u \|_{L_y^6} \lesssim \|\psi_2^{3/2} u^2 \|_{L_y^\infty}^{1/3} \|  u \|_{L_y^2}^{1/3}.
$$ 
For the first term on the right side of this estimate, we apply the 1D fundamental theorem of calculus estimate 
$$
\| \psi_2^{1/2} u \|_{L_y^6} \lesssim \| (\psi_2^{3/2} u^2)_y \|_{L_y^1}^{1/3} \|  u \|_{L_y^2}^{1/3}.
$$ 
Distributing the $y$ derivative, we obtain 
$$
\| \psi_2^{1/2} u \|_{L_y^6} \lesssim (\| \psi_2^{3/2} u_y u \|_{L_y^1} + \| \psi_2^{1/2}(\psi_2)_y u^2 \|_{L_y^1} )^{1/3} \|  u \|_{L_y^2}^{1/3}.
$$ 
Using that $|(\psi_2)_y| \lesssim \psi_2$ and Cauchy-Schwarz, we continue
$$
\| \psi_2^{1/2} u \|_{L_y^6} \lesssim (\| \psi_2^{3/2} u_y \|_{L_y^2}+\| \psi_2^{3/2} u \|_{L_y^2})^{1/3} \|u \|_{L_y^2}^{2/3}.
$$ 
Now apply the $L_x^2$ norm, and use H\"older with partition $\frac12=\frac16+\frac13$ to obtain
$$
\| \psi_2^{1/2} u \|_{L_x^2 L_y^6} \lesssim (\| \psi_2^{3/2} u_y \|_{L_{xy}^2}+\| \psi_2^{3/2} u \|_{L_{xy}^2})^{1/3} \|u \|_{L_{xy}^2}^{2/3}.
$$ 
Insert this into \eqref{E:GN2} to obtain
$$
\int_x\int_y \psi u^4 \, dx \, dy \lesssim (\| \psi_{11} u_x \|_{L_{xy}^2}+\| \psi_{11} u \|_{L_{xy}^2})( \| \psi_2^{3/2} u_y \|_{L_{xy}^2}+\| \psi_2^{3/2} u \|_{L_{xy}^2}) \| u \|_{L_{xy}^2}^2.
$$
Now let us review the weight partitions.  We took $\psi = \psi_1\psi_2$, $\psi_1 = \psi_{11}\psi_{12}$, so that $\psi= \psi_{11}\psi_{12}\psi_2$ but we required that $\psi_{12}=\psi_2^{1/2}$, so that in fact $\psi= \psi_{11}\psi_2^{3/2}$, with $\psi_{12}$ and $\psi_2$ yet to be determined.  From the last inequality, we see that we would like to have $\psi_{11} = \psi^{1/2}$ and $\psi_2^{3/2} = \psi^{1/2}$, and this in fact does meet the condition $\psi= \psi_{11}\psi_2^{3/2}$.  Note that since $\psi_{11} = \psi^{1/2}$, $\psi_{12} = \psi^{1/6}$,  $\psi_2 = \psi^{1/3}$, these weights inherit the property that absolute value of the derivative is bounded by a constant times the function.

Thus, we have
$$
\int_x\int_y \psi u^4 \, dx \, dy \lesssim (\| \psi^{1/2} u_x \|_{L_{xy}^2}+\| \psi^{1/2} u \|_{L_{xy}^2})( \| \psi^{1/2} u_y \|_{L_{xy}^2}+\| \psi^{1/2} u \|_{L_{xy}^2} ) \| u \|_{L_{xy}^2}^2.
$$
\end{proof}

\begin{lemma}[$I_\pm$ estimates]
\label{L:mon1}
Let $t_{-1}<t_0<t_1$ and suppose that $u(t)$ is an $H^1$ solution to ZK on $[t_{-1},t_1]$ with $E(u)<0$ and
$$
\forall \, t\in [t_{-1},t_1]  \qquad \| \nabla u(t) \|_{L_{xy}^2} \geq 0.9 \|\nabla Q\|_{L_{xy}^2}.
$$
There exists an absolute constant $\alpha_4>0$ such that if 
$$
\alpha(u) \defeq \|u\|_{L^2}^2 - \|Q\|_{L^2}^2  \leq \alpha_4,
$$  
we then have the following.  Let
$$
I_{\pm,x_0,t_0}(t) = \iint u^2(x+x(t_0),y, t)\, \phi_\pm(x-x_0-\tfrac12(x(t)-x(t_0))) \, dx \, dy,
$$
where $\phi_-(x)=\phi_+(-x)$ and
$$
\phi_+(x) = \frac{2}{\pi} \arctan(e^{x/K}),
$$
so that $\phi_+(x)$ is increasing with $\lim_{x\to -\infty} \phi_+(x) = 0$ and $\lim_{x\to +\infty} \phi_+(x) = 1$, and $\phi_-(x)$ is decreasing with $\lim_{x\to -\infty} \phi_-(x) = 1$ and $\lim_{x\to +\infty} \phi_-(x) = 0$.  Let $x_0>0$ and $K \geq 4$.  For the increasing weight, we have two estimates that bound the \emph{future} in terms of the \emph{past}
\begin{equation}
\label{E:mo1}
\text{for } t_{-1}<t_0\,, \quad I_{+,x_0,t_0}(t_0) \leq I_{+,x_0,t_0}(t_{-1}) + \vartheta_0 e^{-x_0/K},
\end{equation}
\begin{equation}
\label{E:mo2}
\text{for } t_0<t_1\,, \quad I_{+,-x_0,t_0}(t_1) \leq I_{+,-x_0,t_0}(t_0) +\vartheta_0  e^{-x_0/K}.
\end{equation}
For the decreasing weight, we have two estimates that bound the \emph{past} in terms of the \emph{future}
\begin{equation}
\label{E:mo3}
\text{for }t_{-1}<t_0\,, \quad I_{-,x_0,t_0}(t_{-1}) \leq I_{-,x_0,t_0}(t_0) +\vartheta_0 e^{-x_0/K},
\end{equation}
\begin{equation}
\label{E:mo4}
\text{for }t_0<t_1\,, \quad I_{-,-x_0,t_0}(t_0) \leq I_{-,-x_0,t_0}(t_1) +\vartheta_0 e^{-x_0/K}.
\end{equation}
Here, $\vartheta_0$ is some absolute constant.
These are depicted in Figure \ref{F:mon-paths}, which shows the path of the ``center'' of transition of the weight $\phi_\pm$ through time.

Importantly, both $\alpha_4>0$ and $\vartheta_0>0$ are absolute constants, in particular, independent	of the \emph{upper} bound on $\|\nabla u(t) \|_{L_{xy}^2}$ over the time interval $[t_{-1},t_1]$. 
\end{lemma}

\begin{figure}
\includegraphics[scale=0.75]{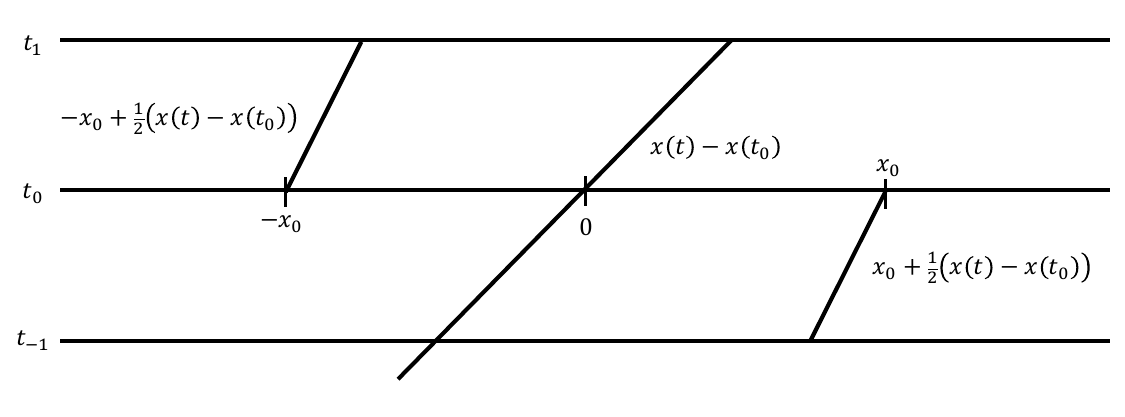}\quad 
\caption{\footnotesize Here, $x_0>0$. The function $u(x+x(t_0),y,t)$ has the soliton component centered at position $(x(t)-x(t_0),y(t))$. In this figure, only the $x$ spatial direction depicted horizontally, and time vertically. At time $t=t_0$ the soliton is centered in this frame of reference at position $x(t)-x(t_0)=0$, and the soliton trajectory is the line $x(t)-x(t_0)$ with slope approximately $1$. For the functional $I_+$, with increasing weight $\phi_+$, we can estimate \emph{forward} in time from $t_{-1}$ to $t_0$ or from $t_0$ to $t_1$.  In the case of $t_{-1}$ to $t_0$, the weight $\phi_+$ has transition centered along the right half-slope line $x_0+ \frac12(x(t)-x(t_0))$ and in the case of $t_0$ to $t_1$, the weight $\phi_+$ has transition centered along the left half-slope line $-x_0+\frac12(x(t)-x(t_0))$.  In either case, the essential aspect is that the soliton and weight center trajectories are separating in time as we move forward or backward. For $I_-$, with decreasing weight $\phi_-$, the trajectories are the same except the estimates are backwards in time from $t_1$ to $t_0$ or from $t_0$ to $t_{-1}$.
\label{F:mon-paths}}
\end{figure}

\begin{proof}
By time translation, it suffices to assume that $t_0=0$, and recall that $x_0>0$.  For ease of reference, let us rewrite them in this case:
For the increasing weight, we have two estimates that bound the \emph{future} in terms of the \emph{past}
\begin{equation}
\label{E:mo10}
\text{for }t_{-1}<0\,, \quad I_{+,x_0,0}(0) \leq I_{+,x_0,0}(t_{-1}) + \vartheta_0 e^{-x_0/K},
\end{equation}
\begin{equation}
\label{E:mo20}
\text{for }0<t_1\,, \quad I_{+,-x_0,0}(t_1) \leq I_{+,-x_0,0}(0) +\vartheta_0 e^{-x_0/K}.
\end{equation}
For the decreasing weight, we have two estimates that bound the \emph{past} in terms of the \emph{future}
\begin{equation}
\label{E:mo30}
\text{for } t_{-1}<0\,, \quad I_{-,x_0,0}(t_{-1}) \leq I_{-,x_0,0}(0) +\vartheta_0 e^{-x_0/K},
\end{equation}
\begin{equation}
\label{E:mo40}
\text{for }0<t_1\,, \quad I_{-,-x_0,0}(0) \leq I_{-,-x_0,0}(t_1) +\vartheta_0 e^{-x_0/K}.
\end{equation}

To compare the two estimates \eqref{E:mo10}, \eqref{E:mo20} with the two estimates \eqref{E:mo30}, \eqref{E:mo40}, we temporarily augment the notation to include reference to the solution $u$:  $I_{\pm,x_0,t_0,u}(t)$.  
Now if $\tilde u(x,t) = u(-x,-y,-t)$, then the corresponding soliton path parameters are $\tilde x(t) = -x(-t)$ and $\tilde y(t) = -y(-t)$. We obtain
$$
I_{+,x_0,0,\tilde u}(t) = \iint \tilde u^2(x+\tilde x(0),y,t) \phi_+(x-x_0-\frac12(\tilde x(t)-\tilde x(0))) \, dx \,dy.
$$
Change variable $x\to -x$ and $y \to -y$ to get
\begin{align*}
&= \iint u^2(x+x(0),y,-t) \phi_+(-x+x_0+\frac12(x(-t)-x(0))) \, dx \, dy \\
& = \iint u^2(x+x(0),y,-t) \phi_-(x-x_0-\frac12(x(-t)-x(0))) \, dx \, dy \\
& = I_{-,-x_0,0,u}(-t).
\end{align*}
In summary, we have established the identity
$$
I_{+,x_0,0,\tilde u}(t) = I_{-,-x_0,0,u}(-t).
$$
Suppose that we have proved \eqref{E:mo10}, \eqref{E:mo20}.  Now applying \eqref{E:mo10} to $\tilde u$, we get
$$
I_{+,x_0,0,\tilde u}(0) \leq I_{+,x_0,0,\tilde u}(t_{-1}) + \vartheta_0 e^{-x_0/K}.
$$
Applying the identity, this converts to
$$
I_{-,-x_0,0,u}(0) \leq I_{-,-x_0,0,u}(-t_{-1}) + \vartheta_0 e^{-x_0/K}.
$$
Taking $t_1 = -t_{-1}$ and replacing $-x_0$ by $x_0$, we obtain \eqref{E:mo40}.  

Applying \eqref{E:mo20} to $\tilde u$, we get
$$
I_{+,x_0,0,\tilde u}(t_1) \leq I_{+,x_0,0,\tilde u}(0) + \vartheta_0 e^{x_0/K}.
$$
Applying the identity,
$$
I_{-,-x_0,0,u}(-t_1) \leq I_{-,-x_0,0,u}(0) + \vartheta_0 e^{x_0/K}.
$$
Taking $t_{-1} = t_1$ and replacing $x_0$ by $-x_0$, we obtain \eqref{E:mo30}.  

Thus, we have  shown that \eqref{E:mo10} and \eqref{E:mo20} imply \eqref{E:mo40} and \eqref{E:mo30}, and it suffices for us to establish \eqref{E:mo10} and \eqref{E:mo20}.  For the remainder of the proof, we will drop the $+$ subscript.  
\smallskip

Both \eqref{E:mo10} and \eqref{E:mo20} are proved by the same principle, which is to take the time derivative of $I_{x_0,0}(t)$ and then note that it suffices to estimate the nonlinear term.  The integration region is divided into two parts -- near the soliton trajectory, which is estimated using the smallness of $\phi'$, and away from the soliton trajectory, which is estimated using the smallness of $u$.
\smallskip

We start by noting that $0.9 \leq \lambda^2 x_t \leq 1.1$  and $\lambda \leq 1.1$ implies that
$$
0.7 < \frac{0.9}{1.1^2} \leq \frac{0.9}{\lambda^2} \leq x_t.
$$

We have
$$
\phi'(x) = \frac{1}{\pi K} \sech(x/K) \,, \quad \phi''(x) = -\frac{1}{\pi K^2}\sech(x/K)\tanh(x/K) \,,
$$
$$
\phi'''(x) = \frac{1}{\pi K^3} (\sech(x/K)\tanh^2(x/K) - \sech^3(x/K)). 
$$
In particular, we have that
\begin{equation}
\label{E:mon10}
|\phi'''(x)| \leq \frac{1}{K^2} \phi'(x).
\end{equation}
Let
$$
I_{x_0,t_0}(t) = \iint_{xy} u^2(x+x(t_0),y,t) \, \phi(x-x_0- \tfrac12(x(t)-x(t_0))) \, dx \, dy.
$$
We compute
\begin{align*}
\tfrac12\partial_t I_{x_0,t_0}(t) &= \iint_{xy} u u_t \phi \, dx \, dy - \tfrac12x'(t) \iint_{xy} u^2 \phi' \, dx \, dy \\
& = \iint_{xy} u(-u_{xxx}-u_{yyx}-3u^2u_x) \phi \, dx \, dy -  \tfrac12x'(t) \iint_{xy} u^2 \phi' \, dx \, dy.
\end{align*}
After several applications of integration by parts, we arrive at
\begin{equation}
\label{E:mon14}
\tfrac12\partial_t I_{x_0,t_0}(t) = 
\begin{aligned}[t]
&-\tfrac32 \iint u_x^2 \phi' -\tfrac12 \iint u_y^2 \phi' - \tfrac12x'(t) \iint u^2 \phi'\\
&  + \tfrac12 \iint u^2 \phi''' +\tfrac34 \iint u^4 \phi' .
\end{aligned}
\end{equation}
Since $0.7< x_t$, the term 
$$ - \tfrac12x'(t) \iint_{xy} u^2 \phi' \leq -\tfrac14  \iint u^2 \phi'$$ 
as well as the first two terms in the first line of \eqref{E:mon14} can be used to absorb error terms.  

By \eqref{E:mon10}, we have, provided we fix $K \geq 4$, 
\begin{equation}
\label{E:mon12}
\tfrac12 \iint u^2 \phi''' \leq \tfrac1{32} \iint u^2 \phi'.
\end{equation}
We will go ahead and fix $K=4$ at the end, since it seems there is no need to take it larger.

Now define the following sets 
$$
B_1 = \{ \, (x,y) \in \mathbb{R}^2 : \, |x-(x(t)-x(t_0))|> R_0  \text{ or } |y-y(t)|>R_0 \, \} \quad \mbox{and}
$$
$$
B_2 = \{ \, (x,y) \in \mathbb{R}^2 : \, |x-(x(t)-x(t_0))|<R_0 \text{ and } |y-y(t)|<R_0 \, \}. \qquad
$$
Then $\mathbb{R}^2 = B_1\cup B_2$, and hence,  
$$
\iint u^4 \phi' = \iint_{(x,y)\in B_1}  u^4\phi' \, dy \, dx + \iint_{(x,y)\in B_2} u^4\phi' \, dy \, dx.
$$
In $B_1$, we apply the weighted Gagliardo-Nirenberg inequality \eqref{E:weightedGN}, for any $R_0>0$, to obtain
\begin{equation}
\label{E:mon11}
\iint u^4 \phi' \leq 
\begin{aligned}[t]
& c\, \| u\|_{L^2_{\substack{|x-x(t)|>R_0 \text{ or }\\ |y-y(t)|>R_0}}}^2 \iint (u^2 + u_x^2 +u_y^2) \, \phi' \, dx \, dy \\
&
+ \iint_{(x,y)\in B_2} u^4\phi' \, dy \, dx. 
\end{aligned}
\end{equation}
Note that the coefficient of the first term is made small using that $\lambda u(\lambda x+x(t),\lambda y+y(t),t)-Q(x,y)$ is small in $L^2$, and $Q$ is small in that region.  This just requires that we use the upper bound on $\lambda$ \emph{and not the lower bound on $\lambda$} and that $\|\epsilon\|_{H^1} \lesssim \sqrt{\alpha}$, and make $\alpha$ smaller than some absolute constant.  The $R_0$ just needs to be taken large  to ensure that $e^{-R_0/\lambda} \leq e^{-R_0/1.1}$ is sufficiently smaller than an absolute constant, so $R_0$ is an absolute constant.

Plug \eqref{E:mon11}, \eqref{E:mon12} into \eqref{E:mon14} to obtain
$$
\partial_t I_{x_0,t_0}(t) \lesssim -\iint (|\nabla u|^2 + |u|^2) \phi'   \, dx\,dy  +  \iint_{(x,y)\in B_2} u^4\phi' \, dx\, dy.
$$
Integrating in time, we have 
$$
I_{x_0,t_0}(t_0) \leq I_{x_0,t_0}(t_{-1}) + \int_{t_{-1}}^{t_0} \iint_{(x,y)\in B_2} u^4\phi' \, dx\, dy 
$$
and
$$
I_{-x_0,t_0}(t_1) \leq I_{-x_0,t_0}(t_0) + \int_{t_0}^{t_1} \iint_{(x,y)\in B_2} u^4\phi' \, dx\, dy. 
$$

For $B_2$, we keep the spatial restriction on the weight $\phi'$, the rest estimating using the (standard) Gagliardo-Nirenberg inequality.  Let $\tilde x =  x- x_0 - \frac12(x(t)-x(t_0))$ for the $[t_{-1},t_0]$ interval and  $\tilde x =  x+x_0 - \frac12(x(t)-x(t_0))$ for the $[t_0,t_1]$ interval. Then
$$
\iint_{(x,y)\in B_2} u^4\phi' \, dy \, dx \lesssim \Big( \sup_{(x,y)\in B_2} \phi'(\tilde x) \Big) \| u\|_{L_{xy}^2}^2 \| \nabla u\|_{L_{xy}^2}^2.
$$
We have $|x-(x(t)-x(t_0))|<R_0$, and thus, writing $\tilde x = \big(x-(x(t)-x(t_0)) \big) - x_0 - \frac12 x(t_0) + \frac12 x(t)$, we see that
$$
|\tilde x| \geq -x_0 + \frac12(x(t)-x(t_0)) - R_0.
$$
If $t<t_0$, then $x(t)<x(t_0)$ and 
$$
|\tilde x| \geq x_0 + \frac12(x(t_0)-x(t)) - R_0.
$$
If $t>t_0$, then $x(t)>x(t_0)$, and replacing $x_0$ by $-x_0$, we have
$$
|\tilde x| \geq x_0 + \frac12(x(t)-x(t_0)) - R_0.
$$
Since
$$
\phi'(\tilde x) \leq  \frac{1}{\pi K} \sech(\tilde x/K) \leq \frac{2}{\pi K} e^{-|\tilde x|/K},
$$ 
we obtain
$$
\phi'(\tilde x) \leq \frac{1}{K} \exp(- \frac{x_0}{K}) \exp (\frac{R_0}{K}) \exp ( - \frac{|x(t)-x(t_0)|}{2K} ),
$$
which is a bound \emph{independent of }$x$.  Plugging in, we obtain
$$
\iint_{(x,y)\in B_2} u^4\phi' \, dy \, dx \lesssim  \| u\|_{L_{xy}^2}^2 \| \nabla u\|_{L_{xy}^2}^2 \frac{1}{K} \exp(- \frac{x_0}{K}) \exp (\frac{R_0}{K}) \exp ( - \frac{|x(t_0)-x(t)|}{2K} ).
$$
Using that $\|u \|_{L^2}^2 \lesssim 1$ and $\|\nabla u(t) \|_{L_{xy}^2}^2 \sim \lambda^{-2} \lesssim x_t$, we obtain
$$
\iint_{(x,y)\in B_2} u^4\phi' \, dy \, dx \lesssim   \frac{1}{K} \exp(- \frac{x_0}{K}) \exp (\frac{R_0}{K}) \, x_t \, \exp ( - \frac{|x(t_0)-x(t)|}{2K} ).
$$
Upon integrating in time, separately considering the cases $t<t_0$ and $t>t_0$, we get
$$
\int_{t_{-1}}^{t_0} \iint_{(x,y)\in B_2} u^4\phi' \, dy \, dx \lesssim \exp(- \frac{x_0}{K}) \exp (\frac{R_0}{K})  \big( 1- \exp( - \frac{x(t_0)-x(t_{-1})}{2K}) \big)
$$
and
$$
\int_{t_0}^{t_1} \iint_{(x,y)\in B_2} u^4\phi' \, dy \, dx \lesssim \exp(- \frac{x_0}{K}) \exp (\frac{R_0}{K}) \big( 1- \exp( - \frac{ x(t_1)-x(t_0)}{2K} )\big).
$$
\end{proof}

Now let us give some details on how to obtain the monotonicity estimate for the rotated function.  Fix an angle $\theta$: $-\frac{\pi}{3}+\delta< \theta< \frac{\pi}{3}-\delta$.   Given $u$, let $u^\theta$ be defined in terms of $u$ by
\begin{equation}
\label{E:rotated-coords}
u(x,y,t) = u^\theta(\bar x, \bar y, t) \,, \qquad \begin{bmatrix} \bar x \\ \bar y \end{bmatrix} = \begin{bmatrix} \cos \theta & - \sin \theta \\ \sin \theta & \cos \theta \end{bmatrix} \begin{bmatrix} x \\ y \end{bmatrix}.
\end{equation}
Note that $u^\theta$ solves
$$
(ZK_\theta) \qquad \qquad 0 = \partial_t u^\theta + (\cos \theta \partial_{\bar x} - \sin \theta \partial_{\bar y})[\Delta u^\theta + (u^\theta)^3]. \qquad \qquad 
$$
Note that rotation does not affect mass or energy, thus,
$$
E(u^\theta)=E(u) ~~ \text{ and } ~~ M(u^\theta)=M(u).
$$
Moreover, note that $u=u^0$ (the case $\theta=0$).

Since 
$$u_n(x+x_n(t_{n,m}+t), y+y_n(t_{n,m}+t), t_{n,m}+t) \rightharpoonup \tilde u(x+\tilde x_n(t), y+\tilde y_n(t), t)$$
it follows that
$$u_n^\theta (\bar x+\bar x_n(t_{n,m}+t), \bar y+\bar y_n(t_{n,m}+t), t_{n,m}+t) \rightharpoonup \tilde u^\theta(\bar x+ \bar{\tilde x}_n(t), y+\bar{\tilde y}_n(t), t)$$
The relation between the soliton center coordinates is
$$
\begin{bmatrix} \bar x(t) \\ \bar y(t) \end{bmatrix} = \begin{bmatrix} \cos \theta & - \sin \theta \\ \sin \theta & \cos \theta \end{bmatrix} \begin{bmatrix} x(t) \\ y(t)\end{bmatrix}.
$$
The trajectory estimates
$$
|x_t - \lambda^{-2}| \lesssim \|\epsilon \|_{L_{xy}^2} \,, \qquad |y_t| \lesssim \| \epsilon \|_{L_{xy}^2},
$$
imply
$$
|\bar x_t - (\cos \theta) \lambda^{-2}| \lesssim \|\epsilon \|_{L_{xy}^2} \,, \qquad |\bar y_t - (\sin \theta) \lambda^{-2} | \lesssim \|\epsilon \|_{L_{xy}^2}.
$$

The following shows that the above monotonicity estimate (Lemma \ref{L:mon1}) generalizes to $u^\theta$ for all $|\theta|\leq \frac{\pi}{3}-\delta $ for any $0<\delta<\frac{\pi}{3}$.  For ease of exposition in the lemma, we drop the bar notation -- that is, we write $(x, y)$ instead of $(\bar x,\bar y)$.

\begin{lemma}[generalized $I^\theta_\pm$ estimates]
\label{L:mon1gen}
Let $t_{-1}<t_0<t_1$ and suppose that $|\theta| \leq \frac{\pi}{3}-\delta$  and $u^\theta(t)$ is an $H^1$ solution to $(\text{ZK}_\theta)$ on $[t_{-1},t_1]$ with $E(u^\theta)<0$ and
$$
\forall \, t\in [t_{-1},t_1] \,, \qquad \| \nabla u^\theta(t) \|_{L_{xy}^2} \geq 0.9 \, \|\nabla Q\|_{L_{xy}^2}.
$$
There exists an absolute constant $\alpha_4>0$ such that if 
$$
\alpha(u^\theta) \defeq \|u^\theta\|_{L^2}^2 - \|Q\|_{L^2}^2  \leq \alpha_4,
$$  
we then we have the following.  Let
$$
I^\theta_{\pm,x_0,t_0}(t) = \iint (u_\theta)(u^\theta)^2(x+x(t_0),y, t) \phi_\pm(x-x_0-\tfrac12(x(t)-x(t_0))) \, dx \, dy,
$$
where $\phi_-(x)=\phi_+(-x)$ and
$$
\phi_+(x) = \frac{2}{\pi} \arctan(e^{x/K})
$$
so that $\phi_+(x)$ is increasing with $\lim_{x\to -\infty} \phi_+(x) = 0$ and $\lim_{x\to +\infty} \phi_+(x) = 1$, and $\phi_-(x)$ is decreasing with $\lim_{x\to -\infty} \phi_-(x) = 1$ and $\lim_{x\to +\infty} \phi_-(x) = 0$.  Let $x_0>0$ and $K \geq 4$.  For the increasing weight, we have two estimates that bound the \emph{future} in terms of the \emph{past}
\begin{equation}
\label{E:mo1b}
\text{for } t_{-1}<t_0\,, \quad I^\theta_{+,x_0,t_0}(t_0) \leq I^\theta_{+,x_0,t_0}(t_{-1}) + \rho_0 e^{-x_0/K},
\end{equation}
\begin{equation}
\label{E:mo2b}
\text{for } t_0<t_1\,, \quad I^\theta_{+,-x_0,t_0}(t_1) \leq I^\theta_{+,-x_0,t_0}(t_0) +\rho_0  e^{-x_0/K}.
\end{equation}
For the decreasing weight, we have two estimates that bound the \emph{past} in terms of the \emph{future}
\begin{equation}
\label{E:mo3b}
\text{for }t_{-1}<t_0\,, \quad I^\theta_{-,x_0,t_0}(t_{-1}) \leq I^\theta_{-,x_0,t_0}(t_0) +\rho_0 e^{-x_0/K},
\end{equation}
\begin{equation}
\label{E:mo4b}
\text{for }t_0<t_1\,, \quad I^\theta_{-,-x_0,t_0}(t_0) \leq I^\theta_{-,-x_0,t_0}(t_1) +\rho_0 e^{-x_0/K}.
\end{equation}
Here, $\, \theta$, $\alpha_4>0$ and $\rho_0>0$ are absolute constants, in particular, independent	of the \emph{upper} bound on $\|\nabla u(t) \|_{L_{xy}^2}$ over the time interval $[t_{-1},t_1]$. 
\end{lemma}
\begin{proof}
The proof follows that of the earlier monotonicity lemma (Lemma \ref{L:mon1}) with minimal modification.    Indeed, if $\phi$ depends only\footnote{Which actually means $\bar x$ in the notation preceding the lemma.} on $x$, then to calculate $\partial_t I^\theta$, we need the following terms
$$
\iint \phi \, u^\theta \, u^\theta_{x x x} \, dx\, dy = \frac32 \iint \phi_{x} \, (u^\theta_{x})^2 \, dx \, d y -\frac12 \iint \phi_{x x x} \, (u^\theta)^2  \, dx \, dy
~~\mbox{and}
$$
$$
\iint \phi \, u^\theta \, u^\theta_{x y y} \, dx \, dy = \frac12 \iint \phi_{x} (u^\theta_{y})^2 \, dx \, d y,
$$
which are the same as before, but now we also have
$$
\iint \phi \, u^\theta \, u^\theta_{y x x} \, dx \, dy = \iint \phi_{x} u^\theta_{y} u^\theta_{x}  \,d x \, d y ~~\mbox{and}
$$
$$
\iint \phi \, u^\theta \, u^\theta_{y y y} \, dx \, d y =0.
$$
We have (we drop the $\theta$ superscript for $u$)
\begin{align*}
\tfrac12\partial_t I^\theta_{x_0,t_0}(t) &= \iint_{xy} u u_t \phi \, dx \, dy - \tfrac12x'(t) \iint_{xy} u^2 \phi' \, dx \, dy \\
& = \cos \theta \iint_{xy} u(-u_{xxx}-u_{yyx}-3u^2u_x) \phi \, dx \, dy  \\
& \qquad - \sin\theta  \iint_{xy} u(-u_{xxy}-u_{yyy}-3u^2u_y) \phi \, dx \, dy -  \tfrac12x'(t) \iint_{xy} u^2 \phi' \, dx \, dy .
\end{align*}
After the indicated applications of integration by parts (recalling that $\phi$ is independent of $y$), we obtain
$$
\tfrac12\partial_t I^\theta_{x_0,t_0}(t) = 
\begin{aligned}[t]
&-\tfrac32\cos\theta \iint u_x^2 \phi' -\tfrac12 \cos\theta \iint u_y^2 \phi' - \tfrac12 x'(t) \cos\theta \iint u^2 \phi'\\
&  + \tfrac12 \cos\theta \iint u^2 \phi''' +\tfrac34 \cos\theta \iint u^4 \phi'  - \sin\theta \iint  u_{y} u_{x} \phi'  \,d x \, d y.
\end{aligned}
$$
We can proceed exactly as in the proof of the earlier monotonicity lemma corresponding to $\theta=0$  (Lemma \ref{L:mon1}) except that now we need to control the extra term 
$$
\left| - \sin\theta \iint  u_{y} u_{x} \phi'  \,d x \, d y \right| \leq (1-\delta)\left(\tfrac32\cos\theta \iint u_x^2 \phi' +\tfrac12 \cos\theta \iint u_y^2 \phi' \right)
$$
for some $\delta>0$.  Recalling the classical inequality $u_xu_y \leq \frac12\omega u_x^2 + \frac12\omega^{-1} u_y^2$ for any $\omega >0$,  we  select $\omega>0$ so that 
$$
\frac12\omega \leq \frac32 \cot \theta (1-\delta) \,, \qquad  \frac12\omega^{-1}\leq \frac12 \cot \theta (1-\delta).
$$
This has a solution only if (multiplying the two equations) $\frac13 \leq (1-\delta)^2 \cot^2\theta$, or equivalently, $|\tan \theta| \leq \sqrt 3 (1-\delta)$.  Thus,  we require $|\theta| < \frac{\pi}{3}$, and hence, setting $\delta = 1-\frac{|\tan \theta|}{\sqrt 3}>0$ and $\omega = \sqrt 3$.  
\end{proof}

\begin{lemma}[applying $I_\pm$ estimates to obtain exponential decay of $\tilde u_n$]
\label{L:mon2}
For $x_0>0$, for all $-t_1(n)< t < t_2(n)$, we have
$$
\| \tilde u_n(x +\tilde x_n(t), y,t) \|_{L^2_{|x|>x_0}L^2_y}^2 \leq 24\, \vartheta_0 \, e^{-x_0/8},
$$
where  $\vartheta_0>0$ is an absolute constant (as in Lemma \ref{L:mon1}).
\end{lemma} 
\begin{proof}
Recall that $n$ is fixed in the proof.  It suffices to prove the claim for any finite length interval $(-t_1,t_2) \subset (-t_1(n),t_2(n))$ (that is, for which $t_1,t_2<\infty$).  

First, we prove the decay on the right.  We will in fact prove the following stronger statement:  For $m$ sufficiently large, for all $-t_1\leq t \leq t_2$,
\begin{equation}
\label{E:mon-30}
\| u_n(x +x_n(t_{n,m}+t) , y+y_n(t_{n,m}+t) ,t_{n,m}+t) \|_{L^2_{x>r_0}L^2_y}^2 \leq 6\, \vartheta_0 e^{-r_0/4}.
\end{equation}

Since by Corollary \ref{C:adjusted-conv}
$$
u_n(\bullet + x_n(t_{n,m}+t), \bullet + y_n(t_{n,m}+t), t_{n,m}+t) \rightharpoonup \tilde u_n(\bullet+\tilde x_n(t), \bullet+\tilde y_n(t), t)
$$
weakly as $m\to \infty$, this will imply that
\begin{equation}
\label{E:right-decay}
\| \tilde u_n(x +\tilde x_n(t) , y ,t) \|_{L^2_{x>x_0}L^2_y}^2 \leq 6 \,\vartheta_0 e^{-x_0/4}.
\end{equation}

Arguing by contradiction, if \eqref{E:mon-30} fails, there exists a subsequence $m'$ and a corresponding sequence of times $t_{m'}$ in $-t_1 \leq t_{m'} \leq t_2$ such that
\begin{equation}
\label{E:mon-31}
\| u_n(x +x_n(t_{n,m'}+t_{m'}) , y ,t_{n,m'}+t_{m'}) \|_{L^2_{x>x_0}L^2_y}^2 \geq 6 \,\vartheta_0 e^{-x_0/4}.
\end{equation}
Passing to another subsequence so that $t_{m''} \to t_*$, and using the uniform continuity\footnote{This can be proved using the local theory estimates in Lemmas \ref{L:linhom} and \ref{L:lininhom} together with the bootstrap assumption of \eqref{E:timeframe}.} of $u_n(t)$ over $[t_{n,m}-t_1,t_{n,m}+t_2]$ for every $m$, and the fact that\footnote{This follows from Lemma \ref{L:epparam}, which gives the bound $|\lambda_n(t)^2 x_n'(t) -1 | \lesssim \alpha(u_n)^{1/2}$, which implies a uniform upper bound on $x_n'(t)$ on $[t_{n,m}-t_1,t_{n,m}+t_2]$ for all $m$.} $x_n(t_{n,m''}+t_{m''}) - x_n(t_{n,m''}+t_*) \to 0$, 
\begin{equation}
\label{E:mon-32}
\| u_n(x +x_n(t_{n,m''}+t_*) , y ,t_{n,m''}+ t_*) \|_{L^2_{x>x_0}L^2_y}^2 \geq 4 \vartheta_0 e^{-x_0/4}.
\end{equation}
Restricting attention to this subsequence, relabeling it as $m$, we can declare that there exists $t_*$ with $-t_1\leq t_* \leq t_2$ such that for $m$ sufficiently large,
$$
\| u_n(x +x_n(t_{n,m}+t_*), y ,t_{n,m}+t_*) \|_{L^2_{x>x_0}L^2_y} ^2 \geq 4 \,\vartheta_0  e^{-x_0/4}.
$$
We apply the $I_-$ estimate with $t_{-1}=0$ and $t_0 = t_{n,m}+t_*$ (where the weight transition occurs on the right of the soliton trajectory).  Now
\begin{align*}
I_{-,x_0,t_0}(t_0) &= \iint u_n(x+x(t_0),y,t_0)^2 \phi_-(x-x_0) \, dx \, dy \\
&= \iint u_n(x+x(t_0),y,t_0)^2\, dx \, dy - \iint u_n(x+x(t_0),y,t_0)^2 (1-\phi_-)(x-x_0) \, dx \, dy.
\end{align*}
Using that $\frac12 \leq (1-\phi_-)(x-x_0)$ for $x>x_0$, we have
\begin{align*}
I_{-,x_0,t_0}(t_0) &\leq M(u_n) - \frac12 \iint_{x>x_0, \; y\in \mathbb{R}} u_n(x+x(t_0),y,t_0)^2 \, dxdy \\
&=M(u_n) - \frac12 \iint_{x>x_0, \; y\in \mathbb{R}} u_n(x+x(t_{n ,m}+t_*),y,t_{n, m}+t_*)^2 \, dxdy, 
\end{align*}
and hence,
\begin{equation}
\label{E:mon17}
I_{-,x_0,t_0}(t_0)  \leq M(u_n) - 2 \vartheta_0  e^{-x_0/4}.
\end{equation}
On the other hand,
\begin{align*}
I_{-,x_0,t_0}(t_{-1}) &= \iint u_n(x+x(t_0),y,t_{-1})^2 \phi_-(x-x_0-\frac12(x(t_{-1})-x(t_0)) \, dx\,dy \\
& = \iint u_n(x+x(t_{n, m}+t_*),y, 0)^2 \phi_-(x-x_0-\frac12(x(0)-x(t_{n ,m}+t_*)) \, dx\,dy .
\end{align*}
Therefore,
$$
I_{-,x_0,t_0}(t_{-1}) = \iint u_n(x,y, 0)^2 \phi_-(x-x_0-\frac12(x(0)+x(t_{n ,m}+t_*)) \, dx\,dy.
$$
Since $u_n(\bullet, \bullet, 0)$ is a fixed function, and $x(t_{n,m}+t_*) \to +\infty$ as $m\to +\infty$, we have
\begin{equation}
\label{E:mon18}
\lim_{m\to \infty} I_{-,x_0,t_0}(t_{-1}) = \iint u_n(x,y,0)^2 \, dx \, dy = M(u_n).
\end{equation}
By Lemma \ref{L:mon1} with $K=4$,
$$
I_{-,x_0,t_0}(t_{-1}) \leq I_{-,x_0,t_0}(t_0) + \vartheta_0 e^{-x_0/4}.
$$
By \eqref{E:mon17}, \eqref{E:mon18} and sending $m\to \infty$, we get
$$
M(u_n) \leq M(u_n)  - \vartheta_0 e^{-x_0/4}
$$
a contradiction.  This completes the proof of \eqref{E:mon-30}.
\smallskip

Following the proof of the inequality \eqref{E:mon-30} above, we can apply Lemma \ref{L:mon1gen} to obtain, for $-\frac{\pi}{3}+\delta < \theta< \frac{\pi}{3} -\delta$, 
\begin{equation}
\label{E:mon-30-gen}
\| u_n^\theta(\bar x +\bar x_n(t_{n,m}+t) , \bar y + \bar y_n(t_{n,m}+t),t_{n,m}+t) \|_{L^2_{\bar x>r_0}L^2_{\bar y}}^2 \leq 6\, \vartheta_0 e^{-r_0/4}.
\end{equation}
The estimate \eqref{E:mon-30-gen} is thus a generalization of \eqref{E:mon-30} that reduces to \eqref{E:mon-30} when $\theta=0$.  
\smallskip

Applying \eqref{E:mon-30-gen} for $\theta=\frac{\pi}{4}$, and $\theta=-\frac{\pi}{4}$, we obtain 
\begin{equation}
\label{E:mon-30-gen-1}
\| u_n( x +x_n(t_{n,m}+t) , y+y_n(t_{n,m}+t) ,t_{n,m}+t) \|_{L^2_{\Omega_0\cup\Omega_-\cup \Omega_+}}^2 \leq 6\, \vartheta_0 e^{-r_0/4},
\end{equation}
where $\Omega_0$, $\Omega_\pm$ are as depicted in Figures \ref{F:1}-\ref{F:2}.

\begin{figure}[ht]
\includegraphics[width=0.6\textwidth,height=0.5\textwidth]
{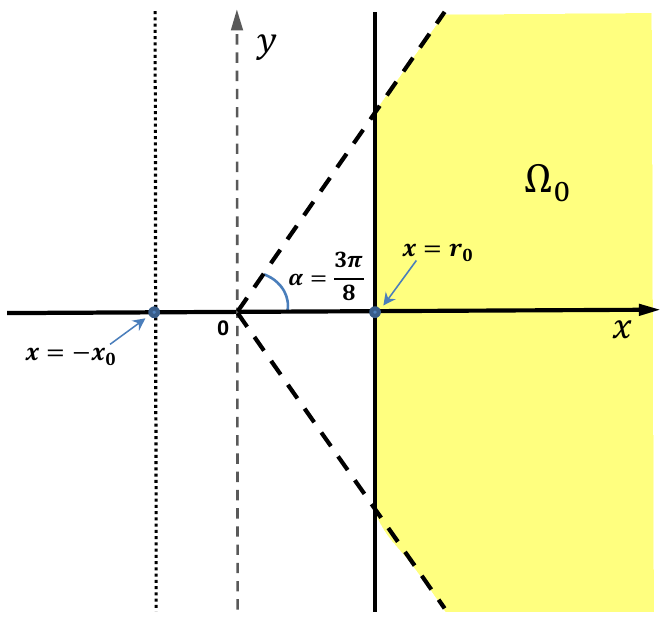}
\caption{ {\footnotesize Region $\Omega_0$, see \eqref{E:mon-30-gen-1}, the set of all $(x,y)$ for which $x>r_0$ and $|\alpha|<\frac{3\pi}{8}$ when polar coordinates $(x,y)=(r\cos\alpha, r\sin\alpha)$ are used.  }}
\label{F:1} 
\end{figure}

\begin{figure}[ht]
\includegraphics[scale=0.72]{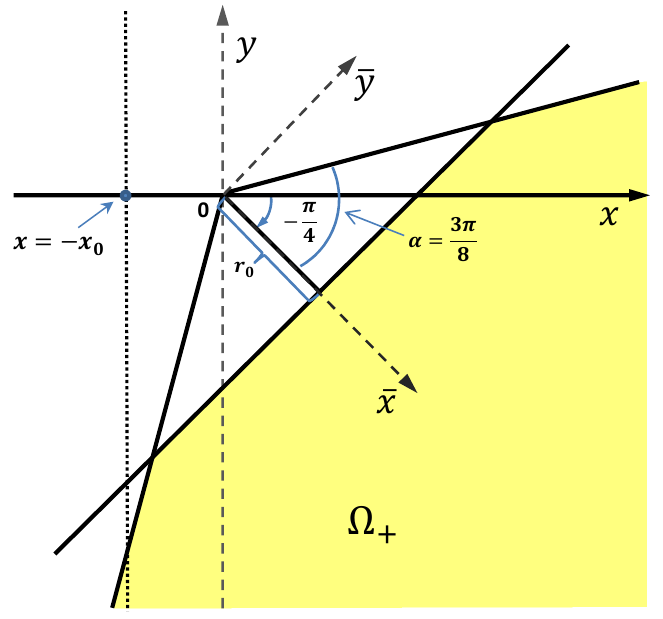} 
\includegraphics[scale=0.72]{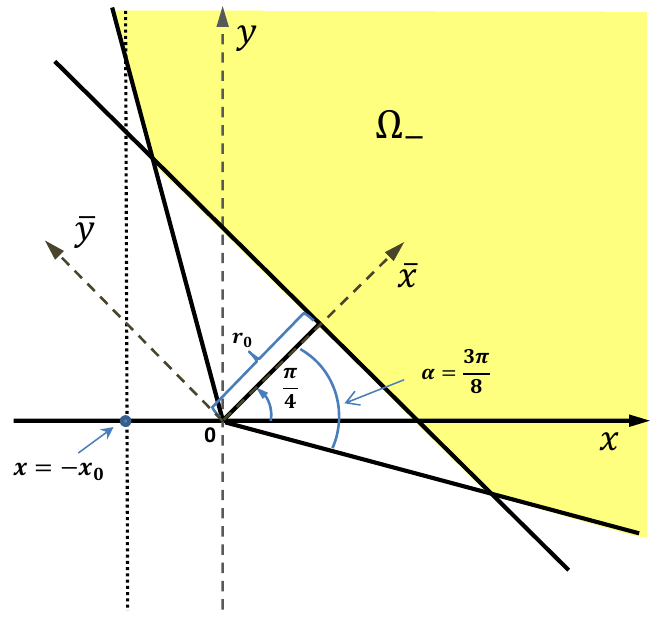} 
\caption{\footnotesize Regions $\Omega_\pm$, see \eqref{E:mon-30-gen-1}, defined analogously to $\Omega^0$ but in the $(\bar x^\theta, \bar y^\theta)$ coordinates introduced in \eqref{E:rotated-coords} for $\theta=\pm\frac{\pi}{4}$ respectively.  
Specifically $\Omega_\theta$ is the set of all $(\bar x^\theta,\bar y^\theta)$, for which $\bar x^\theta>r_0$ and $|\alpha|<\frac{3\pi}{8}$ when polar coordinates $(\bar x^\theta ,\bar y^\theta)=(r\cos\alpha, r\sin\alpha)$ are used.   }
\label{F:2}
\end{figure}

\newpage
Note that \eqref{E:mon-30-gen-1} yields 
\begin{equation}
\label{E:mon-30-gen-2}
\| u_n( x +x_n(t_{n,m}+t) , y+y_n(t_{n,m}+t) ,t_{n,m}+t) \|_{L^2_{\Omega}}^2 \leq 6\, \vartheta_0 e^{-r_0/4},
\end{equation}
where $\Omega$ is as depicted in Figure \ref{F:3}.

\begin{figure}[ht]
\includegraphics[width=0.57\textwidth,height=0.46\textwidth]
{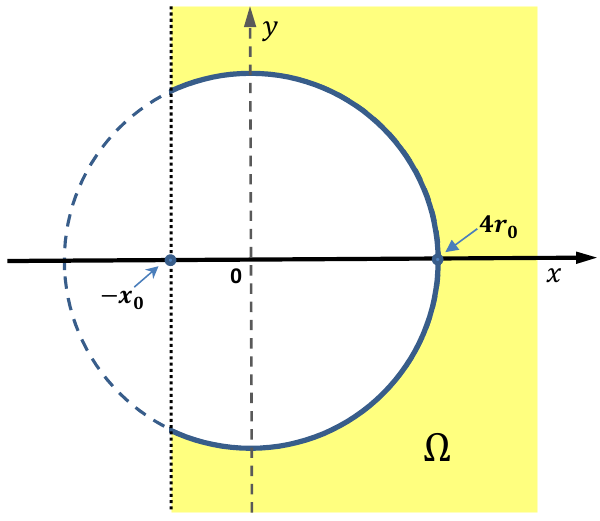} 
\caption{\small Region $\Omega$, see \eqref{E:mon-30-gen-2}, which is the set of all $(x,y)$ such that $x^2+y^2>16r_0^2$ and $x>-x_0$.}
\label{F:3}
\end{figure}

Before proving the decay on the left, let us note a consequence of \eqref{E:mon-30-gen-2}.  Since 
$$
u_n(x+x_n(t_{n,m}+t), y+y_n(t_{n,m}+t), t_{n,m}+t) \rightharpoonup \tilde u_n(x+\tilde x_n(t), y+\tilde y_n(t), t)
$$
weakly in $H^1$, by Rellich-Kondrachov compactness and a diagonal argument, there exists a subsequence (still denoted with $m$), so that
\begin{align*}
& \mathbf{1}_{x>-x_0}(x,y) \; u_n(x+x_n(t_{n,m}+t), y+y_n(t_{n,m}+t), t_{n,m}+t) \\
& \qquad \to \mathbf{1}_{x>-x_0}(x,y) \; \tilde u_n(x+\tilde x_n(t), y+\tilde y_n(t), t) \qquad \text{ as }m \to \infty
\end{align*}
in $L^2_{xy}$. Hence, for $m$ sufficiently large, 
\begin{equation}
\label{E:mon-35}
\begin{aligned}
\indentalign \| u_n(x+x_n(t_{n,m}+t), y+y_n(t_{n,m}+t), t_{n,m}+t) \|^2_{L^2_{x>-x_0}L^2_y} \\
&\leq \| \tilde u_n(x+\tilde x_n(t), y+\tilde y_n(t), t)\|_{L^2_{x>-x_0}L^2_y}^2 + \vartheta_0 e^{-x_0/4}.
\end{aligned}
\end{equation}
It follows that
\begin{equation}
\label{E:mon-37}
\begin{aligned}
\indentalign \| u_n(x+x_n(t_{n,m}+t), y+y_n(t_{n,m}+t), t_{n,m}+t) \|_{L^2_{x>-x_0}L^2_{y}}^2 \\
&\leq  M(\tilde u_n) - \| \tilde u_n(x+\tilde x_n(t), y+\tilde y_n(t), t)\|_{L^2_{x<-x_0} L^2_y}^2 + \vartheta_0 e^{-x_0/4}.
\end{aligned}
\end{equation}

Next, we prove the decay on the left\footnote{Note that the bound is $e^{-x_0/8}$ instead of $e^{-x_0/4}$ as provided by Lemma \ref{L:mon1} with $K=4$.  This is because we use a shift $x_0/2$ instead of $x_0$ -- see below.}, i.e., for all $-t_1\leq t \leq t_2$, 
\begin{equation}
\label{E:left-decay}
\| \tilde u_n(x +\tilde x_n(t), y,t) \|_{L^2_{x<-x_0}L^2_y}^2 \leq  24 \, \vartheta_0 \, e^{-x_0/8}.
\end{equation}
Arguing by contradiction, we assume that there exists $t_*>0$ and such that 
\begin{equation}\label{E:decay3}
\| \tilde u_n(x +\tilde x_n(t_*), y+ \tilde y_n(t_*),t_*) \|_{L^2_{x<-x_0}L^2_y}^2 \geq  22 \, \vartheta_0 \, e^{-x_0/8}.
\end{equation}
For $m$ sufficiently large, we  combine \eqref{E:decay3} with \eqref{E:mon-37}, yielding
\begin{align}
\label{E:mon-38}
\| u_n(x+x_n(t_{n,m}+t_*), y+y_n(t_{n,m}+t_*), t_{n,m}+t_*) & \|_{L^2_{x>-x_0} L^2_y}^2 \\
\leq  M(\tilde u_n)  -21\, \vartheta_0 e^{-x_0/8}. \notag
\end{align}

We apply the $I_+$ estimate for $K=4$ with weight transition on the left, from $t_0=t_{n,m(x_0)}+t_*$ to $t_1 = t_{n,m}$ for $m$ sufficiently large so that $t_{n,m} \geq t_{n,m(x_0)}+t_*$.  

We have
\begin{align*}
I_{+,-x_0/2,t_0}(t_0) &= \iint u_n(x+x_n(t_0),y, t_0)^2 \phi_+(x+\frac12 x_0) \, dx \, dy \\
& = \iint u_n(x+x_n(t_{n,m(x_0)}+t_*),y, t_{n,m(x_0)}+t_*)^2 \phi_+(x+\frac12 x_0) \, dx \, dy .
\end{align*}
Since $\phi_+(x+\frac12 x_0) \leq \phi_+(-\frac12 x_0)$ for $x<-x_0$,
$$
I_{+,-x_0/2,t_0}(t_0) \leq \phi_+(-\frac12x_0)M(u_n) + \iint_{x>-x_0} u_n(x+x_n(t_{n,m(x_0)}+t_*),y, t_{n,m(x_0)}+t_*)^2  \, dx \, dy.
$$
By  \eqref{E:mon-38},
\begin{equation}
\label{E:mon20}
I_{+,-x_0/2,t_0}(t_0) \leq M(\tilde u_n) - 15\vartheta_0  e^{-x_0/8}+\phi_+(-\frac12x_0)M(u_n).
\end{equation}
Since $\phi_+(x) = \frac2{\pi}\arctan(e^{x/K})$, we have, for $x\to -\infty$, $\phi_+(x) \leq e^{x/K}$, and hence\footnote{This is why it is $e^{-x_0/8}$ instead of $e^{-x_0/4}$.},
$$
\phi_+(-\frac12x_0) \leq e^{-x_0/2K} = e^{-x_0/8}.
$$
Moreover, we can assume, without loss, that $\vartheta_0>0$ in Lemma \ref{L:mon1} was taken large enough so that $\vartheta_0 \geq 2\|Q\|_{L^2}^2$.  Then from \eqref{E:mon20}, we obtain
\begin{equation}
\label{E:mon20b}
I_{+,-x_0/2,t_0}(t_0) \leq M(\tilde u_n) - 12\vartheta_0  e^{-x_0/8}.
\end{equation}
On the other hand,
\begin{align*}
I_{+,-x_0/2,t_0}(t_1) &= \iint u_n(x+x_n(t_0),y,t_1)^2 \phi_+(x+\frac12 x_0-\frac12(x_n(t_1)-x(t_0))) \, dx \, dy \\
&= \iint u_n(x+x_n(t_1),y+y_n(t_1),t_1)^2 \\
& \qquad \qquad \phi_+(x+\frac12x_0+\frac12(x_n(t_1)-x_n(t_0))) \, dx \, dy \\
& = \iint u_n(x+x_n(t_{n,m}),y + y_n(t_{n,m}) ,t_{n,m})^2 \\
& \qquad \qquad \phi_+(x+\frac12x_0+\frac12(x_n(t_{n, m}) - x(t_{n,m(x_0)}+t_*))) \, dx \, dy.
\end{align*}
Note that $\phi_+(x+\frac12x_0+\frac12(x_n(t_{n, m}) - x_n(t_{n,m(x_0)}+t_*))) \to 1$ pointwise as $m\to \infty$.  Since 
$$
u_n(x+x(t_{n,m}), y+y(t_{n,m}), t_{n,m}) \rightharpoonup \tilde u_n(x+\tilde x_n(0), y+\tilde y_n(0), 0)
$$
as $m\to \infty$, it follows that also
$$
u_n(x+x_n(t_{n,m}), y+y_n(t_{n,m}), t_{n,m})\phi_+(x+\frac12x_0+\frac12(x_n(t_{n, m}) - x_n(t_{n,m(x_0)}+t_*)))^{1/2} 
$$
$$
\rightharpoonup \tilde u_n(x+\tilde x_n(0), y+\tilde y_n(0), 0)
$$
as $m\to \infty$ (consider that the $\phi_+$ term times the test function converges to the test function).  By the fact that the norm of the weak limit is less than or equal to the limit of the norms,
\begin{equation}
\label{E:mon21}
M(\tilde u_n) = \| \tilde u_n(x+\tilde x_n(0), y+\tilde y_n(0), 0)\|_{L_x^2}^2 \leq \lim_{m\to \infty} I_{+,-x_0/2,t_0}(t_1).
\end{equation}

The $I_+$ estimate from Lemma \ref{L:mon1} states
$$
I_{+,-x_0/2,t_0}(t_1) \leq I_{+,-x_0/2,t_0}(t_0) + \vartheta_0  e^{-x_0/4}.
$$
By \eqref{E:mon20b}, \eqref{E:mon21}, taking $m\to \infty$, we obtain
$$
M(\tilde u_n) \leq M(\tilde u_n) - 11\vartheta_0  e^{-x_0/8},
$$
which is a contradiction. This completes the proof of \eqref{E:left-decay}.  Combining \eqref{E:right-decay} and \eqref{E:left-decay}  completes the proof of lemma.
\end{proof}

\begin{lemma}
\label{L:all-directions}
There exists $\omega>0$ such that 
\begin{equation}
\label{E:improved-decay}
\|  \tilde u_n(x+\tilde x_n(t),y+\tilde y_n(t), t) \|_{L^2_{B(0,r)^c}} \lesssim e^{-\omega r},
\end{equation}
where $B(0,r)$ is the ball with center $0$ and radius $r>0$ in $\mathbb{R}^2$, and $B(0,r)^c$ denotes the complement.   
\end{lemma}
\begin{proof}
Recall the rotated solutions $u_n^\theta$ and corresponding coordinates $(\bar x, \bar y)$.  The proof of Lemma \ref{L:mon2} carries over for $\tilde u_n^\theta$ provided $|\theta|< \frac{\pi}{3}-\frac{\pi}{12} = \frac{\pi}{12}$, and the statement of Lemma \ref{L:mon2} is the case $\theta=0$. 
The combination of the cases $\theta=0$ and $\theta=\frac{\pi}{24}$, that is,
\begin{equation}
\label{E:tilde-decay}
\| \tilde u_n^\theta (\bar x +\bar{\tilde x}_n(t), \bar y+ \bar{\tilde y}_n(t),t) \|_{L^2_{|\bar x|>x_0}L^2_{\bar y}}^2 \leq 24\, \vartheta_0 \, e^{-x_0/8},
\end{equation}
for both $\theta=0$ and $\theta=\frac{\pi}{24}$ gives \eqref{E:improved-decay},  see Figure \ref{F:4}.
\begin{figure}[ht!]
\includegraphics[scale=0.8]{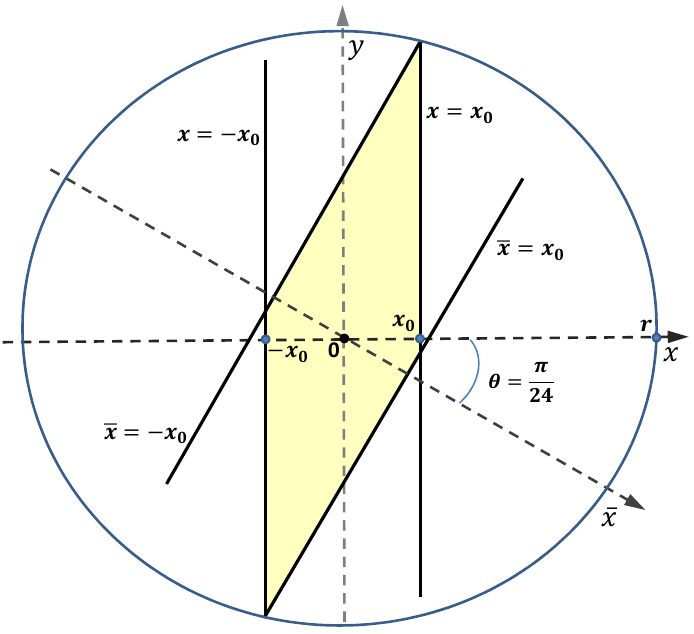}
\caption{\footnotesize Estimate \eqref{E:tilde-decay} for $\theta=0$ gives  a bound by $e^{-x_0/4}$ outside $|x|<x_0$, while estimate \eqref{E:tilde-decay} for $\theta=\frac{\pi}{24}$ gives a bound of $e^{-x_0/4}$ outside $|\bar x|<x_0$, where $\bar x = x \cos \theta  - y\sin \theta$.  The combination gives decay $e^{-x_0/4}$ outside radius $r=x_0/\sin(|\theta|/2)$.
\label{F:4}}
\end{figure}
\end{proof}

Let
\begin{equation}\label{E:epsilon-tilde}
\tilde \epsilon_n(x,y,t) \defeq \tilde \lambda_n(t) \, \tilde u_n(\tilde\lambda_n(t) x + \tilde x_n(t), \tilde \lambda_n(t) y + \tilde y_n(t), t)-Q(x,y).
\end{equation}
\begin{lemma}[pointwise-in-$x$ estimates of $\tilde u_n$ and $\tilde \epsilon_n$]
\label{L:pointwise-in-x}
For $n$ sufficiently large, we have, uniformly in $t$, 
$$
\| \tilde u_n(x+\tilde x_n(t),y,t) \|_{L_y^2} \lesssim \tilde \lambda_n(t)^{-1/2}  e^{-|x|/32}
$$
and
$$
\| \tilde \epsilon_n (x,y,t) \|_{L_y^2} \lesssim  \alpha(\tilde u_n)^{1/4} e^{-\tilde \lambda_n(t) |x|/32}.
$$
\end{lemma}
\begin{proof}
For the proof let us instead write $u$ for $\tilde u_n$, $x(t)$ for $\tilde x_n(t)$, $\lambda(t)$ for $\tilde \lambda_n(t)$, and similarly, we just write $\epsilon$ in place of $\tilde \epsilon_n$.  We have, for $x_0>0$,
$$
\| u(x_0+x(t),y,t) \|_{L_y^2}^2 \lesssim \left\| \, \|u(x+x(t),y,t)\|_{L_{x>x_0}^2} \|u_x(x,y,t) \|_{L_x^2} \right\|_{{L_y^1}}.
$$
By Cauchy-Schwarz in $y$ on the outside, we get
$$
\| u(x_0+x(t),y,t) \|_{L_y^2}^2\lesssim \|u(x+x(t),y,t)\|_{L_y^2L_{x>x_0}^2} \|u_x(x,y,t) \|_{L_y^2L_x^2}.
$$
By Lemma \ref{L:mon2}, 
$$
\| u(x_0+x(t),y,t) \|_{L_y^2}^2 \lesssim \lambda(t)^{-1} e^{-x_0/16},
$$
where we have used that $\lambda(t)^{-1} \sim \| \nabla u(t) \|_{L^2}$ from \eqref{E:gradient-lambda1}.  A similar argument works for $x_0<0$.  

Similar to the above, for $x_0\in \mathbb{R}$, we have
\begin{equation}
\label{E:ep-est1}
\| \epsilon(x_0,y,t) \|_{L_y^2}^2 \lesssim \| \epsilon(x,y,t) \|_{L_{|x|>|x_0|}^2 L_y^2} \| \epsilon_x(x,y,t) \|_{L_{xy}^2}.
\end{equation}
The second term is bounded by $\alpha(u)^{1/2}$.  For the first term, we use the definition of $\epsilon$ in terms of $u$ and $Q$ to estimate
$$
\| \epsilon(x,y,t) \|_{L^2_{|x|>|x_0|}L^2_y} \lesssim \|u(x+x(t),y,t) \|_{L^2_{|x|>\lambda |x_0|}L^2_y} + \| Q(x,y) \|_{L^2_{|x|>|x_0|}L^2_y}.
$$
By Lemma \ref{L:mon2}, we get
$$
\| \epsilon(x,y,t) \|_{L^2_{|x|>|x_0|}L^2_y}\lesssim  e^{-\lambda |x_0|/16} + e^{-|x_0|} \lesssim  e^{-\lambda |x_0|/16}.
$$
Plug this into \eqref{E:ep-est1}, to obtain
$$
\| \epsilon(x,y,t) \|_{L^2_y}^2 \lesssim \alpha(u)^{1/2} e^{-\lambda |x|/16}.
$$
\end{proof}

\section{Control of $\tilde \lambda_n(t)$ via the $L^1$-type invariance}

From \eqref{E:upper} and \eqref{E:gradient-lambda1}, we have $\lambda(t) \leq 1.1$ for all $t\geq 0$.

\begin{lemma}[integral conservation yields control on scale]
\label{L:scale-control}
For any solution $u(t)$ with $\alpha(u) \defeq M(u)-M(Q)$ and $E(u) < 0$, let 
$$\epsilon(x,y,t) = \lambda(t) u(\lambda(t) x+x(t), \lambda(t) y + y(t), t) - Q(x,y)$$
with parameters $\lambda(t)$, $x(t)$, $y(t)$ as given by Lemma \ref{L:geom} and $\| \epsilon(t) \|_{H_{xy}^1} \lesssim \alpha(u)^{1/2}$ with $\alpha(u) \ll 1$.
Suppose that 
$$
0.9 \leq \lambda(0) \leq 1.1
$$
and for all $-T_*^- <t<T_*^+$,  we have both 
$$
0<\lambda(t) \leq 1.1
$$
and the $x$-pointwise estimate uniformly in $t$
\begin{equation}
\label{E:assume-ep-decay}
\| \epsilon(x,y,t) \|_{L^2_y} \lesssim \alpha(u)^{1/4} e^{-\lambda |x|/32}.
\end{equation}
Let $(-T^-,T^+)$ be the maximal time interval around $0$ contained in $(-T_*^-,T_*^+)$ such that for all  $-T^- < t < T^+$, we have 
$$
\lambda(t) \geq \frac34.
$$  
Then there exists an absolute $\alpha_5>0$ such that for $\alpha(u) \leq \alpha_5$, we have $(-T_*^-,T_*^+) = (-T^-,T^+)$.
\end{lemma}

\begin{proof}
First we note that by integrating \eqref{E:assume-ep-decay} in $x$, for each $t$ such that $-T^- < t < T^+$, we have
\begin{equation}
\label{E:epL1}
\| \epsilon(t,x,y) \|_{L_y^2L_x^1} \leq \| \epsilon(t,x,y) \|_{L_x^1L_y^2} \lesssim \alpha(u)^{1/4}. 
\end{equation}
Let
$$
F(t) \defeq \left\| \int_x (Q(x,y) + \epsilon(x,y,t)) \, dx \right\|_{L^2_y}^2 - \left\| \int_x Q(x,y) \, dx \right\|_{L^2_y}^2.
$$
By  expanding the square and using Cauchy-Schwarz (in $y$), we get
\begin{equation}
\label{E:scale1}
F(t) \leq 2\|Q\|_{L_y^2L_x^1} \|\epsilon(t)\|_{L_y^2L_x^1} + \|\epsilon(t)\|_{L_y^2L_x^1}^2 \lesssim \alpha(u)^{1/4}.
\end{equation}
Substituting the definition of $\epsilon$, we obtain
$$ 
F(t) = \left\| \int_x \lambda(t) u(\lambda(t)x+x(t),\lambda(t)y+y(t),t) \, dx \right\|_{L^2_y}^2 - \left\| \int_x Q(x,y) \, dx \right\|_{L^2_y}^2.
$$
Scaling and translating in $x$ and $y$, we write
\begin{equation}
\label{E:scale2}
F(t) = \lambda(t)^{-1} \left\| \int_x u(x,y,t) \, dx \right\|_{L^2_y}^2 - \left\| \int_x Q(x,y) \, dx \right\|_{L^2_y}^2.
\end{equation}
Recalling that
$$ 
\left\| \int_x u(x,y,t) \, dx \right\|_{L^2_y} =  \left\| \int_x u(x,y,0) \, dx \right\|_{L^2_y},
$$
and hence,
\begin{equation}
\label{E:scale3}
F(t) - F(0) = (\lambda(t)^{-1} - \lambda(0)^{-1}) \left\| \int_x u(x,y,t) \, dx \right\|_{L^2_y}  \left\| \int_x u(x,y,0) \, dx \right\|_{L^2_y}.
\end{equation}
Solving \eqref{E:scale2} for $\left\| \int_x u(x,y,t) \, dx \right\|_{L^2_y}$, we obtain
$$
\left\| \int_x u(x,y,t) \, dx \right\|_{L^2_y} = \lambda(t)^{1/2}( \|Q\|_{L_y^2L_x^1}^2 + F(t))^{1/2}.
$$
Substituting this equation at time $t$ and at time $0$ into \eqref{E:scale3}, we obtain
$$
F(t) - F(0) = (\lambda(t)^{-1} - \lambda(0)^{-1})\lambda(t)^{1/2} \lambda(0)^{1/2} ( \|Q\|_{L_y^2L_x^1}^2 + F(0))^{1/2} ( \|Q\|_{L_y^2L_x^1}^2 + F(t))^{1/2}.
$$
By \eqref{E:scale1}
$$
\left| \left( \frac{\lambda(t)}{\lambda(0)} \right)^{1/2} -  \left( \frac{\lambda(0)}{\lambda(t)} \right)^{1/2} \right| \lesssim \alpha(u)^{1/4}.
$$
Thus, provided $\alpha(u)>0$ is sufficiently small, then $\alpha(u)^{1/4}$ is also sufficiently small, and it follows that on $(-T_-,T_+)$, we have $\lambda(t) \geq \frac78> \frac34$.  By continuity, since $(-T_-,T_+)$ is maximal within $(-T_-^*,T_+^*)$, it follows that $(-T_-,T_+)=(-T_-^*,T_+^*)$ as claimed.
\end{proof}

\section{Completion of part (1) of the proof of Proposition \ref{P:reduction}}
For $n$ sufficiently large, $\alpha(u_n)\leq \alpha_5$, and thus, Lemmas \ref{L:pointwise-in-x} and \ref{L:scale-control} apply.  Recall that the bootstrap time frame $(-t_1(n),t_2(n))$ is defined by \eqref{E:timeframe}, and also recall that Lemma \ref{L:conv-params} applies yielding the convergence of the parameters, in particular, that $\lambda_{n,m}(t) \to \tilde \lambda_n(t)$ as $m\to \infty$.  By Lemma \ref{L:scale-control}, \eqref{E:timeframe} is reinforced so that in fact 
$$
\tfrac34 \leq \tilde \lambda_n(t) \leq \tfrac54
$$
on $(-t_1(n),t_2(n))$.  By Lemma \ref{L:conv-params}, 
$$
\tfrac34 \leq \tilde \lambda_n(t) =\liminf_{m\to \infty} \lambda_{n, m}(t) \leq \limsup_{m\to \infty} \lambda_{n,m}(t) =\tilde \lambda_n(t) \leq \tfrac54.
$$
By the method of proof of Lemma \ref{L:nontrivial}, if either $t_1(n)<\infty$ or $t_2(n)<\infty$, then a contradiction to the maximality in the definition of $(-t_1(n),t_2(n))$ is achieved, so we must have $t_1(n)=t_2(n)=\infty$ as claimed.  
 
\section{Spatial localization with sharp coefficient,\\ completion of part (2) of proof of Proposition \ref{P:reduction}}\label{S-10}

In this section, we substantially strengthen the decay estimate on $\tilde u_n$ by proving a monotonicity	estimate directly on (a scaled version of) $\tilde \epsilon_n$.  We state the result for general $u$ and $\epsilon$, although we will invoke it for $\tilde u_n$ and $\tilde \epsilon_n$ later in the compactness argument.

\begin{lemma}
\label{L:ep-strong-decay}
Suppose that $u(t)$ solves the ZK equation \eqref{ZK}, $\alpha(u) \ll 1$ and $E(u)<0$ (so that the geometrical decomposition applies), and furthermore, $\frac12 \leq \lambda(t) \leq 2$ for all $t\in \mathbb{R}$.  Assume, moreover, that $u(t)$ satisfies a weak decay estimate 
\begin{equation}
\label{E:etamon3}
\lim_{t\to \pm \infty} \| u(x+x(t),y, t) \|_{L^2_y L^2_{|x| \geq \frac14 |t|}} = 0.
\end{equation}
Then  for all $t\in \mathbb{R}$ and each $x_0>0$,
\begin{equation}
\label{E:betterdecay}
\| \epsilon(x,y,t) \|_{L^2_y L^2_{|x|\geq x_0}} \lesssim  \| \epsilon\|_{L_t^\infty L_{xy}^2} e^{-x_0/8}.
\end{equation}
\end{lemma}

Note that the fact that the coefficient to the decay on the right side of \eqref{E:betterdecay} is $\|\epsilon \|_{L_t^\infty L_{xy}^2}$, as opposed to $\|\epsilon \|_{L_t^\infty L_{xy}^2}^\gamma$ for some $\gamma<1$, is crucial for the compactness argument that follows.
\smallskip

Before starting the proof,  we note that we have previously proved a monotonicity estimate on $u$ using the functional
$$
I_{\pm,x_0,t_0}(t) = \iint u^2(x+x(t_0),y+y(t_0),t)\,
\phi_\pm(x-x_0-\frac12(x(t)-x(t_0))) \, dx \, dy.
$$
By a change of variable, this is equivalent to an estimate on
$$
I_{\pm,x_0,t_0}(t) = \iint u^2(x+x(t),y+y(t),t)
\, \phi_\pm(x-x_0+\frac12(x(t)-x(t_0))) \, dx \, dy,
$$
where now the soliton is centered at $0$ for all times.  

\begin{proof}
 Define
$$
\eta(x,y,t) = \lambda^{-1} \epsilon(\lambda^{-1}x , \lambda^{-1}y, t),
$$
so that
$$
\eta(x,y,t) = u(x+x(t),y+y(t),t) -
\lambda^{-1}Q(\lambda^{-1}x,\lambda^{-1}y).
$$
Let $\tilde Q(x,y) = \lambda^{-1}Q(\lambda^{-1}x,\lambda^{-1}y)$.
Then we find that $\eta$ solves
$$
0 = \partial_t \eta - (x_t,y_t)\cdot \nabla \eta + \partial_x(\Delta \eta
+(\eta+\tilde Q)^3-\tilde Q^3)
$$
$$
+(\lambda^{-1})_t \partial_{\lambda^{-1}} \tilde Q - (x_t - \lambda^{-2},y_t)
\cdot \nabla \tilde Q.  
$$
Now we define the functional
$$
J_{\pm,x_0,t_0}(t) = \iint  \phi_\pm(x-x_0+\frac12(x(t)-x(t_0)))\, \eta^2(x,y,t) \, dx
\, dy.
$$

As before, for the increasing weight $\phi_+$, we will prove an estimate of the future in terms of the past, and for the decreasing weight $\phi_-$, we will prove an estimate of the past in terms of the future.  However, the difference is that this time, we need $\phi_\pm$ to be small near the origin, so we can only do $\phi_+$ estimates on the right and $\phi_-$ estimates on the left.  See Figure \ref{F:new-monotonicity} for the new configuration, where $x$-space has been shifted so that the soliton is positioned at the origin.

\begin{figure}[ht]
\includegraphics
[scale=0.6]
{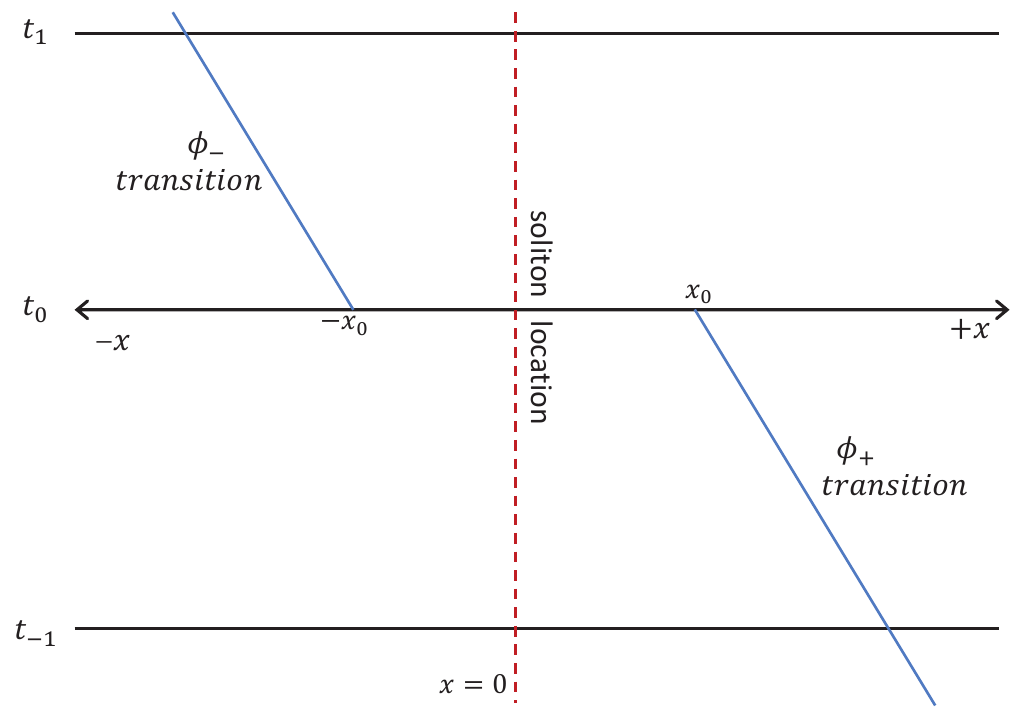}
\label{F:new-monotonicity}
\caption{\footnotesize The frame of reference in the proof of Lemma \ref{L:ep-strong-decay}, where the soliton is at position $x=0$, the $\phi_-$ transition occurs in the $x<0$ region and is only used when $t>t_0$, the $\phi_+$ transition occurs in the $x>0$ region and is only used for $t<t_0$.}
\end{figure}

Thus, we have, for some absolute constant $C>0$, 
\begin{equation}
\label{E:etamon1}
J_{-,-x_0,t_0}(t_0) \leq J_{-,-x_0,t_0}(t_1) + C e^{-x_0} \| \eta \|_{L_t^\infty L_{xy}^2}^2,
\end{equation}
\begin{equation}
\label{E:etamon2}
J_{+,x_0,t_0}(t_0) \leq J_{+,x_0,t_0}(t_{-1}) + C e^{-x_0} \| \eta \|_{L_t^\infty L_{xy}^2}^2.
\end{equation}
By reflection symmetry, it suffices to prove the second one, so we take $\phi= \phi_+$ from here
$$
J_{+,x_0,t_0}'(t) =  2\iint \eta \eta_t \phi_+ \, dx \, dy + \frac12 x_t \iint \eta^2 \phi_+'. 
$$
By several applications of integration by parts, we obtain
$$
J_{+,x_0,t_0}'(t)  = 
\begin{aligned}[t]
&- \frac12x_t \iint \phi_x \eta^2 - 3 \iint \phi_x \eta_x^2 - \iint \phi_x \eta_y^2 + \iint \phi_{xxx} \eta^2 \\
&+ \iint (3\phi_x \tilde Q^2 - 6 \phi \tilde Q \tilde Q_x) \eta^2 + \iint (4\phi_x \tilde Q - 2 \phi \tilde Q_x) \eta^3  + \frac32 \int \phi_x \eta^4\\
&  + 2\lambda^{-2} \lambda_t \iint \phi \partial_{\lambda^{-1}} \tilde Q \eta  + 2(x_t - \lambda^{-2}) \iint \phi  \tilde Q_x \eta + 2y_t \iint \phi   \tilde Q_y \eta.
\end{aligned}
$$
In the first line, the first three terms all have the good sign, and the last term, $\iint \phi_{xxx} \eta^2$ is smaller than the first by taking $K \geq 4$, as before.  In the second line, the last term $\iint \phi_x \eta^4$ is controlled by the weighted Gagliardo-Nirenberg estimate (without the need for spatial cutoff), as was done in the earlier monotonicity result. 

By considering the effective support properties of $\phi$ and $\tilde Q$, and using that $\frac12 \leq \lambda \leq 2$, we have
$$
\| 3 \phi_x \tilde Q^2 - 6\phi \tilde Q \tilde Q_x\|_{L^\infty}+  \|4\phi_x \tilde Q - 2 \phi \tilde Q_x \|_{L^\infty}  \lesssim e^{-x_0} e^{\frac12(x(t)-x(t_0))},
$$
and thus, the first two terms on the second line can be handled by ``suping out''
the weight and, in the case of the middle term, following up with the
Gagliardo-Nirenberg estimate $\iint \eta^3 \lesssim \|\eta\|_{L^2}^2 \|\nabla
\eta\|_{L^2}$.   For the three terms in the last line, use
$$ 
\| \phi \partial_{\lambda^{-1}} \tilde Q\|_{L^2} + \| \phi \tilde Q_x\|_{L^2} + \| \phi \tilde Q_y \|_{L^2}  \lesssim e^{-x_0} e^{ \frac12(x(t)-x(t_0))}
$$
and also the parameter bounds
$$
|\lambda_t | + |x_t-\lambda^{-2}| + |y_t| \lesssim  \| \eta\|_{L^2}.
$$
Thus, we have
$$
| J'_{+,x_0,t_0}(t)| \lesssim e^{-x_0} e^{\frac12 (x(t)-x(t_0))} \| \eta \|_{L_t^\infty L_{xy}^2}^2.
$$
Integrating from $t_{-1}$ to $t_0$, we obtain
$$
J_{+,x_0,t_0}(t_0) \leq J_{+,x_0,t_0}(t_{-1}) + c\, e^{-x_0} \| \eta\|_{L_t^\infty L_{xy}^2}^2 \int_{t_{-1}}^{t_0}  e^{\frac12 (x(t)-x(t_0))} \, dt.
$$
Using that $x(t)-x(t_0) \sim t-t_0$, we obtain \eqref{E:etamon2}. 

We complete the proof  by noting that, for fixed $t_0$, \eqref{E:etamon3} implies 
$$
\lim_{t_1 \to +\infty} J_{-,-x_0,t_0}(t_1) = 0
$$
in \eqref{E:etamon1} and
$$
\lim_{t_{-1} \to +\infty} J_{+,x_0,t_0}(t_{-1}) = 0
$$
in \eqref{E:etamon2}.  From the resulting limiting equations, we obtain 
$$
\| \eta(x,y,t) \|_{L_{y, x<-x_0}^2}^2 \lesssim e^{-x_0} \| \eta\|_{L_t^\infty L_{xy}^2}^2 \quad \mbox{and} \quad \| \eta(x,y,t) \|_{L_{y, x>x_0}^2}^2 \lesssim e^{-x_0} \| \eta\|_{L_t^\infty L_{xy}^2}^2,
$$
which yield \eqref{E:betterdecay}.
\end{proof}

Recall $\tilde \epsilon_n$ from \eqref{E:epsilon-tilde}. 

\begin{corollary}[exponential decay of $\tilde \epsilon_n$ with sharp coefficient]
\label{C:tilde-ep-decay}
For each $r>0$, we have
$$
\| \tilde \epsilon_n \|_{L_s^\infty L_{B(0,r)^c}^2} \lesssim e^{-\omega r} \| \tilde \epsilon_n \|_{L_s^\infty L_{xy}^2},
$$
where $B(0,r)$ is the ball centered at $0$ of radius $r$ in $\mathbb{R}^2$, and $B(0,r)^c$ denotes the complement.
\end{corollary}
\begin{proof}
This follows from the rotation method of Lemmas \ref{L:mon1gen}, \ref{L:mon2} applied to the result of Lemma \ref{L:ep-strong-decay}.  
\end{proof}

Note that Corollary \ref{C:tilde-ep-decay} completes part (2) of the proof of Prop. \ref{P:reduction}.

\section{Comparability of remainder norms} \label{section-C1}

\begin{proposition}
\label{P:normcomp}
Suppose that $E(u)<0$ and $\alpha(u) \ll 1$ so that the geometrical decomposition applies, and $\frac12 \leq \lambda \leq 2$.    Let $a=\| \epsilon \|_{L_s^\infty H_{xy}^1}$ and $b = \| \epsilon \|_{L_s^\infty L_{xy}^2}$.  Moreover, suppose that 
the {\rm $x$-decay property} holds:
$$
\| \epsilon \|_{L_s^\infty L_y^2L_{|x|>x_0}^2} \lesssim \la x_0\ra^{-1} a^\gamma b^{1-\gamma}
$$
for some $0\leq \gamma <1$.  
Then $a\sim b$.
\end{proposition}

\begin{proof}
The proof will use Lemma \ref{L:H1locallyconst}, stated and proved below.   Note that
$$
\int_{y\in \mathbb{R}} \int_{x\in \mathbb{R}} |x| |\epsilon(x,y,s)|^2 \, dx \, dy = \int_{y\in \mathbb{R}} \int_{x\in \mathbb{R}} \int_{z=0}^{|x|} |\epsilon(x,y,s)|^2 \, dz \, dx \, dy
$$
$$ 
=   \int_{z=0}^{+\infty} \int_{y\in \mathbb{R}} \int_{|x|>z} |\epsilon(x,y,s)|^2 \, dx \, dy \, dz \lesssim a^{2\gamma} b^{2-2\gamma} \int_{z=0}^\infty \la z \ra^{-2} \, dz  \lesssim a^{2\gamma} b^{2-2\gamma}.
$$
That is,
\begin{equation}
\label{E:xepbd}
\int_{y\in \mathbb{R}} \int_{x\in \mathbb{R}} |x| |\epsilon(x,y,s)|^2 \, dx \, dy \lesssim a^{2\gamma} b^{2-2\gamma}. 
\end{equation}
Recalling the equation for $\epsilon$, \eqref{E:ep-eqn}, we now compute
\begin{align*}
\indentalign - \frac12 \partial_s \iint x \,\epsilon(x,y,s)^2 \, dx \, dy \\
& = \frac12 \|\epsilon \|_{L_{xy}^2}^2 + \frac32 \|\epsilon_x \|_{L_{xy}^2}^2 + \frac12 \|\epsilon_y \|_{L_{xy}^2}^2 + \frac32 \iint (2xQQ_x - Q^2) \epsilon^2 \, dx \, dy \\
&\qquad - \frac{\lambda_s}{\lambda} \la \epsilon, x \Lambda Q\ra - (\frac{x_s}{\lambda}-1) \la \epsilon, x Q_x \ra - \frac{y_s}{\lambda} \la \epsilon, xQ_y \ra - \frac{\lambda_s}{\lambda} \iint x \epsilon \, \Lambda \epsilon \\
&\qquad - (\frac{x_s}{\lambda}-1) \iint x \epsilon \epsilon_x - \frac{\lambda_s}{\lambda} \iint x \epsilon \epsilon_y - \iint x(3Q\epsilon^2 + \epsilon^3)_x. 
\end{align*}
Using that
$$
\iint x \epsilon \Lambda \epsilon = - \frac12 \iint x \epsilon^2 \,, 
\qquad \iint x\epsilon \epsilon_x =-\frac12 \iint \epsilon^2 \,, 
\qquad \iint x\epsilon \epsilon_y = 0,
$$
and the parameter bounds \eqref{E:param-bds}, we get
$$ 
- \frac12 \partial_s \iint x \epsilon(x,y,s)^2 \, dx \, dy =   \frac32 \|\epsilon_x \|_{L_{xy}^2}^2 + \frac12 \|\epsilon_y \|_{L_{xy}^2}^2 +E,
$$
where  
$$
|E| \lesssim b^2 + b \left| \iint x\epsilon^2 \right| + \| \epsilon\|_{L_{xy}^3}^3.
$$
Applying the Gagliardo-Nirenberg inequality 
$$
\| \epsilon\|_{L_{xy}^3}^3 \lesssim \| \nabla \epsilon \|_{L_{xy}^2} \| \epsilon \|_{L_{xy}^2}^2 \lesssim  b^2 \| \nabla \epsilon \|_{L_{xy}^2}^2 + b^2
$$
and integrating in $s$ over $s_0-\sigma \leq s \leq s_0+\sigma$, where $s_0$ and $\sigma>0$ are as given in Lemma \ref{L:H1locallyconst} below, we obtain
$$
\| \nabla \epsilon \|_{L^2_{s_0-\sigma,s_0+\sigma} L_{xy}^2}^2 \lesssim \sigma b^2 + \left\| \iint x \epsilon^2 \right\|_{L_{[s_0-\sigma,s_0+\sigma]}^\infty}.
$$
By \eqref{E:xepbd} and Lemma \ref{L:H1locallyconst}, we get
$$ 
\sigma a^2 \lesssim \| \epsilon \|_{L^2_{s_0-\sigma,s_0+\sigma} H_{xy}^1}^2 \lesssim \sigma b^2 + a^{2\gamma} b^{2-2\gamma}.
$$
Dividing through by $a^2$, using that $\sigma>0$ is an absolute constant, we now have
$$
1\lesssim (\tfrac{b}{a})^2 + (\tfrac{b}{a})^{2-2\gamma} \lesssim (\tfrac{b}{a})^{2-2\gamma},
$$
where, in the last inequality, we used that $\frac{b}{a} \leq 1$.   Since $\gamma<1$, this implies that $a\lesssim b$.  

\end{proof}

\begin{lemma}
  Let $Q$ be a dyadic decomposition of $x$-space -- specifically take $Q_{-1}= [-1,1]$ and $Q_k = [-2^{k+1},-2^k] \cup [2^k,2^{k+1}]$ for $k\geq 0$.  If $\frac12 \leq \lambda \leq 2$, $|x_t| \lesssim 1$, $|y_t| \lesssim 1$, and $t$ is restricted to a unit-size time interval, and
\begin{equation}
\label{E:bar-notation}
 \bar g(x,y) = \lambda^{-1} g\big(\lambda^{-1}(x-x(t)),\lambda^{-1}(y-y(t)) \big), 
 \end{equation}
then for $1\leq p \leq \infty$,
\begin{equation}
  \label{E:bar-bound}
  \| \bar g \|_{L_x^p L_{yt}^\infty} \lesssim \left\| \, 2^{j/p} \| g\|_{L_{x\in Q_j}^\infty L_y^\infty} \, \right\|_{\ell^1_j}.
  \end{equation}
\end{lemma}
\begin{proof}
The proof is standard, starting with decomposing the outer $x$-integration into the $Q_j$ regions.
\end{proof}

\begin{lemma}
\label{L:H1locallyconst}
Suppose that $E(u)<0$ and $\alpha(u) \ll 1$ so that the geometrical
decomposition applies, and $\frac12 \leq \lambda \leq 2$.  Let $a = \| \epsilon\|_{L_s^\infty H_{xy}^1}$ and $s_0\in \mathbb{R}$ such that $\| \epsilon(s_0)\|_{H_{xy}^1} \geq \frac12a$.  There exists an absolute constant $\sigma>0$ such that for all $s_0-\sigma \leq s \leq s_0+\sigma$, we have $\| \epsilon(s)\|_{H^1_{xy}} \geq \frac1{16} a$.
\end{lemma}
\begin{proof}
Let
$$
\zeta(x,y,t) = \lambda^{-1} \epsilon(\lambda^{-1}(x- x(t)), \lambda^{-1}(y- y(t)),t)
$$
and let $t_0$ correspond to $s_0$ in the time transformation.
Then $\zeta$ solves
\begin{align*}
 \partial_t \zeta &= - \partial_x \Delta \zeta - 3 \partial_x (\bar Q^2 \zeta)    + \la \bar f_1, \zeta \ra \lambda^{-3} \overline{\Lambda Q} + \la \bar f_2, \zeta \ra \lambda^{-3} \overline{ \partial_x Q} + \la \bar f_3, \zeta \ra \lambda^{-3} \overline{ \partial_y Q} \\
& \qquad +  \mu_1 \lambda^{-3} \overline{\Lambda Q} +  \mu_2 \lambda^{-3} \overline{ \partial_x Q} +  \mu_3 \lambda^{-3} \overline{ \partial_y Q}  - 3 \partial_x (\bar Q\zeta^2)  -   \partial_x ( \zeta^3 ),
\end{align*}
where each $|u_j|\lesssim \|\zeta\|_{L_t^\infty L_{xy}^2}^2$, and the bar quantities $\bar Q$, $\overline{\Lambda Q}$, etc. are defined as in \eqref{E:bar-notation}.
Writing this equation in Duhamel form, and applying the $H^1$ theory local estimates, where $t$ is restricted to $[t_0-\sigma, t_0+\sigma]$ on both sides of the equation ($\sigma$ to be determined), we have
$$ 
\| \zeta(t) - U(t) \zeta(t_0) \|_{L_t^\infty H_{xy}^1} \lesssim
\begin{aligned}[t]
 & \| \nabla (\bar Q^2 \zeta) \|_{L_x^1L_{yt}^2}  +|\la \bar f_1, \zeta \ra| \| \nabla \, \overline{\Lambda Q}\|_{L_x^1L_{yt}^2} \\
 &+|\la \bar f_2, \zeta \ra| \|  \nabla \, \overline{ \partial_x Q} \|_{L_x^1L_{yt}^2} +|\la \bar f_3, \zeta \ra| \| \nabla \, \overline{ \partial_y Q} \|_{L_x^1L_{yt}^2} \\
 & +|\mu_1| \| \nabla \, \overline{\Lambda Q} \|_{L_x^1L_{yt}^2} +|\mu_2| \|  \nabla \, \overline{ \partial_x Q} \|_{L_x^1L_{yt}^2} \\
 &+|\mu_3| \| \nabla \, \overline{ \partial_y Q} \|_{L_x^1L_{yt}^2}  + \| \nabla (\bar Q\zeta^2) \|_{L_x^1L_{yt}^2} \\
 &+ \| \nabla ( \zeta^3 ) \|_{L_x^1L_{yt}^2}.
\end{aligned}
$$
(On the right side there are also the (easier to estimate) copies of each term listed without the gradient.)  For the first term,
\begin{align*}
\| \nabla( \bar Q^2 \zeta) \|_{L_x^1 L_{yt}^2} &\lesssim \| \bar Q \, \overline{\nabla Q} \, \zeta \|_{L_x^1L_{yt}^2} + \| \bar Q^2 \, \nabla \zeta \|_{L_x^1L_{yt}^2}\\
&\lesssim \| \bar Q \|_{L_x^4 L_{yt}^\infty} \| \overline{\nabla Q} \|_{L_x^4 L_{yt}^\infty} \| \zeta \|_{L_{xyt}^2} + \| \bar Q \|_{L_x^4L_{yt}^\infty}^2 \| \nabla \zeta \|_{L_{xyt}^2} \\
&\lesssim \sigma^{1/2} \|\zeta \|_{L_t^\infty H_{xy}^1},
\end{align*}
where the $Q$ coefficient terms are controlled by \eqref{E:bar-bound}.

For the term $|\la \bar f_1, \zeta \ra| \| \nabla \, \overline{\Lambda Q}\|_{L_x^1L_{yt}^2}$, we use that
\begin{align*}
  \| \nabla \, \overline{\Lambda Q}\|_{L_x^1L_{yt}^2} = \lambda^{-1} \| \overline{ \nabla \Lambda Q}\|_{L_x^1L_{yt}^2} \lesssim \| \overline{ \nabla \Lambda Q}\|_{L_x^1L_{yt}^\infty}^{1/2} \| \overline{ \nabla \Lambda Q}\|_{L_{xyt}^1}^{1/2} \lesssim \sigma^{1/2}\| \overline{ \nabla \Lambda Q}\|_{L_x^1L_{yt}^\infty}^{1/2} \| \overline{ \nabla \Lambda Q}\|_{L_t^\infty L_{xy}^1}^{1/2},
\end{align*}
which is bounded by $\sigma^{1/2}$ by \eqref{E:bar-bound}.  Thus, we have
$$ 
|\la \bar f_1, \zeta \ra| \| \nabla \, \overline{\Lambda Q}\|_{L_x^1L_{yt}^2} \lesssim \|\bar f_1 \|_{L_t^\infty L_{xy}^2} \| \zeta \|_{L_t^\infty L_{xy}^2} \sigma^{1/2} \lesssim \sigma^{1/2} \| \zeta \|_{L_t^\infty L_{xy}^2}.
$$
The next five terms are treated similarly, which brings us to the next term $\| \nabla( \bar Q \zeta^2) \|_{L_x^1 L_{yt}^2}$.  Distributing the derivative and applying H\"older, we obtain
\begin{align*}
  \| \nabla( \bar Q \zeta^2) \|_{L_x^1 L_{yt}^2} &\lesssim \| \overline{\nabla Q} \|_{L_x^\infty L_{yt}^\infty} \| \zeta \|_{L_x^2 L_{yt}^2} \| \zeta \|_{L_x^2 L_{yt}^\infty} + \| Q \|_{L_x^\infty L_{yt}^\infty} \| \nabla \zeta \|_{L_x^2 L_{yt}^2} \| \zeta \|_{L_x^2 L_{yt}^\infty} \\
 &\lesssim \sigma^{1/2} \| \zeta \|_{L_t^\infty H_{xy}^1},
\end{align*}
using that $\|\zeta\|_{L_x^2L_{yt}^\infty} \lesssim 1$ and $\|\zeta\|_{L_t^\infty H_{xy}^1} \lesssim 1$.  Finally, for the last term, we have
\[ \| \nabla ( \zeta^3) \|_{L_x^1 L_{yt}^2} \lesssim \| \zeta \|_{L_x^4 L_{yt}^\infty}^2 \| \nabla \zeta \|_{L_x^2L_{yt}^2} \lesssim  \sigma^{1/2} \| \zeta \|_{L_x^4 L_{yt}^\infty}^2 \| \nabla \zeta \|_{L_t^\infty L_{xy}^2} \lesssim   \sigma^{1/2} \| \nabla \zeta \|_{L_t^\infty L_{xy}^2}, \]
where we have used that $ \| \zeta \|_{L_x^4 L_{yt}^\infty} \lesssim 1$.

In conclusion, we have obtained
\[\big| \| \zeta\|_{L_t^\infty H_{xy}^1} - \| \zeta(t_0) \|_{H_{xy}^1} \big| \lesssim  \| \zeta(t) - U(t) \zeta(t_0) \|_{L_t^\infty H_{xy}^1} \lesssim \sigma^{1/2} \|\zeta \|_{L_t^\infty H_{xy}^1}. \]
It follows that
\[ \|\zeta(t_0)\|_{H_{xy}^1} - C\sigma^{1/2} \|\zeta\|_{L_t^\infty H_{xy}^1} \leq \| \zeta \|_{L_t^\infty H_{xy}^1} \leq  \|\zeta(t_0)\|_{H_{xy}^1} + C\sigma^{1/2}\| \zeta \|_{L_t^\infty H_{xy}^1},\]
and hence,
\[ \frac{ \| \zeta(t_0) \|_{H_{xy}^1}}{ 1+ C \sigma^{1/2}} \leq \| \zeta(t) \|_{L_t^\infty H_{xy}^1} \leq \frac{ \| \zeta(t_0) \|_{H_{xy}^1}}{ 1- C \sigma^{1/2}}. \]

Taking $\sigma$ small enough so that $C\sigma^{1/2} \leq \frac12$, we obtain the claim.
\end{proof}

\section{Convergence of renormalized remainders}\label{section-C}

\begin{proposition}
\label{P:conv-linear}
If the statement ``for $n$ sufficiently large, $\tilde \epsilon_n \equiv 0$'' is false, then we can pass to a subsequence (still denoted $\tilde \epsilon_n$) such that $\tilde \epsilon_n \not\equiv 0$ for all $n$.   Let $b_n = \|\epsilon_n(t)\|_{L_t^\infty L_{xy}^2}$, so that, by our assumption, $b_n>0$ for all $n$, and there exists $s_n\in \mathbb{R}$ such that $\|\epsilon_n(s_n) \|_{L_{xy}^2} \geq \frac12 b_n$.  Let
$$
w_n(x,y,s) = \frac{\tilde \epsilon_n(x,y,s+s_n)}{b_n}.
$$
Then, passing to another subsequence, $w_n \to w_\infty$ in $C_{\text{loc}}(\mathbb{R}; L^2(\mathbb{R}^2))$, where $w_\infty$ satisfies 
\begin{enumerate}
\item $w_\infty \in L_s^\infty H_{xy}^1$ and solves
\begin{equation}
\label{E:w-inf-eqn}
\partial_s w_\infty = \partial_x \mathcal{L} w_\infty + \alpha(s) \Lambda Q + \beta(s) Q_x  + \gamma(s) Q_y, 
\end{equation}
\item $w_\infty$ satisfies the orthogonality conditions 
\begin{equation}\label{E:w-ortho}
\la Q^3, w_\infty\ra =0 \quad \mbox{and} \quad \la \nabla Q, w_\infty \ra =0,
\end{equation}
\item $w_\infty$ is nontrivial, in fact, $\|w_\infty(0)\|_{L^2_{xy}} \geq \frac12$,
\item $w_\infty$ satisfies a spatial localization, uniformly in time
$$
\| w_\infty(s) \|_{L^2_{B(0,r)^c}} \lesssim e^{-\omega r}
$$
for some $\omega>0$, and any radius $r>0$; here, $B(0,r)$ denotes the ball in $\mathbb{R}^2$ of center $0$ and radius $r$, and $B(0,r)^c$ denotes the complement of this ball.
\end{enumerate}
\end{proposition}

\begin{proof}

Let $t_n$ correspond to $s_n$ (in the time conversion).  Since $\frac12 \leq \tilde \lambda_n(t) \leq 2$ for all $n$ and all $t\in \mathbb{R}$, we can pass to a subsequence so that $\tilde \lambda_n(t_n) \to \lambda_\infty$. Let
$$
\zeta_n(x,y,t) = b_n^{-1} \tilde \lambda_n^{-1} \tilde \epsilon_n(\tilde \lambda_n^{-1}(x-\tilde x_n(t+t_n)+\tilde x_n(t_n)),\tilde \lambda_n^{-1}(y-\tilde y_n(t+t_n)+\tilde y_n(t_n)),t+t_n).
$$
Every instance of $\tilde \lambda_n$ is evaluated at time $t+t_n$, although this has been suppressed.  Then $\zeta_n$ satisfies
$$
b_n\zeta_n(x,y,t) = 
\begin{aligned}[t]
&\tilde u_n(x+\tilde x_n(t_n),y+\tilde y_n(t_n),t+t_n) \\
&- \tilde \lambda_n^{-1}Q\big(\tilde \lambda_n^{-1}(x-\tilde x_n(t+t_n)+\tilde x_n(t_n)),\tilde \lambda_n^{-1}(y-\tilde y_n(t+t_n)+\tilde y_n(t_n)) \big).
\end{aligned}
$$  

By Corollary \ref{C:tilde-ep-decay} (the exponential decay estimate on $\tilde \epsilon_n$),  and $\frac12\leq \tilde \lambda_n \leq 2$, we have, uniformly in $n$,
$$
\| \zeta_n \|_{L_t^\infty L^2_{B(0,R)^c}} \lesssim e^{-\omega R}
$$
for any ball of radius $R$.  Moreover, by Prop. \ref{P:normcomp} (giving the comparability between $L_{xy}^2$ and $H_{xy}^1$ norms of $\epsilon_n$), we have, uniformly in $n$,
$$
\| \zeta_n \|_{L_t^\infty H_{xy}^1} \lesssim 1.
$$
Hence, at time $t=0$, by the Rellich-Kondrachov compactness theorem, we can pass to a subsequence such that $\zeta_n(0) \to \zeta_\infty(0)$ strongly in $L_{xy}^2$, for some $H_{xy}^1$ function $\zeta_\infty(0)$. By relabeling, we assume that $\zeta_n$ follows this subsequence. We will now use local theory estimates to prove that this convergence holds at all times $t$.  In fact, we show that for each $T>0$,
\begin{equation}
\label{E:zeta-conv}
\| \zeta_n(t) - \zeta_\infty(t)\|_{L_{-T<t<T}^\infty L_{xy}^2} \to 0
\end{equation}
as $n\to \infty$, where $\zeta_\infty$ solves the equation \eqref{E:zeta-inf} with initial condition $\zeta_\infty(0)$.  Since $\frac12 \leq \tilde \lambda_n \leq 2$, this implies the same convergence for $w_n$ to $w_\infty$ in the proposition statement.

Let $\bar Q$ denote the second term, for convenience, in the definition of $\zeta_n$, which has implicit $t$ and $n$ dependence,
\begin{equation}
\label{E:bar-Q-def}
\bar Q \defeq \tilde \lambda_n^{-1}Q(\tilde \lambda_n^{-1}(x-\tilde x_n(t+t_n)+\tilde x_n(t_n)),\tilde \lambda_n^{-1}(y-\tilde y_n(t+t_n)+\tilde y_n(t_n))).
\end{equation}
Since the space and time shifts are by constants for each $n$, $\tilde u_n(x+\tilde x_n(t_n),y+\tilde y_n(t_n),t+t_n)$ still solves the ZK equation \eqref{ZK}, and we compute that $\zeta_n$ solves
$$
\begin{aligned}
\indentalign \partial_t \zeta_n + \partial_x \Delta \zeta_n + \partial_x [3\bar Q^2 \zeta_n + 3\bar Q\zeta_n^2 b_n + \zeta_n^3 b_n^2] \\
&= - b_n^{-1}(\tilde \lambda_n^{-1})_t \,\partial_{\lambda^{-1}} \bar Q + b_n^{-1}(\tilde y_n)_t \,\partial_y \bar Q + b_n^{-1}((\tilde x_n)_t - \tilde \lambda_n^{-2}) \partial_x \bar Q.
\end{aligned}
$$
We rewrite the last line as
$$
(b_n^{-1}\tilde \lambda_n^2 (\tilde \lambda_n)_t) \, \tilde \lambda_n^{-3} \, \overline{\Lambda Q} + (b_n^{-1}\tilde \lambda_n^2 (\tilde y_n)_t) \, \tilde \lambda_n^{-3} \, \overline{\partial_y Q} + (b_n^{-1}(\tilde \lambda_n^2 (\tilde x_n)_t - 1)) \, \tilde \lambda_n^{-3} \, \overline{\partial_x Q},
$$
where $\overline{\Lambda Q}$, $\overline{ \partial_x Q}$, and $\overline{\partial_y Q}$ are defined in terms of $\Lambda Q$, $\partial_x Q$, and $\partial_y Q$ in the same way as $\bar Q$ is defined in terms of $Q$.  We can now substitute the equations for the parameters at the expense of $O(b_n)$ terms.  Let us write
$$
b_n (\mu_n)_1 = b_n^{-1} \tilde \lambda_n^2 \tilde (\lambda_n)_t - \la \bar f_1, \zeta_n \ra, \qquad 
$$
$$
\qquad b_n (\mu_n)_2 = b_n^{-1}( \tilde \lambda_n^2 (x_n)_t -1 ) - \la \bar f_2, \zeta_n \ra, ~~\mbox{and}
$$
$$
b_n (\mu_n)_3  = b_n^{-1}\tilde\lambda_n^2 (y_n)_t  - \la \bar f_3, \zeta_n \ra,
$$
so that each $|(\mu_n)_j| \lesssim 1$.  Then
\begin{equation}
\label{E:zeta-n}
\begin{aligned}
\partial_t \zeta_n &= - \partial_x \Delta \zeta_n - 3 \partial_x (\bar Q^2 \zeta_n)    + \la \bar f_1, \zeta_n \ra \tilde \lambda_n^{-3}\, \overline{\Lambda Q} + \la \bar f_2, \zeta_n \ra \tilde \lambda_n^{-3}\, \overline{ \partial_x Q} + \la \bar f_3, \zeta_n \ra \tilde \lambda_n^{-3}\, \overline{ \partial_y Q} \\
& \qquad + b_n (\mu_n)_1 \tilde \lambda_n^{-3}\, \overline{\Lambda Q} +  b_n (\mu_n)_2 \tilde \lambda_n^{-3} \,\overline{ \partial_x Q} + b_n (\mu_n)_3 \tilde \lambda_n^{-3} \,\overline{ \partial_y Q}  - 3 b_n \partial_x (\bar Q\zeta_n^2)  -  b_n^2 \partial_x ( \zeta_n^3 ).
\end{aligned}
\end{equation}

On the time interval $-T\leq t \leq T$, the parameter bounds imply the following.  First, the fact that $|(\tilde \lambda_n)_t| \lesssim b_n$ implies 
\begin{equation}
\label{E:tilde-lambda-conv}
| \tilde \lambda_n(t+t_n) - \lambda_\infty | \lesssim b_n T + |\tilde \lambda_n(t_n) - \lambda_\infty|.
\end{equation}
Since $\frac12 \leq \tilde \lambda_n \leq 2$ and $\frac12 \leq \lambda_\infty \leq 2$, it follows that
\begin{equation}
\label{E:tilde-lambda-conv-2}
| \tilde \lambda_n(t+t_n)^{-2} - \lambda_\infty^{-2} | \lesssim b_n T + |\tilde \lambda_n(t_n) - \lambda_\infty|
\end{equation}
as well.  The fact that $|(\tilde x_n)_t - \tilde \lambda_n^{-2}| \lesssim b_n$ implies that 
$$
\left| \tilde x_n(t+t_n) - \tilde x_n(t_n) - \int_0^t \tilde \lambda_n(\sigma+t_n)^{-2} \, d\sigma \right| \lesssim b_n t.
$$
Using \eqref{E:tilde-lambda-conv-2}, we obtain
\begin{equation}
\label{E:tilde-x-conv}
| \tilde x_n(t+t_n) - \tilde x_n(t_n) - \lambda_\infty^{-2} t | \lesssim b_nT + (b_n T + |\tilde \lambda_n(t_n) - \lambda_\infty|) \,T .
\end{equation}
Finally, since $|(\tilde y_n)_t| \lesssim b_n$, it follows that
\begin{equation}
\label{E:tilde-y-conv}
|\tilde y_n(t+t_n) - \tilde y_n(t_n)| \lesssim b_n \,T .
\end{equation}
The convergence statements \eqref{E:tilde-lambda-conv}, \eqref{E:tilde-x-conv} and \eqref{E:tilde-y-conv} will be used below.
To deduce the equation for the expected limit $\zeta_\infty$, we replace $\bar Q$ with its limiting value
\begin{equation}
\label{E:hat-Q-def}
\hat Q(x,y,t) \defeq \lambda_\infty^{-1}Q(\lambda_\infty^{-1}(x-\lambda_\infty^{-2}t), \lambda_\infty^{-1} y),
\end{equation}
and similarly, for $\overline{\Lambda Q}$, $\overline{\partial_x Q}$,
$\overline{\partial_y Q}$, and drop all $O(b_n)$ terms.  Note that \eqref{E:tilde-lambda-conv}, \eqref{E:tilde-x-conv} and \eqref{E:tilde-y-conv} imply that for any $f$ (which could be $Q$, $\Lambda Q$, $\partial_x Q$, or $\partial_y Q$ in the analysis), we have
\begin{equation}
\label{E:bar-hat-conv}
\| \bar f - \hat f \|_{L_{[-T,T]}^\infty H_{xy}^1} \to 0
\end{equation}
as $n\to \infty$.  From this, we deduce the expected form for the equation for the limit $\zeta_\infty$ from the equation \eqref{E:zeta-n} for $\zeta_n$ and the above convergence assertions, thus, we have
\begin{equation}
\label{E:zeta-inf}\partial_t \zeta_\infty = - \partial_x \Delta \zeta_\infty - 3 \partial_x (\hat Q^2 \zeta_\infty) + \la \hat f_1, \zeta_\infty \ra \tilde \lambda_\infty^{-3}\, \widehat{\Lambda Q} + \la \hat f_2, \zeta_\infty \ra \tilde \lambda_\infty^{-3}\, \widehat{ \partial_x Q} + \la \hat f_3, \zeta_\infty \ra \tilde \lambda_\infty^{-3}\, \widehat{ \partial_y Q}. 
\end{equation}
Thus, let us take $\zeta_\infty$ to be the solution to this equation on $[-T,T]$ with the initial condition $\zeta_\infty(0)$.  We will now prove \eqref{E:zeta-conv}.  

Into the first line of \eqref{E:zeta-n}, we make the following replacements
\begin{itemize}
\item $\zeta_n = \rho_n + \zeta_\infty$
\item $\overline{Q} = \widehat{Q} + (\overline{Q} - \widehat{Q})$
\item $\overline{\Lambda Q} = \widehat{\Lambda Q} + (\overline{\Lambda Q} - \widehat{\Lambda Q})$
\item $\overline{\partial_x Q} = \widehat{\partial_x Q} + (\overline{\partial_x Q} - \widehat{\partial_x Q})$
\item $\overline{\partial_y Q} = \widehat{\partial_y Q} + (\overline{\partial_y Q} - \widehat{\partial_y Q})$
\item $\tilde \lambda_n = \tilde \lambda_\infty + (\tilde \lambda_n - \tilde \lambda_\infty)$
\end{itemize}
and then use \eqref{E:zeta-inf} to simplify the result.    This gives us a $\rho_n$ equation, that we estimate using the local theory estimates, \eqref{E:bar-hat-conv}, \eqref{E:tilde-lambda-conv}, and the fact that $\rho_n(0) \to 0$ strongly in $L^2_{xy}$, and $b_n \to 0$.  
\end{proof}

\section{The linearized virial estimate}\label{S:virial-1}

We now study the limit function $w_\infty$; for simplicity we drop the index and write $w$. We recall the equation \eqref{E:w-inf-eqn} that $w$ satisfies and the orthogonality conditions \eqref{E:w-ortho}.  It is not clear how to directly control a virial for $w$, so instead we consider an adjoint problem (formally, $v=\mathcal L w$) and obtain the virial estimate for $v$. In order to do so, we use numerics to obtain the discrete spectrum for a certain operator (see discussion in Lemma \ref{L:v-virial} below) and then use an angle lemma, to be precise Lemma \ref{L:angle}, to finish the proof of the virial estimate. In the following Section 
\ref{S:liouville} we deduce that the only possible $w$ that satisfies all the conditions is the trivial $w\equiv 0$, and hence, our assumption that Theorem \ref{T:main} does not hold is false, thus, concluding that the solution $u(t)$ as considered in the theorem does indeed blow up. 

\begin{lemma}[linearized virial  estimate for $w$]
\label{L:w-virial}
Suppose that $w \in C^0(\mathbb{R}_t; H_{xy}^1) \cap C^1(\mathbb{R}_t; H_{xy}^{-2})$  solves
$$
\partial_t w = \partial_x \mathcal{L} w + \alpha \Lambda Q + \beta Q_x + \gamma Q_y
$$
for time-dependent coefficients $\alpha$, $\beta$, and $\gamma$.  Suppose, moreover, that $w$ satisfies the orthogonality conditions 
$$
\la w, Q^3\ra =0, \quad \la w, Q_x \ra =0, \quad \la w, Q_y \ra =0.
$$  
Then
\begin{equation}
\label{E:w-virial-1}
\| w \|_{L_t^2 H_{xy}^1} \lesssim \| \la x \ra^{1/2} w \|_{L_t^\infty L_{xy}^2},
\end{equation}
where $t$ is carried out globally $-\infty<t< \infty$.  
\end{lemma}
\begin{proof}
We will reduce this to Lemma \ref{L:v-virial} below.  For $\delta>0$ to be chosen small later, let
$$
v = (1-\delta \Delta)^{-1} \mathcal{L} w.
$$
Since $\mathcal{L}Q_x=0$ and $\mathcal{L}Q_y =0$ and $\mathcal{L}\Lambda Q = -2 Q$, we compute
\begin{align*}
\partial_t v &= (1-\delta \Delta)^{-1} \mathcal{L} \partial_t w \\
&= (1-\delta \Delta)^{-1}[\mathcal{L} \partial_x \mathcal{L} w + \alpha \mathcal{L} \Lambda Q] \\
&= (1-\delta \Delta)^{-1}[\mathcal{L}\partial_x (1-\delta\Delta) v -2 \alpha Q] \\
&= \mathcal{L}\partial_x v - 2\alpha Q + E_\delta v,
\end{align*}
where $E_\delta$ is a zero-order operator with the property that
\begin{equation}
\label{E:E-delta-bound}
|\la   \la x \ra  E_\delta v, v \ra | \lesssim \delta \|v\|_{H_{xy}^1}^2
\end{equation}
(sacrificing regularity one gains a power of $\delta$).   Specifically, \eqref{E:E-delta-bound} is derived as follows.  By definition, we have the formula
$$
E_\delta v \defeq [ (1-\delta \Delta)^{-1} \mathcal{L} (1-\delta \Delta) - \mathcal{L} ] \partial_x v - 2\alpha ( (1-\delta \Delta)^{-1} - 1) Q.
$$
Note that via the orthogonality conditions on $w$, the equation for $w$, and the definition of $v$ in terms of $w$, we have $|\alpha| \lesssim \|v\|_{L_{xy}^2}$.  Substituting $\mathcal{L} = 1-\Delta - 3 Q^2$, we get
$$
E_\delta v = - 3 [ (1-\delta\Delta)^{-1} Q^2 (1-\delta \Delta) - Q^2] \partial_x v - 2\alpha( ( 1-\delta \Delta)^{-1} - 1)Q.
$$
Using the commutator identity
$$
Q^2 (1-\delta \Delta) = (1-\delta \Delta) Q^2  + \delta (\Delta Q^2) + 2 \delta (\nabla Q) \nabla, 
$$
we reduce to
$$
E_\delta v = -3\delta (1-\delta \Delta)^{-1} [\Delta Q^2 + 2\nabla Q \cdot \nabla] \partial_x v - 2\alpha( ( 1-\delta \Delta)^{-1} - 1)Q.
$$
Finally, we use 
$$
(1 - \delta \Delta)^{-1} - 1 =  \delta (1-\delta \Delta)^{-1} \Delta 
$$ 
to simplify the second term:
$$
E_\delta v = -3\delta (1-\delta \Delta)^{-1} [\Delta Q^2 + 2\nabla Q \cdot \nabla] \partial_x v - 2 \delta \alpha ( 1-\delta \Delta)^{-1} \Delta Q.
$$
From this formula, we see that \eqref{E:E-delta-bound} follows from Lemma \ref{L:weight-transfer}. Each term has a $\delta$ factor, and each term has a $Q$-weight to absorb the $\la x \ra$ in \eqref{E:E-delta-bound}.

Furthermore, we have
$$
\la v, (1-\delta \Delta) Q_x \ra = \la \mathcal{L}w, Q_x \ra = \la w, \mathcal{L}Q_x \ra =0,
$$
$$
\la v, (1-\delta \Delta) Q_y \ra = \la \mathcal{L}w, Q_y \ra = \la w, \mathcal{L}Q_y \ra =0.
$$
By the orthogonality condition $\la w, Q^3 \ra =0$,
$$
\la v, (1-\delta \Delta) Q \ra = \la \mathcal{L} w, Q \ra = \la w, \mathcal{L}Q \ra = - 2\la w, Q^3 \ra =0.
$$
Thus, by spectral stability, the given estimate for the $E_\delta$ error, and the $v$-to-$w$ conversion estimates in Lemma \ref{L:transfer}, the estimate \eqref{E:w-virial-1} is reduced to Lemma \ref{L:v-virial} below.
\end{proof}

\begin{lemma}[linearized virial  estimate for $v$]
\label{L:v-virial}
Suppose that $v \in C^0(\mathbb{R}_t; H_{xy}^1) \cap C^1(\mathbb{R}_t; H_{xy}^{-2})$  solves
$$\partial_t v =  \mathcal{L}\partial_x v  -2 \alpha Q$$
for some time dependent coefficient $\alpha$, and moreover, $v$ satisfies the orthogonality conditions 
\begin{equation}\label{E:ortho-v}
\la v, Q \ra =0, \quad \la v, Q_x \ra =0, \quad \la v, Q_y\ra =0.
\end{equation}
Then 
\begin{equation}
\label{E:v-virial-1}
\| v\|_{L_t^2H_{xy}^1} \lesssim \| \la x \ra^{1/2} v \|_{L_t^\infty L_{xy}^2},
\end{equation}
where $t$ is carried out over all time $-\infty < t< \infty$.
\end{lemma}
\begin{proof}
Using the orthogonality condition $\la v,Q \ra =0$, we compute
$$
0= \partial_t \la v, Q \ra = \la \mathcal{L}\partial_x v, Q \ra -2 \alpha \la Q, Q \ra.
$$
This yields 
$$
\alpha = \frac{\la v, 3Q^2Q_x \ra}{\la Q, Q \ra}
$$
so that
$$
\partial_t v = \mathcal{L} \partial_x v - \frac{\la v, 6Q^2Q_x \ra}{\la Q, Q\ra} Q.
$$
Now compute
\begin{equation}
\label{E:v-virial-2}
-\frac12 \partial_t \int x v^2 = \la Bv,v\ra + \la P v,v\ra  =: \la A v,v\ra,
\end{equation}
where
\begin{equation}\label{E:B}
B := \frac12 - \frac32\partial_x^2 - \frac12\partial_y^2  - \frac32 Q^2 - 3 x QQ_x
\end{equation}
and $P$ can be taken as the rank $2$ self-adjoint operator
\begin{equation}\label{E:P}
Pv := \frac12 \frac{ 6\, Q^2Q_x}{\la Q, Q\ra} \la v, xQ\ra+ \frac12 \frac{xQ}{\la Q, Q\ra} \la v, 6Q^2Q_x\ra.
\end{equation}
From \eqref{E:B} it follows that the continuous spectrum of the operator $A 
 \, (\equiv B+P)$ is $[\frac12,+\infty)$. 

We would like to understand the discrete spectrum of the operator $A$, however, since $Q$ is not given explicitly and due to a non-trivial representation of the operator $P$, we study the discrete spectrum of $A$ numerically (see details in appendix).
In this paper we offer one possible numerical approach, and in \cite{HRY2025} we show several other numerically-assisted studies of the spectral properties of this operator. 
We point out that even in the 1D case of the critical gKdV equation, when $Q$ is given explicitly as a $\sech$ function, and where the spectral properties of certain classical Schr\"odinger-type operator are well-known, certain function (an inverse of the ground state $Q$ under a specific second order nonlinear operator) had to be computed numerically as well as signs of some inner products also had to be checked numerically, see  (177), Remark on p. 415 and the proof of Lemma 26 in \cite{MM-liouville}.

We find that the operator $A$ has two discrete eigenvalues below $\frac12$, that is, 
\begin{equation}\label{E:lambdas}
\lambda_o=-0.5368 \quad \mbox{and} \quad \lambda_e=-0.1075.
\end{equation}
For any simple eigenstate of $A$, the corresponding eigenfunction is either even or odd in $x$, and either even or odd in $y$.  This follows from parity preserving properties of the differential operator $B$ and the projection $P$.  Since $B$ is a standard Schr\"odinger operator, the ground state is simple and moreover has a nonnegative eigenfunction (so is even in $x$ and even in $y$), but this property need not hold for the perturbed operator $A=B+P$.
However, the numerics confirm the simplicity of all the eigenstates for $A$, so that we can assume eigenstates are either even or odd in $x$, and either even or odd in $y$.  For illustration we provide the graphs of the eigenfunctions in appendix, in Figure \ref{F:f} (see also cross-sections in Figure \ref{F:f2}). We obtain that the eigenfunction corresponding to $\lambda_e$, denoted by $f_e$, is even in $x$ and $y$ (the index `e' corresponds to `even'). The lowest eigenvalue $\lambda_o$ has the eigenfunction $f_o$, which is odd with respect to the $x$-axis (the index `o' corresponds to `odd').

Redefining the normalized eigenfunctions again by $f_o$ and $f_e$, and denoting by $g_1 = \frac{Q}{\|Q\|_{L^2}}$ and $g_2 = \frac{Q_x}{\|Q_x\|_{L^2}}$ (see Figure \ref{F:f2}), we compute the inner products (this uses the fact that $f_i$'s are computed numerically):
\begin{align}\label{D:geometry}
&\la f_o, g_1 \ra =0 \,, \qquad\qquad \la f_o,g_2 \ra = -0.8739,\\
&\la f_e, g_1 \ra = -0.9902  \,, \quad \la f_e, g_2 \ra = 0.
\end{align}

Next, we decompose $L^2$ into an orthogonal decomposition $L^2(\mathbb{R}^2) = H_o \oplus H_e$, where the closed subspace $H_o$ of $L^2(\mathbb{R}^2)$ is given by functions that are odd in $x$ (no constraint in $y$) and the closed subspace $H_e$ of $L^2(\mathbb{R}^2)$ is given by functions that are even in $x$ (no constraint in $y$).  Taking $P_o$ and $P_e$ to be the corresponding orthogonal projections and observing that $AP_o = P_oA$ and $AP_e = P_eA$, we have 
\begin{equation}\label{E:v}
\la v, v \ra = \la P_o v, P_o v \ra + \la P_e v, P_e v\ra,
\end{equation}
\begin{equation}\label{E:Av}
\la Av, v \ra = \la AP_o v, P_ov \ra + \la AP_e v, P_e v\ra.
\end{equation}
If we prove that for some positive constants $\mu_e$ and $\mu_o$ 
$$\la  A P_e v, P_e v \ra \geq \mu_e \la P_ev, P_e v \ra \quad \mbox{and} \quad \la A P_o e, P_o v \ra \geq \mu_o \la P_o v, P_o v \ra,$$
then together with \eqref{E:v} and \eqref{E:Av} yields 
\begin{equation}\label{E:A}
\la Av, v \ra \geq \min(\mu_e, \mu_o) \la v, v \ra \equiv c \, \|v\|^2_{L^2}.
\end{equation}

Observe that $f_o$ and $g_2$ (normalized $Q_x$) belong to $H_o$, while $f_e$ and $g_1$ (normalized ground state) belong to $H_e$. Thus, $A\big|_{H_o}$ has spectrum $\{ \lambda_o \} \cup [\frac12, +\infty)$. 
Applying the angle lemma (Lemma \ref{L:angle} below) with $H=H_o$ and $\lambda_\perp = \frac12$, we obtain 
$$
(\lambda_\perp - \lambda_o) \sin^2\beta = (0.5+0.5368)*(1-0.8739^2)=0.2450,
$$
and thus, 
\begin{equation}\label{E:Ao}
\la AP_o v, P_o v \ra \geq (0.5000-0.2450) \la P_ov, P_o v \ra.
\end{equation}
The operator $A\big|_{H_e}$ has spectrum $\{ \lambda_e \} \cup [\frac12, +\infty)$, thus,  
applying again the angle lemma (Lemma \ref{L:angle} below) with $H=H_e$, $\lambda_\perp = \frac12$, we obtain  
$$
(\lambda_\perp - \lambda_e) \sin^2\beta = (0.5+0.10755)*(1-0.9902^2)=0.0118,
$$ 
and hence, 
\begin{equation}\label{E:Ae}
\la AP_e v, P_e v \ra \geq (0.5000-0.0118) \la P_e v, P_e v\ra.
\end{equation}
Combining \eqref{E:Ao} and \eqref{E:Ae}, we obtain \eqref{E:A}, showing the positivity of $A$, provided  
$v$ satisfies the orthogonality conditions \eqref{E:ortho-v}. Integrating \eqref{E:v-virial-2} in time and using the energy estimates (or often referred to as elliptic regularity), we obtain \eqref{E:v-virial-1}.
\end{proof}

We provide below the angle lemma used in the previous proof, and mention that a version of this angle lemma was used in a similar way for spectral computations in \cite[Lemma 4.9]{HPZ}.

\begin{lemma}[angle lemma]
\label{L:angle}
Suppose that $A$ is a self-adjoint operator on a Hilbert space $H$ with an eigenvalue $\lambda_1$ and the corresponding eigenspace spanned by a  function $e_1$ with $\|e_1\|_{H}=1$.  Let $P_1f = \la f,e_1\ra e_1$ be the corresponding orthogonal projection.   Assume that $(I-P_1)A$ has spectrum bounded below by $\lambda_\perp$, with $\lambda_\perp>\lambda_1$.  Suppose that $f$ is some other function such that $\|f\|_{H}=1$ and $0 \leq \beta \leq \pi$ is defined by $\cos \beta = \la f, e_1\ra$.   Then if $v$ satisfies $\la v, f \ra =0$, we have
$$
\la Av,v \ra \geq (\lambda_\perp - (\lambda_\perp - \lambda_1)\sin^2\beta)\|v\|_{H}^2.
$$
\end{lemma}
\begin{proof}
It suffices to assume that $\|v\|_{H}=1$.  Decompose $v$ and $f$ into their orthogonal projection onto $e_1$ and its orthocomplement:
$$v =  (\cos \alpha) e_1 +  v_\perp \,, \qquad \|v_\perp\|_{H}=\sin \alpha,$$
$$f = (\cos \beta) e_1 + f_\perp \,, \qquad \|f_\perp\|_{H} = \sin \beta$$
for $0\leq \alpha, \beta \leq \pi$.  
Then 
$$ 
0 = \la v,f \ra = \cos\alpha \cos \beta + \la v_\perp, f_\perp \ra,
$$
from which it follows that
$$
|\cos \alpha \cos \beta| = |\la v_\perp, f_\perp \ra| \leq \|v_\perp\|_H \|f_\perp\|_{H} = \sin \alpha \sin \beta,
$$
and thus, it follows that $|\cos \alpha| \leq \sin \beta$.  Now
\begin{align*}
\la Av,v\ra &= \lambda_1 \cos^2\alpha + \la Av_\perp,v_\perp \ra \geq \lambda_1 \cos^2\alpha + \lambda_\perp\sin^2 \alpha \\
&= \lambda_\perp - (\lambda_\perp-\lambda_1) \cos^2\alpha \geq  \lambda_\perp - (\lambda_\perp-\lambda_1) \sin^2\beta.
\end{align*}
\end{proof}

\section{Reduction of linear Liouville to linearized virial}\label{S:liouville}

\begin{proposition}[linear Liouville theorem]
\label{P:linear-liouville}
Suppose that $w \in C^0(\mathbb{R}_t; H_{xy}^1) \cap C^1(\mathbb{R}_t; H_{xy}^{-2})$ solves
$$\partial_t w = \partial_x \mathcal{L} w + \alpha \Lambda Q + \beta Q_x + \gamma Q_y$$
for time-dependent coefficients $\alpha$, $\beta$, and $\gamma$.  Suppose, moreover, that $w$ satisfies the orthogonality conditions 
$$ 
\la w, Q^3\ra =0, \quad \la w, Q_x \ra =0, \quad \la w, Q_y \ra =0,
$$ 
and is globally-in-time uniformly $x$-spatially localized
\begin{equation}\label{E:loc1}
\| \la x \ra^{1/2+} w \|_{L_t^\infty L_{xy}^2} < \infty.
\end{equation}
Then $w\equiv 0$.  
\end{proposition}

\begin{proof}
We can argue that the assumption that $w$ is uniformly in time $x$-localized implies 
\begin{equation}
\label{E:extra-orth}
\la w, Q \ra =0.
\end{equation}
Indeed, let
$$
F(x,y) = \int_0^x \Lambda Q(x',y) \, dx',
$$
which does not decay as $x\to \pm \infty$.  Note that since $\Lambda Q$ is even in both $x$ and $y$, it follows that $F(x,y)$ is odd in $x$ and even in $y$.
Let
$$
J(t) = \la w, F\ra.
$$
Despite the lack of decay in $F$ as $x\to \pm \infty$, this quantity is finite due to the $x$ decay assumption for $w$.  Specifically,
\begin{equation}
\label{E:J-bounded}
| J(t) |_{L_t^\infty} \leq \| \la x \ra^{1/2+} w \|_{L_t^\infty L_{xy}^2} \|  \la x \ra^{-1/2-} F \|_{L_{xy}^2} < \infty,
\end{equation}
due to our assumption \eqref{E:loc1} and since
$$
\| \la x \ra^{-1/2-} F\|_{L_{xy}^2} \leq \| \Lambda Q \|_{L_x^1L_y^2}  < \infty.
$$
Since $F$ is odd in $x$ and $\Lambda Q$ is even in $x$,  $\la \Lambda Q, F \ra =0$.  By integration by parts in $x$, $\la Q_x, F\ra = - \la Q, \Lambda Q\ra =0$.  Since $Q_y$ is even in $x$ and $F$ is odd in $x$, $\la Q_y, F \ra =0$ (alternatively this holds since $Q_y$ is odd in $y$ and $F$ is even in $y$).  These orthogonality statements and the fact that $\mathcal{L}\Lambda Q = - 2Q$ imply
$$
J'(t) =  2 \la w, Q \ra.
$$
Since $\mathcal{L}Q_x =0$, $\la \Lambda Q, Q \ra=0$, $\la Q_x, Q\ra =0$, and $\la Q_y, Q\ra=0$, we have
$$
J''(t)=0.
$$
Thus $J(t) = a_0 + a_1 t$, but $J(t)$ is bounded by \eqref{E:J-bounded}, so we must have $a_1=0$, i.e., \eqref{E:extra-orth} holds.  

Now by \eqref{E:extra-orth} and $\mathcal{L}\Lambda Q = -2Q$, we deduce
$$
\partial_t \la \mathcal{L}w,w \ra = 0,
$$
i.e., $\la \mathcal{L}w, w \ra$ is constant in time.  

By the trivial estimate $\la \mathcal{L} w, w \ra \lesssim \|w\|_{H_{xy}^1}^2$ uniformly for all $t$ and the linearized virial  estimate Lemma \ref{L:w-virial},
$$
\int_{t=-\infty}^{+\infty} \la \mathcal{L}w,w \ra \, dt  \lesssim  \|  w \|_{L_t^2 H_{xy}^1}^2 \lesssim \| \la x \ra^{1/2} w \|_{L_t^\infty L_{xy}^2}^2 < \infty.
$$
Since $\la \mathcal{L}w,w\ra$ is constant in time, this implies that
$$
\la \mathcal{L}w,w \ra =0
$$
for all $t\in \mathbb{R}$.  
By the orthogonality conditions, $\mathcal{L}$ is strictly positive definite (Lemma \ref{Lemma-ort2}), and this implies $w\equiv 0$.
\end{proof}

\section{Conversion between $w$ and $v=(1-\delta\Delta)^{-1}\mathcal{L}w$}

\label{S:conversion}

In this section, we prove two lemmas (Lemma \ref{L:weight-transfer} and Lemma \ref{L:H1transfer}) that allow us to transfer the estimate \eqref{E:v-virial-1} to \eqref{E:w-virial-1} (from $v$ to $w$), as summarized in Lemma \ref{L:transfer} below.  The lemma provides spatial estimates (independent of, and hence, uniform) in time $t$.

\begin{lemma}
\label{L:transfer}
Suppose that $v= (1-\delta \Delta)^{-1} \mathcal{L}w$ and $0< \delta \leq 1$.   Then
\begin{equation}
\label{E:weight-transfer}
\| \la x \ra^{1/2} v \|_{L_{xy}^2} \lesssim \delta^{-1} \| \la x \ra^{1/2} w \|_{L_{xy}^2}.
\end{equation}
Suppose further that $\la w, \nabla Q \ra =0$.  Then there exists $\delta_0>0$ such that for all $0<\delta \leq \delta_0$, we have
\begin{equation}
\label{E:H1transfer}
\| w\|_{L_{xy}^2}^2 + \delta^{1/4} \|\nabla w\|_{L_{xy}^2}^2 \lesssim \| v\|_{L_{xy}^2}^2 + \delta \| \nabla v \|_{L_{xy}^2}^2.
\end{equation}
\end{lemma}
\begin{proof}
These estimates are a consequence of Lemmas \ref{L:weight-transfer} and \ref{L:H1transfer} below.  Note that
\begin{align*}
v &= (1-\delta \Delta)^{-1}\mathcal{L}w = (1-\delta \Delta)^{-1} (1-\Delta - 3Q^2) w \\
&= (1-\delta \Delta)^{-1} w - \Delta (1-\delta \Delta)^{-1}w - 3 (1-\delta \Delta)^{-1}Q^2 w.
\end{align*}
In the middle term, use $-\Delta = - \delta^{-1}\delta \Delta =  \delta^{-1}(1-\delta \Delta) -\delta^{-1}$ , which implies $-\Delta (1-\delta \Delta)^{-1} = \delta^{-1} - \delta^{-1}(1-\delta \Delta)^{-1}$ to get
$$
v = (1-\delta^{-1}) (1-\delta \Delta)^{-1} w + \delta^{-1} w  -3 (1-\delta \Delta)^{-1}Q^2 w.
$$
By Lemma \ref{L:weight-transfer}, we have
$$
\la x \ra^{1/2} (1-\delta \Delta)^{-1} w  = \underbrace{\la x \ra^{1/2} (1-\delta \Delta)^{-1} \la x\ra^{-1/2}}_{L^2_{xy}\to L^2_{xy} \text{ bounded }} \underbrace{\la x\ra^{1/2} w}_{\in L^2_{xy}}
$$
so $\la x \ra^{1/2}(1-\delta \Delta)^{-1}w \in L^2_{xy}$ and similarly for the last term in the expression for $v$ above.    This completes the proof of \eqref{E:weight-transfer}.

Also, \eqref{E:H1transfer} is just a rephrasing of \eqref{E:H1transfer-2} in Lemma \ref{L:H1transfer}.
\end{proof}

\begin{lemma}
\label{L:weight-transfer}
For any $\alpha \in \mathbb{R}$, $0<\delta\leq 1$, the 2D operator $\la x \ra^\alpha (1-\delta \Delta)^{-1} \la x \ra^{-\alpha}$ is $L^2_{xy}\to L^2_{xy}$ bounded with operator norm independent of $\delta$. Note the weight is only in $x$(not both $x$ and $y$).  
\end{lemma}

\begin{proof}
Let $K$ be such that $\hat k(\xi,\eta) = (1+ |(\xi,\eta)|^2)^{-1}$.  We know that $k(x,y)$ is radial, it behaves like $\ln|(x,y)|$ as $(x,y)\to 0$, and $|(x,y)|^{-1/2} e^{-|(x,y)|}$ as $|(x,y)| \to \infty$.   To prove the lemma, we just note that it suffices to prove it for $\alpha>0$ by duality.  Thus, $\la x \ra^\alpha (I-\delta\Delta)^{-1} \la x \ra^{-\alpha}$   is an operator with kernel
$$
K((x,y),(x',y')) = \frac{1}{\delta} \, k\left(\frac{x-x'}{\delta^{1/2}}, \frac{y-y'}{\delta^{1/2}}\right) \frac{ \la x \ra^\alpha}{\la x' \ra^{\alpha}}.
$$
For $\alpha>0$, we use $\la x \ra^\alpha \lesssim \la x-x'\ra^{\alpha}+ \la x' \ra^\alpha \leq \la \frac{x-x'}{\delta^{1/2}} \ra^\alpha + \la x' \ra^\alpha$ to obtain
$$
|K( (x,y),(x',y'))| \leq \frac{1}{\delta}\,  \left|k \left(\frac{x-x'}{\delta^{1/2}}, \frac{y-y'}{\delta^{1/2}} \right)\right| \left\la \frac{x-x'}{\delta^{1/2}} \right\ra^\alpha + \frac{1}{\delta} \left|k\left(\frac{x-x'}{\delta^{1/2}}, \frac{y-y'}{\delta^{1/2}} \right) \right|.
$$
Thus,
$$
\| K( (x,y), (x',y')) \|_{L_{x'y'}^\infty L_{xy}^1} + \| K( (x,y), (x',y')) \|_{L_{xy}^\infty L_{x'y'}^1} \lesssim_\alpha 1.
$$
The result then follows from Schur's test.
\end{proof}

\begin{lemma}
\label{L:H1transfer}
There exists $\delta_0>0$ such that if $\la w, \nabla Q \ra =0$ and $0<\delta \leq \delta_0$, then
\begin{equation}
\label{E:H1transfer-2}
\la (1-\delta^{1/4}\Delta) w, w \ra \lesssim \la (1-\delta \Delta) v, v \ra,
\end{equation}
where  $v= (1-\delta \Delta)^{-1} \mathcal{L} w$, and the implicit constant is independent of $\delta$. 
\end{lemma}
\begin{proof}
Replacing $v$ by its definition and taking $f = (1-\delta^{1/4} \Delta)^{1/2}w$, this is
$$ 
\la f, f \ra \lesssim \la \mathcal{L}(1-\delta^{1/4} \Delta)^{-1/2} f, (1-\delta \Delta)^{-1} \mathcal{L} (1-\delta^{1/4} \Delta)^{-1/2} f \ra,
$$
which is operator inequality
$$
1 \lesssim (1-\delta^{1/4} \Delta)^{-1/2} \mathcal{L} (1-\delta \Delta)^{-1} \mathcal{L} (1-\delta^{1/4} \Delta)^{-1/2}
$$
on the subspace given by orthogonality condition.  The following identity can be checked by expanding out the commutators
\begin{align*}
\indentalign (1-\delta^{1/4} \Delta)^{-1/2} \mathcal{L} (1-\delta \Delta)^{-1} \mathcal{L} (1-\delta^{1/4} \Delta)^{-1/2} \\
&= (1-\delta^{1/4}\Delta)^{-1/2} (1-\delta \Delta)^{-1/2} \mathcal{L}^2 (1-\delta \Delta)^{-1/2} (1-\delta^{1/4}\Delta)^{-1/2} \\
&\qquad +(1-\delta^{1/4} \Delta)^{-1/2} [ \mathcal{L}, (1-\delta \Delta)^{-1/2}] (1-\delta \Delta)^{-1/2} \mathcal{L} (1-\delta^{1/4} \Delta)^{-1/2} \\
& \qquad - (1-\delta^{1/4} \Delta)^{-1/2} (1-\delta \Delta)^{-1/2} \mathcal{L} [ \mathcal{L}, (1-\delta \Delta)^{-1/2}] (1-\delta^{1/4}\Delta)^{-1/2}.
\end{align*}
Since 
$\ds (1-\delta^{1/4} \Delta)^{-1/2}\mathcal{L} (1-\delta^{1/4} \Delta)^{-1/2}$ 
is $L^2\to L^2$ bounded with operator norm $\leq \delta^{-1/4}$, by Lemma \ref{L:com1} below, the second and third lines are $L^2\to L^2$ bounded operators with norm $\leq \delta^{-1/4} \delta^{1/2} = \delta^{1/4}$.  Thus, for sufficiently small $\delta_0$, the positivity reduces to the question of positivity of the first line:
\begin{equation}
\label{E:com7c}
1 \lesssim (1-\delta^{1/4}\Delta)^{-1/2} (1-\delta \Delta)^{-1/2} \mathcal{L}^2 (1-\delta \Delta)^{-1/2} (1-\delta^{1/4}\Delta)^{-1/2}.
\end{equation}
Take 
$$
g = (1-\delta \Delta)^{-1/2} (1-\delta^{1/4}\Delta)^{-1/2} w
$$
so that \eqref{E:com7c} can be written as
\begin{equation}
\label{E:com7b}
\la w, w \ra \lesssim \la \mathcal{L}^2 g, g \ra.
\end{equation}
We claim that to prove \eqref{E:com7b}, it suffices to establish
\begin{equation}
\label{E:com7}
\la g, g \ra \lesssim \la \mathcal{L}^2 g, g \ra.
\end{equation}
To see that \eqref{E:com7} implies \eqref{E:com7b}, note that
\begin{equation}
\label{E:com10}
\| w\|_{L^2}^2 = \| (1-\delta \Delta)^{1/2} (1-\delta^{1/4} \Delta)^{1/2} g \|^2 \lesssim \| g \|_{L^2}^2 + \delta^{5/4} \| \Delta g \|_{L^2}^2.
\end{equation}
Standard elliptic regularity estimates give
\begin{equation}
\label{E:com9}
\| (1-\Delta) g \|_{L^2}^2 \lesssim \la \mathcal{L}^2 g, g \ra + \|g\|_{L^2}^2.
\end{equation}
Plugging \eqref{E:com9} into \eqref{E:com10} gives
\begin{equation}
\label{E:com11}
\| w\|_{L^2}^2 \lesssim \|g\|_{L^2}^2 + \delta^{5/4}\la \mathcal{L}^2 g, g\ra. 
\end{equation}
Thus, if we assume \eqref{E:com7}, then \eqref{E:com11} and \eqref{E:com7} yields \eqref{E:com7b}.  

Thus, it remains to prove \eqref{E:com7}.  Note that the spectrum of $\mathcal{L}^2$ (being the square of the spectrum of $\mathcal{L}$) consists of $0$ as an isolated eigenvalue with eigenspace spanned by $\nabla Q$, and the rest of the spectrum of $\mathcal{L}^2$ lies in $[\alpha,+\infty)$ for some $\alpha>0$.   Since $\la w, \nabla Q \ra=0$, \eqref{E:com7} would follow immediately if $g=w$.  We can, however, use Lemma \ref{L:com2} to show that 
\begin{equation}
\label{E:com8}
|\la g, \nabla Q \ra | \lesssim \delta^{1/4} \|w\|_{L^2}
\end{equation}
which suffices to establish \eqref{E:com7}  as follows.   (We explain how to obtain \eqref{E:com8} from Lemma \ref{L:com2} at the end of the proof.)

Let $P_0$ be the operator of orthogonal projection onto $\nabla Q$, and let $P_c = I-P_0$.  Explicitly,
$$
P_0 g = \frac{ \la g, \nabla Q\ra}{\|\nabla Q\|_{L^2}} \frac{\nabla Q}{\| \nabla Q\|_{L^2}}.
$$
By \eqref{E:com8} and \eqref{E:com11}, we get
\begin{equation}
\label{E:com12}
\|P_0 g\|_{L^2}^2 \lesssim \delta^{1/2} (\|g\|_{L^2}^2 + \la \mathcal{L}^2g,g \ra).
\end{equation}
Since the spectrum of $P_c\mathcal{L}^2$ starts at $\alpha>0$,
\begin{equation}
\label{E:com13}
\alpha \|P_c g \|_{L^2}^2 \leq \la \mathcal{L}^2 g, g \ra.
\end{equation}
We have 
$$
\|g\|_{L^2}^2 = \|P_0g\|_{L^2}^2 + \|P_c g\|_{L^2}^2.
$$
Plugging in \eqref{E:com12} and \eqref{E:com13},
$$
\|g\|_{L^2}^2 \leq  \frac{1}{\alpha} \la \mathcal{L}^2 g, g \ra + C \delta^{1/2}( \|g\|_{L^2}^2+ \la \mathcal{L}^2 g, g \ra)
$$
for some constant $C>0$.  Taking $\delta_0$ sufficiently small gives \eqref{E:com7}.

Finally, we explain the proof of \eqref{E:com8} from Lemma \ref{L:com2}.  Applying Lemma \ref{L:com2} to $f=(1-\delta^{1/4}\Delta)^{-1/2}\nabla Q$, we obtain
$$
|\la (1- \delta \Delta)^{-1/2} (1-\delta^{1/4} \Delta)^{-1/2} \nabla Q, w \ra - \la (1-\delta^{1/4}\Delta)^{-1/2}  \nabla Q, w \ra | \lesssim  \delta \| w\|_{L^2}.
$$
Replacing $\delta$ by $\delta^{1/4}$ in Lemma \ref{L:com2} and taking $f=\nabla Q$, yields
$$
| \la (1-\delta^{1/4} \Delta)^{-1/2} \nabla Q, w \ra - \la  \nabla Q, w \ra | \lesssim \delta^{1/4} \|w\|_{L^2}.
$$
Combing the above two estimates, we have
$$
| \la (1-\delta \Delta)^{-1/2} (1-\delta^{1/4}\Delta)^{-1/2} \nabla Q, w \ra - \la \nabla Q, w \ra | \lesssim \delta^{1/4} \|w\|_{L^2}.
$$
Since $\la \nabla Q, w \ra =0$, this reduces to \eqref{E:com8}.
\end{proof}

The following two commutator lemmas (Lemma \ref{L:com1} and \ref{L:com2}) were used in the proof of Lemma \ref{L:H1transfer} above.

\begin{lemma}[1st commutator lemma]
\label{L:com1}
The compositions
\begin{equation}
\label{E:com5}
(1-\delta^{1/4} \Delta)^{1/2} [ \mathcal{L}, (1-\delta \Delta)^{-1/2} ] (1-\delta^{1/4} \Delta)^{-1/2}
\end{equation}
and
\begin{equation}
\label{E:com6}
(1-\delta^{1/4} \Delta)^{-1/2} [ \mathcal{L}, (1-\delta \Delta)^{-1/2} ] (1-\delta^{1/4} \Delta)^{+1/2}
\end{equation}
are $L^2\to L^2$ bounded with operator norm $\lesssim \delta^{1/2}$.
\end{lemma}
\begin{proof}
In this proof, we use $x$ to represent the 2D coordinate $\mathbf{x}=(x,y)$.
Since \eqref{E:com6} is the adjoint of \eqref{E:com5}, it suffices to prove the claim for \eqref{E:com5}.  For this, we start by showing that 
\begin{equation}
\label{E:com3}
[ \mathcal{L}, (1-\delta \Delta)^{-1/2} ] \text{ is }L^2\to L^2 \text{ bounded with norm }\lesssim \delta^{1/2}.
\end{equation}
Let $\hat k(\xi) = (1+|\xi|^2)^{-1/2}$ (in two dimensions).   Then the Fourier transform of $\delta^{-1} k(\delta^{-1/2} x)$ is $(1+\delta |\xi|^2)^{-1/2}$.  Note that
$$
[\mathcal{L}, (1-\delta \Delta)^{-1/2} ] = -3 [ Q^2, (1-\delta \Delta)^{-1/2} ]. 
$$
Dropping the factor $-3$, the kernel is $K(x,x')$, where
$$
K(x,x') = \delta^{-1}k(\delta^{-1/2}(x-x'))(Q(x)^2 - Q(x')^2).
$$
Since $|Q(x)^2 - Q(x')^2 | \lesssim |x-x'|$, we have
$$
| K(x,x')| \lesssim \delta^{1/2} \cdot \delta^{-1}\tilde k( \delta^{-1/2}(x-x')),
$$
where $\tilde k(x) = |x| k(x)$.  Since $\tilde k \in L^1$, we obtain
$$
\| K \|_{L_{x'}^\infty L_x^1} \lesssim \delta^{1/2} \,, \qquad \| K \|_{L_x^\infty L_{x'}^1} \lesssim \delta^{1/2}.
$$
By Schur's test, we obtain \eqref{E:com3}.   To prove that \eqref{E:com5} is $L^2\to L^2$ bounded with operator norm $\lesssim \delta^{1/2}$, it suffices to prove the following two statements:
\begin{equation}
\label{E:com1}
[ \mathcal{L}, (1-\delta \Delta)^{-1/2} ] (1-\delta^{1/4} \Delta)^{-1/2}  \text{ is }L^2\to L^2 \text{ bounded with norm }\lesssim \delta^{1/2}
\end{equation}
and
\begin{equation}
\label{E:com2}
\delta^{1/8} \nabla[ \mathcal{L}, (1-\delta \Delta)^{-1/2} ] (1-\delta^{1/4} \Delta)^{-1/2}   \text{ is }L^2\to L^2 \text{ bounded with norm }\lesssim \delta^{1/2}.
\end{equation}
The claim \eqref{E:com1} follows immediately from \eqref{E:com3}, since $(1-\delta^{1/4}\Delta)^{-1/2}$ is $L^2\to L^2$ bounded with operator norm $\lesssim 1$.   For \eqref{E:com2}, we use the notation $K(x,x')$ and $k$ introduced in the proof of \eqref{E:com3}.  The operator in \eqref{E:com2} applied to a function $f$ takes the form
\begin{equation}
\label{E:oper1}
\delta^{1/8} \nabla_x \int_{x'} K(x,x') [(1-\delta^{1/4}\Delta)^{-1/2} f](x') \, dx'.
\end{equation}
Substituting $K(x,x') = \delta^{-1} k(\delta^{-1/2}(x-x')) (Q^2(x)-Q^2(x'))$ and distributing the $x$ derivative, we write \eqref{E:oper1} as
\begin{align*}
& = \delta^{1/8} \int_{x'}  \left[\nabla_x \left( \delta^{-1} k(\delta^{-1/2}(x-x')) \right) \right] (Q^2(x)-Q^2(x')) [(1-\delta^{1/4}\Delta)^{-1/2} f](x') \, dx' \\
& \qquad + \delta^{1/8} \int_{x'} \delta^{-1} k(\delta^{-1/2}(x-x')) (\nabla_xQ^2(x)) \, [(1-\delta^{1/4}\Delta)^{-1/2} f](x') \, dx'.
\end{align*}
In the first term, we can replace $\nabla_x$ by $-\nabla_{x'}$ and then integrate by parts to continue
\begin{align*}
& = \int_{x'} \delta^{-1} k(\delta^{-1/2}(x-x')) (Q^2(x)-Q^2(x')) \,  \delta^{1/8} \nabla_{x'} [(1-\delta^{1/4}\Delta)^{-1/2} f](x') \, dx' \\
& \qquad + \delta^{1/8} \int_{x'} \delta^{-1} k(\delta^{-1/2}(x-x')) (\nabla(Q^2)(x)- \nabla(Q^2)(x')) \, [(1-\delta^{1/4}\Delta)^{-1/2} f](x') \, dx'.
\end{align*}
By the same argument that established \eqref{E:com3}, the second line is an $L^2\to L^2$ bounded operator acting on $f$ with operator norm $\lesssim \delta^{1/8}\cdot \delta^{1/2}$.  The first line is equal to the following operator acting on $f$:
\begin{equation}
\label{E:com4}[ \mathcal{L}, (1-\delta \Delta)^{-1/2}] \delta^{1/8} \nabla (1-\delta^{1/4}\Delta)^{-1/2}. 
\end{equation}
Note that $\delta^{1/8}\nabla (1-\delta^{1/4}\Delta)^{-1/2}$ is $L^2\to L^2$ bounded with operator norm $\lesssim 1$.  This, combined with \eqref{E:com3}, gives that \eqref{E:com4} is $L^2\to L^2$ bounded with operator norm $\lesssim \delta^{1/2}$, completing the proof of \eqref{E:com5}.
\end{proof}

\begin{lemma}[2nd commutator lemma]
\label{L:com2}
For any functions $f$ and $w$, we have the estimate
$$
| \la (1-\delta \Delta)^{-1/2} f, w \ra - \la f, w \ra | \lesssim \delta \| f \|_{\dot H^2} \|w\|_{L^2}.
$$
\end{lemma}
\begin{proof}
We have
$$
| \la (1-\delta \Delta)^{-1/2} f, w \ra - \la f, w \ra | \lesssim \| [(1-\delta \Delta)^{-1/2}-1] f \|_{L^2} \|w\|_{L^2}.
$$
As a Fourier multiplier, the operator $(1-\delta \Delta)^{-1/2}-1$ takes the form
$$
(1+\delta |\xi|^2)^{-1/2} - 1 = \frac{ - \delta |\xi|^2}{ (1+\delta |\xi|^2)^{1/2}( 1 + (1+\delta |\xi|^2)^{1/2})},
$$
and hence,
$$
|(1+\delta |\xi|^2)^{-1/2} - 1 | \lesssim \delta |\xi|^2,
$$
from which the conclusion follows.
\end{proof}

\section{Appendix}
\label{S:numerical}
Here, we discuss the discrete spectrum of the operator $A$ that arizes in the virial estimate in Lemma \ref{S:virial-1}. We note that here we show a direct numerical computation of the spectrum of the operator $A$, and in \cite{HRY2025} we provide several other numerically-assisted methods to prove the spectral properties of $A$. For convenience of the numerical computations, we double the operator $A$ and search for the spectrum below 1 (instead of $\frac12$ as discussed after \eqref{E:P}): 
\begin{equation}
2A \equiv 2(B+P) =-3\partial_{xx}-\partial_{yy}+1-3Q^2-6xQQ_x+2P,
\end{equation}
where $P$ is defined in \eqref{E:P} as the following operator with inner products:
\begin{align}\label{E:2P}
2Pv=\frac{6Q^2Q_x}{\|Q\|_{L^2}^2}\langle v, xQ \rangle + \frac{xQ}{\|Q\|_{L^2}^2}\langle v, 6Q^2Q_x \rangle.
\end{align}

Similar to, for example, \cite{CGNT} or as we did for the 3D ZK in \cite{FHRY}, we numerically calculate the spectrum of the operator $2(B+P)$ in the following steps:
\begin{itemize}
\item[1.] 
discretize the operator into a form of a matrix,
\item[2.] 
find the eigenvalues and corresponding eigenvectors of the matrix; that will produce the spectrum of the operator $2(B+P)$,
\item[3.] 
the eigenvalues that are less than $1$ are the ones that are relevant (give the discrete spectrum below $1$), since $2(B+P)$ has continuous spectrum from $1$, see Lemma \ref{L:v-virial}.
\end{itemize}

It has been shown in \cite{CGNT} that the matlab commands \texttt{eig} or \texttt{eigs}, which incorporate ARPACK in \cite{ARPACK}, is an efficient way to compute the eigenvalues for large matrices. Therefore, it suffices to show how to discretize the operator $2(B+P)$ into a matrix form.
The discretization of the operator $B$ and imposing the homogeneous Dirichlet boundary conditions are standard, for that we follow the same procedure as in \cite[Chapters 6, 9, 12]{T2001}.

We next describe how we discretize the projection operator $P$. We also introduce a mapping to show how we choose the collocation points via Chebyshev method and make them more concentrated in the central region, where the functions (including the ground state $Q$ and eigenfunction) have the largest amplitude and gradient.
\smallskip

{\it Discretization of the projection term.}  
In this part, we introduce our discretization of the projection term $P$. We give a general formula for discretizing the operator $P$ in the form 
$$
Pu= \langle u,f  \rangle \, g,
$$
where $u, f, g \in L^2(\mathbb{R}^d)$. We discuss the 1D case, then it can be easily extended to $d\geq 2$ by standard numerical integration technique for multi-dimensions, e.g., see \cite[Chapter 6, 12]{T2001}.

We use the notation $f_i$ to be the discretized value of the function $f(x)$ at the point $x_i$ and write the vector $\vec{f}=(f_0,f_1,\cdots,f_N)^T$ to represent the values of the function $f$ (at a given time step), similar notation is used for other variables. The pointwise multiplication of the vectors or matrices with the same dimension is denoted by ``$.*$", i.e., $\vec{a}.*\vec{b}:=(a_0b_0,\cdots,a_Nb_N)^{\mathrm{T}}$; the notation ``$*$" is kept for the standard vector or matrix multiplication.

Let $w(x)$ to be a weight for a given quadrature. For example, for the composite trapezoid rule with step-size $h$ the weight is 
$\vec{w}=(w_0,w_1,\cdots ,w_N)^T=\tfrac{h}{2}(1,2,\cdots,2,1)^T,$
since the composite trapezoid rule can be written as
{\small
\begin{align*}
\int_a^b f(x) dx \approx \sum_{i=0}^{N} f_i w_i=\vec{f}^{~T} *\vec{w}.
\end{align*} }
For the Chebyshev Gauss-Lobatto quadrature, which is exactly what we use in this work, 
$$
\int_{-1}^1 f(x) dx \approx \sum_{i=0}^N w_i f(x_i)=\vec{f}^{~T} * \vec{w}
$$
with the corresponding weight $\vec{w}=(w_0,w_1,\cdots ,w_N)^T$, where
{\small
\begin{align*}
w_i=\frac{\pi}{N}\sqrt{1-x_i^2}, ~i=1,2,\cdots, N-1, ~
w_0=\frac{\pi}{2N}\sqrt{1-x_0^2}, ~~ \mbox{and} ~~ w_N=\frac{\pi}{2N}\sqrt{1-x_N^2}.
\end{align*}
}
Using this, we write
{\small
\begin{align*}
Pu=\langle u,f \rangle g= (\sum_{i=0}^N w_i \,f_i\, u_i) \, \vec{g} 
= \left[\begin{matrix}
g_0\\
g_1\\
\vdots\\
g_N
\end{matrix}\right]  (\sum_{i=0}^N w_i f_i u_i) = \left[\begin{matrix}
g_0\\
g_1\\
\vdots\\
g_N
\end{matrix}\right]  (\vec{w}^{{T}}.*\vec{f}^{{~T}}) * \vec{u} 
~:= \mathbf{P} \vec{u},
\end{align*}
with the matrix
\begin{align}\label{E: P term}
\mathbf{P}=\vec{g}*(\vec{w}^{\,\mathrm{T}}.*\vec{f}^{~\mathrm{T}})
\end{align}}
being the discretized approximation of the projection operator $P$. 
Observe that the matrix $\mathbf{P}$ is a dense matrix. This is the reason that in 2D computation we can only use a limited number of spectral collocation points (or Chebyshev collocation points). 

We next introduce a mapping, which changes the domain from $[-1,1]$ to $[-L,L]$, with more points concentrated around the center.
\smallskip

{\it Chebyshev collocation points.} 
We take the computational square region $[-L,L] \times [-L,L]$ to approximate the real space $\mathbb{R}^2$, and since $Q$ decays exponentially fast, this is a reasonable computational domain even if $L$ is not too large. 
We can not assign too many grid points in each dimension, however, we put more grid points in the center region of the computational domain $[-L,L] \times [-L,L]$, where our functions such as the ground state $Q$ have the largest amplitude and gradient. We, therefore, redistribute the mesh grid based on Chebyshev collocation points. 

Consider the mapping $T$ that concentrates more points near the origin:
\begin{align}
T: \, [-1,1] \rightarrow [-L,L], \quad T(\xi)=x,
\end{align}
here, $\xi$ represents the Chebyshev collocation points and $x$ represents the grid points in the computational interval, which are more concentrated at the center. 
One possible mapping is
\begin{align}\label{D: x mapping}
x(\xi)=L\frac{e^{a\xi}-e^{-a\xi}}{e^{a}-e^{-a}}
\end{align}
for some constant parameter $a$.

During our computation, we take $L=20$ and $a=4$ or $5$. Without this process, the stiffness and inaccuracy may occur, since we are restricted to using just a few points (for example, taking $N=72$ leads to an eigenvalue problem with a $72^2 \times 72^2$ matrix) in each dimension to have an acceptable computational cost. With this process, on the other hand, the number of grid points in each dimension could be taken to be equal to $N=32,48$ and $64$. These three different choices lead to almost the same results. 

After the mapping is applied, we also apply the chain rule for both derivatives, $\partial_x$ and $\partial_y$. Denote by the vectors $\vec{x}_{\xi}$, $\vec{x}_{\xi \xi}$, $\vec{y}_{\eta}$ and $\vec{y}_{\eta \eta}$ -- the vectors discretized from $x_{\xi}$, $x_{\xi \xi}$, $y_{\eta}$, $y_{\eta \eta}$. 
Denote by \texttt{diag}$(\vec{v})$ the diagonal matrix generated from the vector $\vec{v}$. The matrices $\mathbf{{D}^{(1)}_{\xi}}$ and $\mathbf{{D}^{(2)}_{\xi}}$ are the differential matrices of $\partial_{\xi}$ and $\partial_{\xi \xi}$ generated by the Chebyshev collocation differentiation. Similarly, the matrices $\mathbf{{D}^{(1)}_{\eta}}$ and $\mathbf{{D}^{(2)}_{\eta}}$ are those for $\partial_{\eta}$ and $\partial_{\eta \eta}$.

By the chain rule, we have 
$$
\mathbf{\tilde{D}_x^{(1)}}=\operatorname{diag}(\frac{1}{\vec{x}_{\xi}})\mathbf{D_{\xi}^{(1)}},
$$
and 
$$
\mathbf{\tilde{D}_x^{(2)}}=\operatorname{diag}(\frac{1}{\vec{x}_{\xi}^{~2}})\mathbf{D_{\xi}^{(2)}}+ \operatorname{diag}((\mathbf{D_{\xi}^{(1)}}*\frac{1}{\vec{x}_{\xi}}).*\frac{1}{\vec{x}_{\xi}})\mathbf{D_{\xi}^{(1)}}.
$$
Similarly, we can generate the differential matrices $\mathbf{\tilde{D}_y^{(1)}}$ and $\mathbf{\tilde{D}_y^{(2)}}$ in terms of $\eta$.

Finally, the operator $2(B+P)$ is discretized in the form (denoted by $\mathbf{M})$:
\begin{align}\label{E: M-matrix}
\mathbf{M}=-3\mathbf{\tilde{D}_x^{(2)}}-\mathbf{\tilde{D}_y^{(2)}}+ \operatorname{diag}(\vec{1}-3*\vec{Q^2}-6*\vec{x}.*\vec{Q}.*\vec{Q}_x) +\mathbf{P},
\end{align}
where $\mathbf{P}$ is the matrix for the inner products that can be discretized from the formula \eqref{E: P term}, and $\vec{1}=(1,1,\cdots,1)^T$ is the vector with the same size of variables as $\vec{Q}$.
\smallskip

{\it Numerical results.} 
After setting the operator $2(B+P)$ into a matrix form $\mathbf{M}$ as in \eqref{E: M-matrix}, we use the matlab command \texttt{eigs} to find the eigenvalues below $1$. 
We obtain two negative eigenvalues of $2(B+P)$ below $1$:
\begin{align}\label{D: numerical eigenvalues}
\lambda_o =  -1.0735, \quad \lambda_e = -0.2151.
\end{align}
(Hence, for $B+P$, the eigenvalues in \eqref{E:lambdas} 
are halves of the above.)  
\begin{figure}[ht]
\includegraphics[width=0.47\textwidth,height=.38\textwidth]{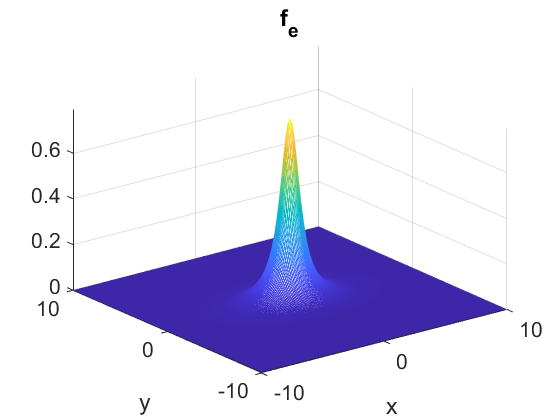}
\includegraphics[width=0.47\textwidth,height=.38\textwidth]{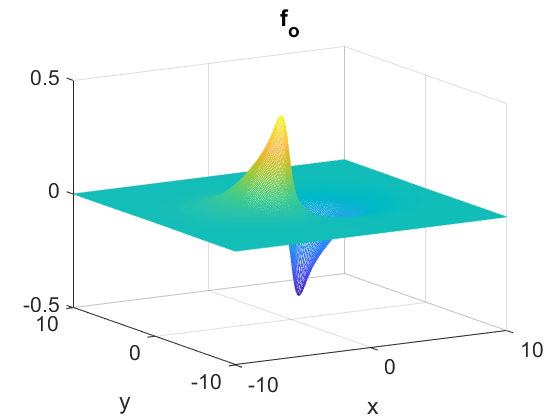}
\caption{{\footnotesize The eigenfunctions: $f_e$, even in both $x$ and $y$ (left), and $f_o$, odd in $x$ (right).}}
\label{F:f} 
\end{figure}
\begin{figure}[ht]
~~\includegraphics[width=0.45\textwidth,height=.32\textwidth]{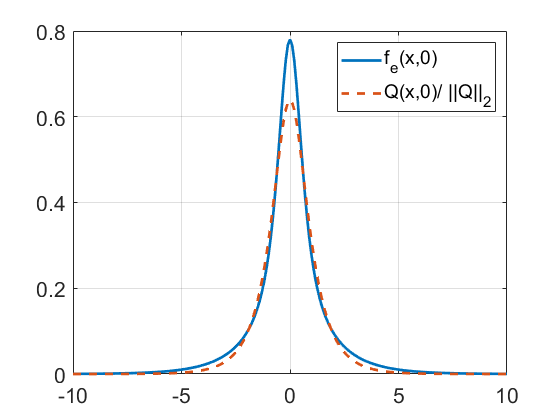}
\includegraphics[width=0.45\textwidth,height=.32\textwidth]{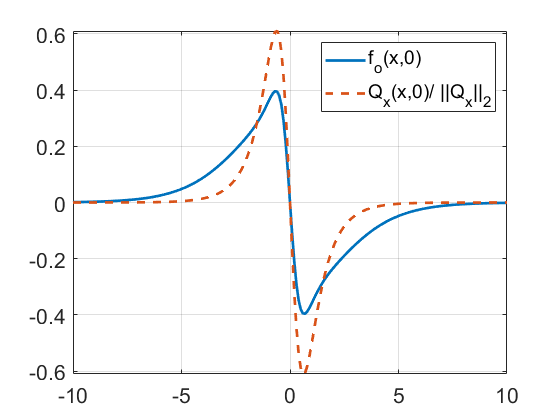}
\caption{{\footnotesize Cross-sections (in $y=0$ plane) of the eigenfunctions $f_e$ and $f_o$ in comparison with  $\frac{Q}{\|Q\|_2}$ and $\frac{Q_x}{\|Q_x\|_2}$ from the orthogonal conditions \eqref{E:ortho-v}.}}
\label{F:f2}
\end{figure}
We also obtain two eigenfunctions corresponding to the above eigenvalues, $f_o$ and $f_e$, illustrated in Figure \ref{F:f}.

After that we obtain the normalized inner products, stated in \eqref{D:geometry}. We show the cross-section (by $y=0$ plane) of profiles of $f_e$ and $f_o$ compared to the functions $\frac{Q}{\|Q\|_{L^2}}$ and $\frac{Q_x}{\|Q_x\|_{L^2}}$ in Figure \ref{F:f2}.

We note that the values of inner products and eigenvalues are consistent when using $N=32, \, 48, \, 64, \, \mbox{or} \,\, 72 $ collocation points up to the first 4 digits.


\newpage


{\bf Data Availability.} 
Data sharing not applicable to this article as no datasets were generated or analyzed during the current study.

\bigskip

{\bf Competing Interests.} 
The authors have no competing interests to declare that are relevant to the content of this article.

\bigskip

{\bf ORCID}

{\it Luiz Gustavo Farah} 
\qquad {\ https://orcid.org/0000-0003-1034-3480}

{\it Justin Holmer}
\qquad {https://orcid.org/0000-0003-3644-5549}

{\it Svetlana Roudenko}
\qquad {https://orcid.org/0000-0002-7407-7639}

{\it Kai Yang}
\qquad {http://orcid.org/0000-0002-2289-9403}


\bibliographystyle{mrl}

\end{document}